\documentclass{amsart}
\usepackage{verbatim}
\usepackage{amssymb}
\usepackage{amsmath}
\usepackage[usenames]{color}
\usepackage{graphicx}
\usepackage{amscd}

\theoremstyle{plain}
\newtheorem{theorem}{Theorem}
\newtheorem{prop}[theorem]{Proposition}

\newtheorem{proposition}[theorem]{Proposition}
\newtheorem{lemma}[theorem]{Lemma}
\newtheorem{corollary}[theorem]{Corollary}

\newtheorem{definition}[theorem]{Definition}

\theoremstyle{definition}

\newtheorem{remark}[theorem]{Remark}

\newtheorem{example}[theorem]{Example}
\numberwithin{equation}{section}
\numberwithin{theorem}{section}
\allowdisplaybreaks

\usepackage{hyperref}

\renewcommand{\P}{{\mathbb P}}

\newcommand{\eps}{\varepsilon}

\newcommand{\wt}{\widetilde}

\newcommand{\calL}{{\mathcal L}}

\newcommand{\one}{{{\bf 1}}}

\newcommand{\lb}{\langle}
\newcommand{\rb}{\rangle}

\newcommand{\limn}{\lim_{n\to\infty}}

\newcommand{\vv}{a}

\newcommand{\ud}[0]{\,\mathrm{d}}

\begin{document}

\title[Cylindrical continuous martingales]
{Cylindrical continuous martingales and stochastic integration in infinite dimensions}

\author{Mark Veraar}
\address{Delft Institute of Applied Mathematics\\
Delft University of Technology \\ P.O. Box 5031\\ 2600 GA Delft\\The
Netherlands}
\email{M.C.Veraar@tudelft.nl}

\author{Ivan Yaroslavtsev}
\email{I.S.Yaroslavtsev@tudelft.nl}

\begin{abstract}
In this paper we define a new type of quadratic variation for cylindrical continuous local martingales on an infinite dimensional spaces.
It is shown that a large class of cylindrical continuous local martingales has such a quadratic variation.
For this new class of cylindrical continuous local martingales we develop a stochastic integration theory for operator valued processes under the condition that the range space is a UMD Banach space. We obtain two-sided estimates for the stochastic integral in terms of the $\gamma$-norm. In the scalar or Hilbert case this reduces to the Burkholder-Davis-Gundy inequalities. An application to a class of stochastic evolution equations is given at the end of the paper.
\end{abstract}

\thanks{The first named author is supported by the Vidi subsidy 639.032.427 of the Netherlands Organisation for Scientific Research (NWO)}



\maketitle

\tableofcontents

\section{Introduction}

Cylindrical local martingales play an important role in the theory of stochastic PDEs. For example the classical cylindrical Brownian motion $W_H$ on a Hilbert space $H = L^2(D)$ can be used to give a functional analytic framework to model a space-time white noise on $\mathbb R_+\times D$.
A cylindrical (local) martingale $M$ on a Banach space $X$ is such that for every $x^*\in X^*$ (the dual space of $X$) one has that $M x^*$ is a (local) martingale and the mapping $x^*\to M x^*$ is linear and satisfies appropriate continuity conditions (see Section \ref{sec:defcyllocalmart}).

Cylindrical (local) martingales have extensively studied in the literature (see \cite{JKFR,Ond,Ond1,MP,MR1,Sch96,SuchWe}).
In this paper we introduce a new type of  \emph{quadratic variation} $[[M]]$ for a cylindrical continuous local martingale $M$ on a Banach space $X$ (see Definition \ref{def:cylmart}). Moreover, we develop a stochastic calculus for those $M$ which admit such a quadratic variation. The process $[[M]]$ is continuous and increasing and it is given by
\begin{equation}\label{eq:intro[[M]]}
[[M]]_t := \lim_{{\rm mesh}\to 0}\sum_{n=1}^N \sup_{x^*\in X^*, \|x^*\|= 1}([M(t_n)x^*]-[M(t_{n-1})x^*]),
\end{equation}
where the a.s.\ limit is taken over partitions $0= t_0 < \ldots < t_N = t$. The definition \eqref{eq:intro[[M]]} can be given for any Banach space $X$, but for technical reasons we will assume that $X^*$ is separable. The definition \eqref{eq:intro[[M]]} is new even in the Hilbert space setting. Our motivation for introducing this class comes from stochastic integration theory and in this case $M$ is a cylindrical continuous local martingale on a Hilbert space. A more detailed discussion on stochastic integration theory will be given in the second half of the introduction.

In many cases $[[M]]$ is simple to calculate. For instance for a cylindrical Brownian motion one has $[[W_H]]_t = t$. More generally, if $M = \int_0^\cdot \Phi \ud W_H$ where $\Phi$ is an $\calL(H,X)$-valued adapted process, then one has
\[[[M]]_t = \int_0^t \|\Phi(s) \Phi(s)^*\|_{\calL(X^*,X)} \ud s.\]
These examples illustrate that Definition \eqref{eq:intro[[M]]} is a natural object. However, one has to be careful, as there are  \emph{cylindrical} continuous martingales (even on Hilbert spaces) which do not have a quadratic variation $[[M]]$. From now on let us write $M_{\rm var}^{\rm loc}(X)$ for the class of cylindrical continuous local martingales which admit a quadratic variation.

If $M$ is a continuous local martingale with values in a Hilbert space, then it is well known that it has a classical quadratic variation $[M]$ in the sense that there exists an a.s.\ unique increasing continuous process $[M]$ starting at zero such that $\|M\|^{{2}} - [M]$ is a continuous local martingale again. It is simple to check that in this case $[[M]]$ exists and a.s.\ for all $t\geq 0$, $[[M]]_t \leq [M]_t$. Clearly, $[M]$ does not exist in the cylindrical case, but as we will see, $[[M]]$ gives a good alternative for it.

Previous attempts to define quadratic variation are usually given in the case $M$ is actually a martingale (instead of a cylindrical martingale) and in the case $X$ is a Hilbert space (see \cite{DGFR, Ond, MP, MR1}).
We will give a detailed comparison with the earlier attempts to define the quadratic variation in Section \ref{sec:cyl}.

\medskip

To study SPDEs with a space-time noise one often models the noise as a cylindrical local martingale on an infinite dimensional space. We refer the reader to \cite{DPZ} for the case of cylindrical Brownian motion.
In order to  study SPDEs one uses a theory of stochastic integration for operator-valued processes $\Phi:\mathbb R_+\times\Omega\to \calL(H,X)$.
Our aim is to develop a stochastic integration theory where the integrator $M$ is from $M_{\rm var}^{\rm loc}(H)$ and the integrand takes values in $\calL(H,X)$, where $X$ is a Banach space which has the UMD property.

The history of stochastic integration in Banach spaces has an interesting history which goes back 40 years. Important contributions have been given in
several papers and monographs (see \cite{BD, Brz1, BN, HP76, NW1, Nh, Ondr04, RS} and references therein). We refer to \cite{NVW15} for a detailed discussion on the history of the subject.
Influenced by results from Garling \cite{Ga1} and McConell \cite{MC}, a stochastic integration theory for $\Phi:[0,T]\times\Omega\to \calL(H,X)$ with integrator $W_H$ was developed in \cite{NVW} by van Neerven and Weis and the first author. The theory is naturally restricted to the class of Banach spaces $X$ with the UMD property (e.g.\ $L^q$ with $q\in (1, \infty)$). The main result is that $\Phi$ is stochastically integrable with respect to an $H$-cylindrical Brownian motion if and only if $\Phi\in \gamma(0,T;H,X)$ a.s. Here $\gamma(0,T;H,X)$ is a generalized space of square functions as introduced in the influential paper \cite{KW} (see Subsection \ref{subsec:gamma} for the definition). Furthermore, it was shown that for any $p\in (0,\infty)$ the following two-sided estimate holds
\[c \|\Phi\|_{L^p(\Omega;\gamma(0,T;H,X))} \leq \Big\|\sup_{t\in [0,T]}\Big\|\int_0^t \Phi \ud W_H\Big\|_X \Big\|_{L^p(\Omega)} \leq C \|\Phi\|_{L^p(\Omega;\gamma(0,T;H,X))},\]
which can be seen as an analogue of the classical Burkholder-Davis-Gundy inequalities. This estimate is strong enough to obtain sharp regularity results for stochastic PDEs (see \cite{NVWSMR}) which can be used for instance to extend some of the sharp estimates of Krylov \cite{Kry} to an $L^p(L^q)$-setting.

The aim of our paper is to build a calculus for the newly introduced class of cylindrical continuous local martingales which admit a quadratic variation. Moreover, if $M$ is a cylindrical continuous local martingale on a Hilbert space $H$, we show that there is a natural analogue of the stochastic integration theory of \cite{NVW} where the integrator $W_H$ is replaced by $M$. At first sight it is quite surprising that the $\gamma$-norms again play a fundamental role in this theory although the cylindrical martingales do not have a Gaussian distribution. Our theory is even new in the Hilbert space setting. The proof of the main result Theorem \ref{main} is based on a sophisticated combination of time change arguments and Brownian representation results for martingales with absolutely continuous quadratic variations from \cite[Theorem 2]{Ond}. Theorem \ref{main} gives that $\Phi$ is stochastically integrable with respect to $M$ if and only if $\Phi Q_M^{1/2}\in \gamma(0,T,[[M]];H,X)$ a.s.\ Here $Q_M$ is a predictable operator with norm $\|Q_M\| = 1$. Moreover, two-sided Burkholder--Davis--Gundy inequalities hold again. We will derive several consequence of the integration theory among which a version It\^o's formula.

\medskip

We finish this introduction with some related contributions to the theory of stochastic integration in an infinite dimensional setting. In \cite{MP} M{\'e}tivier and Pellaumail developed an integration theory for cylindrical martingales which are not necessarily continuous and two-sided estimates are derived in a Hilbert space setting. A theory for SDEs and SPDEs with semimartingales in Hilbert spaces is developed by Gy{\"o}ngy and Krylov in  \cite{GK1,GK2,G3}. The integration theory with respect to cylindrical L\'evy processes in Hilbert cases and its application to SPDEs is developed in the monograph by Peszat and Zabczyk \cite{PZ}. Some extensions in the Banach space setting have been considered and we refer to \cite{App07,ApRi,ManRud,Ros,RieGaa} and references therein. In \cite{Dirk14} Dirksen has found an analogue of the two-sided Burkholder--Davis--Gundy inequalities in the case the integrator is a Poisson process and $X = L^q$ (also see \cite{DMN,Marinelli13,MarRo}). By the results of our paper and the previously mentioned results, it is a natural question what structure of a cylindrical  \emph{noncontinuous} local martingales is required to build a theory which allows to have two-sided estimates for stochastic integrals.

\medskip

Outline:
\begin{itemize}
\item In Section \ref{sec:prel} some preliminaries are discussed.
\item In Section \ref{sec:cyl} the quadratic variation of a cylindrical continuous local martingale is introduced.
\item In Section \ref{sec:char} the stochastic integrable $\Phi$ are characterized.
\item In Section \ref{sec:SE} the results are applied to study a class of stochastic evolution equations.
\end{itemize}

 \emph{Acknowledgment} – The authors would like to thank Jan van Neerven for helpful comments. We are indebted to the anonymous referee for his/her valuable comments to improve the paper.

\section{Preliminaries\label{sec:prel}}

Let $F:\mathbb R_+ \to \mathbb R$ be a right-continuous function of bounded variation (e.g. nondecreasing c\'adl\'ag). Then we define $\mu_F$ on subintervals of  $\mathbb R_+$ as follows:
\begin{align*}
 \mu_F({(a,b]}) &= F(b)-F(a),\;\;\; 0\leq a< b<\infty,\\
 \mu_F(\{0\}) &=0.
\end{align*}
By the Carath\'eodory extension theorem, $\mu_F$ extends to a measure, which we will call  \emph{the Lebesgue-Stieltjes measure} associated to $F$. Conversely, if $\mu$ is a measure such that $\mu({(a,b]}) = F(b)-F(a)$ for a given function $F$, then $F$ has to be right-continuous.

Let $(S,\Sigma)$ be separable measurable space and let $(\Omega, \mathbb F, \mathbb P)$ be a probability space.
A mapping $\nu:\Sigma\times \Omega\to [0,\infty]$ will be called a  \emph{random measure} if for all $A\in \Sigma$, $\omega\mapsto \nu(A,\omega)$ is measurable and for almost all $\omega\in \Omega$, $\nu(\cdot, \omega)$ is a measure on $(S,\Sigma)$ and $(S, \Sigma, \nu(\cdot, \omega))$ is separable (i.e.\ such that the corresponding $L^2$-space is separable).

\begin{example}\label{ex:randommeasL}
Let $F:\mathbb R_+\times \Omega \to \mathbb R$ be a c\'adl\'ag process which is of bounded variation a.s.\ Then one can define a random measure $\mu_F:\mathcal B(\mathbb R_+)\times \Omega \to [0,\infty]$ such that $\mu_F(A,\omega) = \mu_{F(\omega)}(A)$.
\end{example}

Random measures arise naturally when working with continuous local martingales $M$. Indeed, for almost all $\omega \in \Omega$, the quadratic variation process $[M](\cdot, \omega)$ is continuous and increasing (see \cite{Kal,MP,Prot}), so as in Example \ref{ex:randommeasL} we can associate a Lebesgue-Stieltjes measure with it. Often we will denote this measure again by $[M](\cdot, \omega)$ for convenience.

Let $(S,\Sigma,\mu)$ be a measure space. Let $X$ and $Y$ be Banach spaces. An operator valued function $f: S \to \mathcal L(X,Y)$ is called  \emph{$X$-strongly measurable} if for all $x\in X$, the function $s\mapsto f(s)x$ is strongly measurable. It is called  \emph{scalarly measurable} if for all $y^*\in Y^*$, $f(s)^*y^*$ is strongly measurable. If $Y$ is separable both measurability notions coincide.

Often we will use the notation $A \lesssim_Q B$ to indicate that there exists a constant $C$ which depends on the parameter(s) $Q$ such that $A\leq C B$.

\subsection{Positive operators and self-adjoint operators on Banach spaces}
Let $X$, $X$ be Banach spaces. We will denote the space of all bilinear operators from
$X\times Y$ to $\mathbb R$ as $\mathbf B(X,Y)$.
Notice, that for each continuous $b \in \mathbf B(X,Y)$ there exists an operator
$B \in \mathcal L(X,Y^*)$ such that
\begin{equation}\label{bil}
 b(x,y) = \langle Bx,y\rangle,\;\;\;x\in X, y\in Y.
\end{equation}

We will call an operator $B\colon X \to X^*$  \textit{self-adjoint},
if for each $x, y \in X$
\[
\langle Bx,y\rangle = \langle Bx,y\rangle.
\]
A self-adjoint operator $B$ is called  \textit{positive}, if $\langle Bx,x\rangle\geq 0$
for all $x \in X$.

 \begin{remark}\label{posoper}
Notice, that if $B:X\to X^*$ is a positive self-adjoint operator, then the Cauchy-Schwartz inequality holds for the bilinear form $(x, y) :=\lb Bx, y\rb$ (see \cite[4.2]{Rud}). From the latter one deduces that
\begin{equation}\label{eq:normB}
\|B\| = \sup_{x\in X, \|x\|=1}|\langle Bx,x\rangle|
\end{equation}
Moreover, if $X$ is a Hilbert space, then \eqref{eq:normB} holds for any self-adjoint operator.
\end{remark}

Further we will need the following lemma proved in \cite[Proposition~32]{Ond}:
\begin{lemma}\label{funcselfadop}
 Let $(S,\Sigma)$ be a measurable space, $H$ be a separable Hilbert space, $f: S \to \mathcal L(H)$
 be a scalarly measurable self-adjoint operator-valued function. Let $F : \mathbb R \to \mathbb R$
 be locally bounded measurable. Then $F(f): S \to \mathcal L(H)$ is a scalarly measurable self-adjoint operator-valued function.
\end{lemma}

The next lemma allows us to define a square root of a positive operator
in case of a~reflexive separable Banach space:
\begin{lemma}\label{sqroot}
Let $X$ be a reflexive separable Banach space,
$B\colon X \to X^*$ be a positive operator. Then
there exists a separable Hilbert space $H$ and an operator
$B^{1/2} \colon X \to H$ such that $B = B^{1/2*} B^{1/2}$.
\end{lemma}

\begin{proof}
 Since $X$ is reflexive separable, $X^*$ is also separable.
 We will use the space $H$, constructed in \cite[p.154]{Kuo} (see also \cite[p.15]{BN} and \cite[Part~3.3]{Ond1}).
 Briefly speaking,
 one can find such a separable Hilbert space $H$ that there exists
 a continuous dense embedding $j \colon X^* \hookrightarrow H$. Because
 of the reflexivity, $j^*: H \hookrightarrow X^{**} = X$ is a continuous
 dense embedding and as an embedding it has a trivial kernel.

 Consider the operator $jBj^*:H \to H$. Obviously this operator is positive,
 so one can define a positive square root $\sqrt{jBj^*}:H \to H$  (see \cite[Chapter 6.6]{Fri}).
 Now define
 \[
  B^{1/2} = \sqrt{jBj^*} j^{*-1} \colon \text{ran }j^* \to H.
 \]
 This operator is bounded, because for each $x \in  \text{ran }j^*$
 \begin{align*}
   \|\sqrt{jBj^*} j^{*-1} x\|_H^2 &= \langle \sqrt{jBj^*} j^{*-1} x,\sqrt{jBj^*} j^{*-1} x\rangle
 =\langle jBj^* j^{*-1} x, j^{*-1} x\rangle\\
 &= \langle B x, j^*j^{*-1} x\rangle = \langle B x,  x\rangle\leq \|B\|\|x\|^2,
 \end{align*}
 therefore it can be extended to the whole $X$. Moreover, for all $x,y \in \text{ran }j^*$
 \[
 \langle B^{1/2*} B^{1/2} x,y\rangle = \langle   B^{1/2} x,B^{1/2}y\rangle
 =\langle   \sqrt{jBj^*} j^{*-1} x,\sqrt{jBj^*} j^{*-1}y\rangle = \langle B  x,y\rangle,
 \]
Thus $B^{1/2*} B^{1/2} = B$ on $\text{ran }j^*$, and hence on $X$ by density and continuity.
\end{proof}

\begin{remark}\label{sqrootmes}
 The square root obtained in such a way is not determined uniquely, since the operator
 $j$ can be defined in different ways. The following measurability property holds:
 if there exists a measurable space $(S,\Sigma)$ and a scalarly measurable
 function $f : S \to \mathcal L(X,X^*)$ with values in positive operators, then defined
 in such a polysemantic way $f^{1/2}$ will be also scalarly measurable. Indeed,
 since $f$ is scalarly measurable, then $jfj^*$ and, consequently by Lemma \ref{funcselfadop}
 and the fact that $jfj^*$ is positive operator-valued, the square root $\sqrt{jfj^*}$
 is scalarly measurable. So, $f^{1/2} = \sqrt{jfj^*}j^{*-1}x$ is measurable
 for all $x \in \text{ran }j^*$, and because of the boundedness of an operator $\sqrt{jfj^*}j^{*-1}$
 and the density of $\text{ran }j^*$ in $X$ one has that $f^{1/2}x$ is measurable for all $x\in X$.
 \end{remark}

\subsection{Supremum of measures}

In the main text we will often need explicit descriptions of the supremum of measures. The results are elementary, but we could not find suitable references in the literature. All positive measures are assumed to take values in $[0,\infty]$ (see \cite[Definition~1.6.1]{Bog}). In other words, a positive measure of a set could be infinite.

\begin{lemma}\label{maxmes}
Let $(\mu_{\alpha})_{\alpha\in \Lambda}$
be positive measures on a measurable space $(S,\Sigma)$.
Then there exists the smallest measure $\check{\mu}$ s.t.\ $\check{\mu} \geq \mu_{\alpha}$ $\forall \alpha\in \Lambda$. Moreover,
\begin{equation}\label{eq:supofmeas}
 \check{\mu}(A)= \sup\sum_{n=1}^N \sup_{\alpha} \mu_{\alpha} (A_n),\;\;\; A \in \Sigma,
\end{equation}
where the first supremum is taken over all the partitions $A=\bigcup_{n=1}^N A_n$ of $A$.
\end{lemma}

From now on we will write $\sup_{\alpha\in \Lambda} \mu_{\alpha} = \check{\mu}$, where $\check{\mu}$ is as in the above lemma. A similarly formula as \eqref{eq:supofmeas} can be found in \cite[Exercise~213Y(d)]{Frem2} for finitely many measures.

\begin{proof}
The existence of the measure $\check{\mu}$ is well-known (see e.g. \cite[Exercise~2.2]{Kal}  \cite[Exercise~213Y(e)]{Frem2}), but it also follows from the proof below. To prove \eqref{eq:supofmeas} let $\nu$ denote the right-hand side of \eqref{eq:supofmeas}.

We first show that $\nu$ is a measure. It suffices to show that $\nu$ is additive and $\sigma$-subadditive.
To prove the $\sigma$-subadditivity, let $(B_k)_{k\geq 1}$ be sets in $\Sigma$ and let $B = \bigcup_{k\geq 1} B_k$. Let
$A_1, \ldots, A_N\in \Sigma$ be disjoint and such that $B = \bigcup_{n=1}^N A_n$. Let $B_{nk} = A_n\cap B_k$. Then by the $\sigma$-subadditive of the $\mu_{\alpha}$, we find
\begin{align*}
\sum_{n=1}^N \sup_{\alpha} \mu_{\alpha} (A_n) & =\sum_{n=1}^N \sup_{\alpha} \sum_{k\geq 1} \mu_{\alpha} (B_{nk})
\leq \sum_{k\geq 1} \sum_{n=1}^N  \sup_{\alpha}  \mu_{\alpha} (B_{nk}) \leq \sum_{k\geq 1}  \nu(B_k).
\end{align*}
Taking the supremum over all $A_n$, we find $\nu(B) \leq \sum_{k\geq1} \nu(B_k)$.

To prove the additivity let $B,C\in \Sigma$ be disjoint. By the previous step it remains to show that  $\nu(B) + \nu(C)\leq \nu(B\cup C)$. Fix $\varepsilon>0$ and choose $A_1, \ldots, A_N\in \Sigma$ disjoint, $\alpha_1, \ldots, \alpha_N\in \Lambda$ and $1\leq M<N$ such that $\bigcup_{n=1}^M A_n = B$, $\bigcup_{n=M+1}^N A_n = C$ and
\[\nu(B)\leq \sum_{n=1}^M \mu_{\alpha_n}(A_n) +\varepsilon \ \ \ \text{and} \ \ \nu(C)\leq \sum_{n=M+1}^N \mu_{\alpha_n}(A_n)+\varepsilon .\]
Then we find that
\[\nu(B) + \nu(C) \leq \sum_{n=1}^N \mu_{\alpha_n}(A_n) +2\varepsilon\leq \nu(B\cup C) + 2\varepsilon,\]
and the additivity follows.

Finally, we check that $\nu = \check{\mu}$. In order to check this let $\tilde{\nu}$ be a measure such that $\mu_{\alpha}\leq\tilde{\nu}$ for all $\alpha$. Then for each $A\in \Sigma$ we find
\[\nu(A)  = \sup\sum_{n=1}^N \sup_{\alpha} \mu_{\alpha} (A_n) \leq \sup\sum_{n=1}^N \tilde{\nu} (A_n) = \tilde{\nu}(A)\]
and hence $\nu\leq \tilde{\nu}$. Thus we may conclude that $\nu = \check{\mu}$.
\end{proof}

\begin{remark}\label{rem:mu}
If the conditions of Lemma \ref{maxmes} are satisfied and there exists a measure $\mu$ such that $\mu_{\alpha}\leq \mu$, then $\check{\mu}\leq \mu$. In particular if $\mu$ is finite, then $\check{\mu}$ is finite as well.
\end{remark}

\begin{lemma}\label{lem:maxmeasuref}
Let $(S, \Sigma,\nu)$ be a measure space. Let $F$ be a set of measurable functions from $S$ into $[0,\infty]$. Let $\{f_j\}_{j=1}^\infty$ be a sequence in $F$. Let $\overline{f} = \sup_{j\geq 1} f_j$ and assume $\sup_{f\in F}f = \overline{f}$.
For each $f\in F$ let $\mu_f$ be the measure given by
\[\mu_f(A) = \int_A f \ud \nu.\]
Let $\check{\mu} = \sup_{f\in F} \mu_f$. Then $\check{\mu} = \sup_{j\geq 1} \mu_{f_j}$ and
\begin{equation}\label{eq:identityBorel}
\check{\mu}(A) = \int_A \overline{f} \ud \nu, \ \ A\in \Sigma.
\end{equation}
\end{lemma}
\begin{proof}
Since $\overline{f}$ is the supremum of countably many measurable functions, it is measurable.
Since $A\mapsto \int_A \overline{f} \ud \nu$ defines a measure which dominates all measures $\mu_{f}$, the estimate $``\leq"$ in \eqref{eq:identityBorel} follows.

To prove the converse estimate, let $A\in \Sigma$, $\varepsilon>0$ and $n\in \mathbb N$.
Let $A_1 = \{s\in A: f_1(s)>(1-\varepsilon){(\overline{f}(s)\wedge n)}\}$ and let
\[A_{j+1} = \{s\in A: f_{j+1}(s)>(1-\varepsilon){(\overline{f}(s)\wedge n)}\}\setminus \bigcup_{k=1}^j A_k, \ \ j\geq 1.\]
Then the $(A_j)_{j\geq 1}$ are pairwise disjoint and $\bigcup_{j\geq 1} A_j = A$, and therefore,
\begin{align*}
\check{\mu}(A) & = \sum_{j\geq 1} \check{\mu}(A_j) \geq \sum_{j\geq 1} \mu_{f_j}(A_j) = \sum_{j\geq 1} \int_{A_j} f_j \ud \nu \\ & \geq (1-\varepsilon) \sum_{j\geq 1} \int_{A_j} {\overline{f}(s)\wedge n} \ud \nu = (1-\varepsilon) \int_{A} {\overline{f}(s)\wedge n} \ud \nu.
\end{align*}
Since $\varepsilon>0$ and $n\in \mathbb N$ were arbitrary the required estimate follows. The identity $\check{\mu} = \sup_{j\geq 1} \mu_{f_j}$ follows if we replace $F$ by $\{f_j: j\geq 1\}$ and apply \eqref{eq:identityBorel} in this situation.
\end{proof}

\begin{lemma}\label{lemma:limmaxmes}
Let $(\mu_{n})_{n\geq 1}$ be a sequence of measures on a measurable space $(S,\Sigma)$.
Let $\check{\mu} = \sup_{n\geq 1} \mu_n$. Then for each $A\in \Sigma$,
\[
 (\sup_{1\leq n\leq N}\mu_n)(A) \to \check{\mu}(A),\;\;\;\text{as} \ \ N\to \infty.
\]
\end{lemma}
\begin{proof}
Let $A\in \Sigma$. Without loss of generality suppose that $\check{\mu}(A)<\infty$. Fix $\eps>0$. According to \eqref{eq:supofmeas} there exists $K>0$, a partition $A=\bigcup_{k=1}^K A_k$ of $A$ into pairwise disjoint sets and an increasing sequence $(n_k)_{1\leq k \leq K}\subseteq \mathbb N$ such that $\sum_{k=1}^K \mu_{n_k} (A_k)>\check{\mu}(A)-\eps$. Hence $(\sup_{1\leq n\leq n_K}\mu_n)(A)\geq \check{\mu}(A)-\eps$, which finishes the proof.
\end{proof}

\begin{remark}\label{maxbormes}
 Assume that in the situation above $S = \mathbb R$, $\Sigma$ is a Borel $\sigma$-algebra.
 Define $\bar{\mu}$ on segments as follows:
 \begin{equation}\label{maxbormesform}
   \bar{\mu}(a,b] = \sup\sum_{n=1}^N \sup_{\alpha} \mu_{\alpha} (A_n),
 \end{equation}
 where the first supremum is taken over all the partitions $(a,b]=\bigcup_{n=1}^N A_n$ of the segments $(a,b]$
 into pairwise disjoint segments. Then by Carath\'eodory's extension theorem $\bar{\mu}$ extends to a measure on the Borel $\sigma$-algebra. Obviously $\bar{\mu}\geq \mu_{\alpha}$ for each $\alpha\in \Lambda$ (because
 $(\bar{\mu}-\mu_{\alpha})((a,b])\geq 0$ for every segment $(a,b]$, and so, by \cite[Corollaries~1.5.8 and 1.5.9]{Bog}  for every Borel set)
 and $\bar{\mu}\leq \check{\mu}$. Consequently, $\bar{\mu} = \check{\mu}$.
 Notice that the segments in the partition $(a,b]=\bigcup_{n=1}^N A_n$ can be chosen with rational endpoints (of course except $a$ and $b$). Then the supremum obtained in \eqref{maxbormesform} will be the same.
\end{remark}

\section{Cylindrical continuous martingales and quadratic variation\label{sec:cyl}}

In this section we assume that $X$ is a Banach space with a separable dual space $X^*$.
 Let $(\Omega, \mathbb F, \mathbb P)$ be a complete probability space with filtration
$\mathbb F := (\mathcal F_t)_{t \in \mathbb R_+}$ that satisfies the usual conditions,
and let $\mathcal F := \sigma(\bigcup_{t\geq 0} \mathcal F_t)$.
We denote the predictable $\sigma$-algebra by $\mathcal P$.

In this section we introduce a class of cylindrical continuous local martingales on a Banach space $X$ which have a certain quadratic variation. We will show that it extends several previous definitions from the literature even in the Hilbert space setting.

\subsection{Definitions\label{sec:defcyllocalmart}}

A scalar-valued process $M$ is called a continuous local martingale if there exists a sequence of stopping times $(\tau_n)_{n\geq 1}$ such that $\tau_n\uparrow \infty$ almost surely as $n\to \infty$ and $\one_{\tau_n>0} M^{\tau_n}$ is a continuous martingale.

Let $\mathcal M$ (resp.\ $\mathcal M^{\rm loc}$) be the class of continuous (local) martingales.
On $\mathcal M^{\rm loc}$
define the translation invariant metric given by
\begin{equation}\label{eq:metric}
 \|M\|_{\mathcal M^{\rm loc}} = \sum_{n=1}^{\infty} 2^{-n}\mathbb E[1 \wedge \sup_{t\in [0,n]} |M_t|].
\end{equation}
Here and in the sequel we identify indistinguishable processes.
One may check that this coincides with the ucp topology (uniform convergence compact sets in probability). The following characterization will be used frequently.
\begin{remark}\label{rem:Kal176}
For a sequence of continuous local martingales one has
$M^n\to 0$ in $\mathcal M^{\rm loc}$ if and only if for every $T\geq 0$, $[M^n]_T\to 0$ in probability and $M^n_0\to 0$ in probability (see \cite[Proposition 17.6]{Kal}).
\end{remark}

The space $\mathcal M^{\rm loc}$ is a complete metric space. This is folklore, but we include a proof for convenience of the reader. Let $(M^n)_{n\geq 1}$ be a Cauchy sequence in $\mathcal M^{\rm loc}$ with respect to the ucp topology. By completeness of the ucp topology we obtain an adapted limit $M$ with continuous paths. It remains to shows that $M$ is a continuous local martingale. By taking an appropriate subsequence without loss of generality we can suppose that $M^n \to M$ a.s.\ uniformly on compacts sets. Define a sequence of stopping times $(\tau_k)_{k=1}^{\infty}$ as follows:
\begin{equation*}
\tau_k = \begin{cases}
\inf\{t\geq 0:\sup_{n}\|M^n(t)\|>k\}, &\text{if}\; \sup_{t\in \mathbb R_+}\sup_{n}\|M^n(t)\|>k;
\\
\infty,&\text{otherwise.}
\end{cases}
\end{equation*}
Since each $(M^n)^{\tau_k}$ is a bounded local martingale with continuous paths, $(M^n)^{\tau_k}$ is a martingales as well by the dominated convergence theorem. Letting $n\to \infty$, it follows again by dominated convergence theorem  that $M^{\tau_k}$ is a martingale. Therefore, $M$ is a continuous local martingale with a localizing sequence $(\tau_k)_{k=1}^{\infty}$.

\smallskip

Let $X$ be a Banach space. A continuous linear mapping $M:X^* \to \mathcal M^{\rm loc}$ is called a  \textit{cylindrical continuous local martingale}.
(Details on cylindrical martingales can be found in \cite{BU,JKFR}).
For a cylindrical continuous local martingale $M$ and a stopping time $\tau$ we can define $M^{\tau}:X^* \to \mathcal M^{\rm loc}$ by $M^{\tau}x^*(t) = Mx^*(t\wedge\tau)$. In this way $M^{\tau}$ is a cylindrical continuous (local) martingale again. Two cylindrical continuous local martingales $M$ and $N$ are called  \textit{indistinguishable} if $\forall x^*\in X^*$ the local martingales $Mx^*$ and $Nx^*$ are indistinguishable.

\begin{remark}
On $\mathcal M^{\rm loc}$ it is also natural to consider the Emery topology, see \cite{Em} and also \cite{Kar,BU,JKFR}. Because of the continuity of the local martingales we consider, this turns out to be equivalent. We find it therefore more convenient to use the ucp topology instead.
\end{remark}

\begin{remark}\label{rem:bili}
Since $X^*$ is separable, we can find linearly independent vectors $(e^*_n)_{n \geq 1}\subseteq X^*$ with linear span $F$ which is dense in $X^*$. For fixed $t \geq 0$ and almost all $\omega \in \Omega$
one can define $B_t : \Omega \to \mathbf B(F,F)$ such that
$B_t(x^*,y^*) = [Mx^*,My^*]_t$ for all $x^*,y^* \in F$.
Unfortunately, one can not guarantee, that $t\mapsto B_t$ is continuous a.s.\
Moreover, as we will see in Example \ref{countex} for $X$ a Hilbert space, it may already happen that
for a.a. $\omega\in \Omega$, for some $t>0$, $B_t\notin\mathbf B(X^*, X^*)$.
\end{remark}

In the following definition we introduce a new set of cylindrical martingales for which the above phenomenon does not occur.

Let $(\Omega, \mathcal F, \mathbb P)$ be a probability space, $(S, \Sigma)$ be a measure space and let $\mathcal M_+(S,\Sigma)$ be a set of all positive measures on $(S, \Sigma)$. For $f,g:\Omega\to \mathcal M_+(S,\Sigma)$ we say that $f\geq g$ if $f(\omega)\geq g(\omega)$ for $\mathbb P$-a.a. $\omega\in \Omega$.

\begin{definition}\label{def:cylmart}
 Let $M:X^* \to \mathcal M^{\rm loc}$ be a linear mapping. Then $M$ is said to have a  \emph{quadratic variation} if
 \begin{enumerate}
 \item[(1)] There exists a smallest $f:\Omega \to \mathcal M_+(\mathbb R_+, \mathcal B(\mathbb R_+))$ such that $f\geq \mu_{[Mx^*]}$ for each $x^*\in X^*$, $\|x^*\| = 1$,
   \item[(2)] $f(\omega)[0,t]$ is finite for a.e. $\omega\in \Omega$ for all $t\geq 0$.
 \end{enumerate}
Let $[[M]]:\mathbb R_+\times \Omega \to \mathbb R_+$ be such that \[[[M]]_t(\omega) = \one_{f(\omega)[0,t]<\infty} f(\omega)[0,t].\]
Then $[[M]]$ is called the  \emph{quadratic variation} of $M$ and we write $M \in \mathcal M_{\rm var}^{\rm loc}(X)$.

If additionally, for each $x^*\in X^*$, $M x^*$ is a martingale, we write $M\in \mathcal M_{\rm var}(X)$.
\end{definition}

Notice that in the definition above $f = \mu_{[[M]]}$ a.s.\ In the next proposition we collect some basic properties of $[[M]]$.
\begin{proposition}
Assume $M \in \mathcal M_{\rm var}^{\rm loc}(X)$. Then
$M$ is a cylindrical continuous local martingale and the following properties hold:
\begin{enumerate}
\item $[[M]]$ has a continuous version.
\item $[[M]]$ is predictable.
\item $[[M]]_0 = 0$ a.s.
\item $[[M]]$ is increasing.
\item For all $x^*\in X^*$ a.s.\ for all $s\leq t$,
\[[M x^*]_t - [Mx^*]_s \leq ([[M]]_t - [[M]]_s) \|x^*\|^2.\]
\end{enumerate}
\end{proposition}
In Example \ref{countex} we will see that not every cylindrical continuous local martingale is in $M_{\rm var}^{\rm loc}(X)$.

\begin{proof}
Properties (3), (4) and (5) are immediate from the definitions. Properties (1) and (2) will be proved in Proposition \ref{prop:[[M]]andF} below.

To prove that $M$ is a cylindrical continuous local martingale, fix $t\geq 0$ and a sequence $(x_n^*)_{n\geq 1}$ such that $x_n^*\to 0$. Then by (5), $[Mx_n^*]_t \to 0$ a.s., so by Remark \ref{rem:Kal176} $M$ is a continuous linear mapping.
\end{proof}

\begin{remark}
 Let $M \in \mathcal M_{\rm var}^{\rm loc}(X)$. Then $M$ is a cylindrical continuous local martingale.

\end{remark}

\begin{proposition}\label{prop:[[M]]andF}
Let $M:X^*\to \mathcal M^{\rm loc}$ be a cylindrical continuous local martingale. Then the following assertions are equivalent:
\begin{enumerate}
\item $M \in \mathcal M_{\rm var}^{\rm loc}(X)$;
\item For any dense subset $(x_n^*)_{n\geq 1}$ of the unit ball in $X^*$ there exists a nondecreasing right-continuous process $F:\mathbb R_+ \times \Omega \to \mathbb R_+$ such that for a.a.\ $\omega \in \Omega$ we have that $\mu_{F(\omega)}=\sup_{n} \mu_{[Mx_n^*](\omega)}$;
\item For any dense subset $(x_n^*)_{n\geq 1}$ of the unit ball in $X^*$ there exists a nondecreasing right-continuous process  $G:\mathbb R_+ \times \Omega \to \mathbb R_+$ such that for a.a. $\omega \in \Omega$ we have that $\mu_{[Mx_n^*](\omega)}\leq \mu_{G(\omega)}$.
\end{enumerate}
Moreover, in this case $F$ is a.s.\ continuous, predictable and $F = [[M]]$ a.s.
\end{proposition}
\begin{proof}
(1) $\Rightarrow$ (2): Since $\mu_{[[M]]}\geq \mu_{[Mx_n^*]}$ a.s.\ for each $n\geq 1$, it follows that a.s.\ there exists $\check{\mu}:=\sup_n \mu_{[Mx^*_n]}\leq \mu_{[[M]]}$ by the definition of a supremum of measures given in Lemma \ref{maxmes}. By Remark \ref{maxbormes} one can write $\check{\mu} = \mu_F$ where the process $F$ is given by
\begin{equation}\label{eq:defFdiff}
F(t) = \sup \sum_{j=1}^J \sup_{n\geq 1} \Big([Mx_n^*]_{t_j}-[Mx_n^*]_{t_{j-1}}\Big),
\end{equation}
where the outer supremum is taken over all $0=t_0<t_1<\ldots <t_J<t$ with $t_j\in \mathbb Q$ for $j\in \{0,\ldots, J\}$. The fact that $F$ is increasing is clear from \eqref{eq:defFdiff}. The right-continuity of $F$ follows from the fact that $\check{\mu}$ is a measure.

(2) $\Rightarrow$ (3): This is trivial.

(3) $\Rightarrow$ (2): Since each of the measures $\mu_{[Mx^*_n]}$ is nonatomic a.s., by \eqref{eq:supofmeas} $\mu_F$ is nonatomic a.s.\ and finite by Remark \ref{rem:mu} and hence $F$ is finite and a.s.\ continuous.

(2) $\Rightarrow$ (1):
We claim that for each $x^*\in X^*$ with $\|x^*\|=1$ a.s.\ $\mu_F \geq \mu_{[Mx^*]}$. Fix $x^*\in X^*$ of norm 1. Since $M$ is a cylindrical continuous local martingale we can find $(n_k)_{k\geq 1}$ such that $x_{n_k}^*\to x^*$ and a.s.\ $[Mx_{n_k}^*] \to [Mx^*]$ uniformly on compact sets as $k\to \infty$ (see \cite[Exercise 17.8]{Kal}). Then a.s.\ for all $0\leq s<t<\infty$ one has that $[Mx_{n_k}^*]_t-[Mx_{n_k}^*]_s\leq F(t)-F(s)$ for each $k\geq 1$, so a.s.\
 \[
 [Mx^*]_t-[Mx^*]_s =\lim_{k\to \infty}[Mx_{n_k}^*]_t-[Mx_{n_k}^*]_s\leq \lim_{k\to \infty}F(t)-F(s) = F(t)-F(s),
 \]
 and therefore $\mu_F\geq \mu_{[Mx^*]}$ a.s. We claim that $F$ is a.s.\ the least function with this property. Let $\phi:\Omega\to \mathcal{M}_+(\mathbb R_+,\mathcal{B}(\mathbb R_+))$ be such that for all $x^*\in X^*$ with $\|x^*\|=1$, $\phi\geq \mu_{[Mx^*]}$ a.s. Then $\phi \geq \sup_n \mu_{[Mx^*_n]} = \mu_F$ a.s.\ and hence $\mu_F$ is the smallest measure with the required property. By the definition of a quadratic variation we find that $F=[[M]]$ a.s.

Finally, note that by \eqref{eq:defFdiff}, $F$ is adapted and therefore $F$ is predictable by the a.s.\ pathwise continuity of $F$.
\end{proof}

\begin{remark}
 Notice that by the above proposition the quadratic variation of $M \in \mathcal M_{\rm var}^{\rm loc}(X)$ has the following form a.s.
 \begin{equation}\label{eq:[[M]]assup}
   [[M]]_t = \sup\sum_{n=1}^N \sup_{m} ([Mx_m^*]_{t_{i+1}}-[Mx_m^*]_{t_{i}}),\;\;\; t\geq 0,
 \end{equation}
 where the limit is taken over all rational partitions $0= t_0 < \ldots < t_N = t$ and $(x_m^*)_{m\geq 1}$ is a dense subset of the unit ball in $X^*$.
\end{remark}

Next we give another characterization of $M$ being in $\mathcal M_{\rm var}^{\rm loc}(X)$.

\begin{theorem}\label{thm:aMdef}
Let $M:X^*\to \mathcal M^{\rm loc}$. Then $M \in \mathcal M_{\rm var}^{\rm loc}(X)$ if and only if there exists a mapping $\vv_M:\mathbb R_+\times\Omega\to \mathbf B(X^*,X^*)$ such that for every $x^*, y^*\in X^*$, a.s.\ for all $t\geq 0$, $\vv_{M}(t)(x^*, y^*) = [Mx^*, My^*]_t$ and a.s.\ for all $t\geq 0$, $(x^*, y^*)\mapsto \vv_{M}(t)(x^*, y^*)$ is bilinear and continuous, and for all $t\geq 0$ the following limit exists
    \begin{equation}\label{eq:funcG}
         G(t) := \lim_{{\rm mesh}\to 0}\sum_{n=1}^N \sup_{\|x^*\|= 1}(\vv_{M}(t_n)(x^*, x^*)-\vv_{M}(t_{n-1})(x^*, x^*)).
    \end{equation}
    Here the limit is taken over partitions $0= t_0 < \ldots < t_N = t$.

    If this is the case then $G(t) = [[M]]_t$ a.s. for all $t\geq 0$.
\end{theorem}

\begin{proof}
 Let $M \in \mathcal M_{\rm var}^{\rm loc}(X)$. Fix a set $(x_m^*)_{m\geq 1}\subset X^*$ of linearly independent vectors such that $\text{span}(x_m^*)_{m\geq 1}$ is dense in $X^*$. Let $F=(y_n^*)\subset X^*$ be the $\mathbb Q$-span of $(x_m^*)_{m\geq 1}$. Then there exists $\vv_M:\mathbb R_+ \times\Omega \to\mathbf B(F,F)$ such that for each $n,k\geq 1$ $\vv_M(y_n^*,y_k^*)$ is a version of $[My_n^*, My_k^*]$ such that $\mu_{\vv_M(y_n^*,y_k^*)}\ll \mu_{[[M]]}\|y_n^*\|\|y_k^*\|$. Since by the last inequality $\vv_M$ is bounded on $F\times F$, it can be extended to $X^*\times X^*$, and by the continuity of $M$, $\vv_M(x^*,y^*)$ is a version of $[Mx^*, My^*]$. To prove \eqref{eq:funcG} notice that because of the boundedness of $\vv_M$ and a density argument one replace the supremum over the unit sphere by the supremum over $x^*\in \{y_n^*:\|y_n^*\|\leq 1\}$. Then this formula coincides with \eqref{eq:[[M]]assup}, therefore a.s.\ $G(t) = [[M]]_t$ for all $t\geq 0$.

To prove the converse first note that for all $x^*\in X^*$, $\mu_{[Mx^*]}\leq \mu_{G} \|x^*\|^2$ a.s. and hence $M$ is a cylindrical continuous local martingale by Remark \ref{rem:Kal176}. Since $\vv_M$ is continuous one can replace the supremum by the supremum over a countable dense subset of the unit ball again. Now one can apply Proposition \ref{prop:[[M]]andF} and use \eqref{maxbormesform}.
\end{proof}

\begin{definition}\label{def:cylDoleansmeasure}
Given $M\in \mathcal M_{\rm var}^{\rm loc}(X)$ we define its  \textit{cylindrical Dol\'eans measure} $\mu_M$ on the predictable $\sigma$-algebra $\mathcal P$ as follows:
\[
\mu_M(C) = \mathbb E \int_0^{\infty}\mathbf 1_C\ud[[M]],\;\;\; C \in \mathcal P.
\]
\end{definition}

\begin{lemma}\label{lemma:stoppingtimequadvar}
Let $M\in \mathcal M_{\rm var}^{\rm loc}(X)$ and let $\tau$ be a stopping time and define a sequence of stopping times by
\[\tau_n = \inf\{t\geq 0: [[M]]_t\geq n\}, \ \  \text{for $n\geq 1$.}\]
Then the following assertions hold:
\begin{enumerate}
\item $M^{\tau}\in M_{\rm var}^{\rm loc}(X)$, $[[M^{\tau}]] = [[M]]^{\tau}$.
\item For each $n\geq 1$, $M^{\tau_n}\in \mathcal M_{\rm var}(X)$.
\end{enumerate}
\end{lemma}
\begin{proof}

(1): It is obvious from the scalar theory that for every $x^*\in X^*$ with $\|x^*\|\leq 1$, $M^{\tau}x^*$ is a continuous local martingale.
Moreover,
\[\ud\mu_{[M^{\tau} x^*]} = \one_{[0,\tau]} \ud\mu_{[M x^*]} \leq \one_{[0,\tau]} \ud\mu_{[[M]]}.\]
Since $\mu_{[[M]]}$ is the least measure which majorizes  $\mu_{[M x^*]}$ for $\|x^*\|=1$, it follows that $\one_{[0,\tau]} \ud\mu_{[[M]]}$ is the least measure which majorizes $\mu_{[M^{\tau} x^*]}$ for $\|x^*\|=1$.

(2): To check that $M^{\tau_n}\in \mathcal M_{\rm var}(X)$ it remains to show that $\one_{\tau_n>0}M^{\tau_n} x^*$ is a martingale. By the Burkholder-Davis-Gundy inequality \cite[Theorem 26.12]{Kal} and the continuity of $[[M]]$ we have for all $x^*\in X^*$
\[\mathbb E \sup_{s\leq t} |\one_{\tau_n>0} M^{\tau_n}_t x^*|\leq C \mathbb E [M^{\tau_n} x^*]^{1/2}_t = C\mathbb E [[M]]^{1/2}_{t\wedge \tau_n} \|x^*\|\leq C n^{1/2}\|x^*\|.\]
Therefore, the martingale property follows from the dominated convergence theorem and the fact that $\one_{\tau_n>0} M^{\tau_n} x^*$ is a local martingale.
\end{proof}

We end this subsection with a simple but important example.
\begin{example}[Cylindrical Brownian motion]\label{wiencyl}
Let $X$ be a Banach space and $Q\in \mathcal L(X^*,X)$ be a positive self-adjoint operator.
Let $W^Q:\mathbb R_+ \times X^* \to L^2(\Omega)$ be a cylindrical $Q$-Brownian motion  (see \cite[Chapter 4.1]{DPZ}), i.e.
\begin{itemize}
 \item $W^Q(\cdot)x^*$ is a Brownian motion for all $x^* \in X$,
 \item $\mathbb E W^Q(t)x^*\,W^Q(s)y^* = \langle Qx^*,y^*\rangle \min\{t,s\}$
 $\forall x^*,y^* \in X^*$, $t,s \geq 0$.
\end{itemize}
The operator $Q$ is called the  \textit{covariance operator} of $W^Q$.
Then $W^Q\in \mathcal M_{\rm var}(X)$. Indeed, since $\vv_{W^Q}(t)(x^*, x^*) = t \lb Q x^*, x^*\rb$ we have $[[M]]_t = t \|Q\|$.

In the case $Q = I$ is the identity operator on  a Hilbert space $H$, we will call $W_H = W^I$ an  \textit{$H$-cylindrical Brownian motion}. In this case $[[M]]_t = t$.
\end{example}

\subsection{Quadratic variation operator\label{subsec:quadratvar}}
Let $M\in \mathcal M_{\rm var}^{\rm loc}(X)$.
From Example \ref{wiencyl} one sees that essential information about the cylindrical martingale is lost when one only considers $[[M]]$. For this reason we introduce the quadratic variation operator $A_M$ and its $[[M]]$-derivative $Q_M$.

Let $\Omega_0\subset \Omega$ be a set of a full measure such that $G(t)$ from \eqref{eq:funcG} is finite for all $t\geq 0$ in $\Omega_0$. Note that pointwise in $\Omega_0$ for all $t\geq 0$,
\[|\vv_{M}(t)(x^*, y^*)| \leq [[M]]_t \|x^*\| \|y^*\| \ \ \ \forall x^*,y^*\in F.\]
It follows that for all $\omega\in \Omega_0$ for all $t\geq 0$ and all $x^*\in X^*$, the bilinear map $(x^*,y^*)\mapsto \vv_{M}(t,\omega)(x^*, y^*)$ is bounded by $[[M]]_t(\omega)$ in norm, and therefore it defines a mapping $A_M(t,\omega)\in \calL(X^*, X^{**})$. For $\omega\notin \Omega_0$ we set $A_M = 0$. Note that for each $x^*, y^*\in X^*$, for almost all $\omega\in \Omega$, and for all $t\geq 0$, $\lb A_M(t) x^*, y^*\rb$ is a version of $[Mx^*, My^*]_t$.
The function $A_M$ is called  \textit{the quadratic variation operator of $M$}. By construction, for every $x^*, y^*\in X^*$, $(t,\omega)\mapsto \lb A_M (t,\omega)x^*, y^*\rb$ is predictable.
Moreover, one can check that for each $t\geq 0$ and $\omega\in \Omega$, and $x^*,y^*\in X^*$,
\[\lb A_M(t,\omega)x^*, x^*\rb\geq 0, \ \ \ \text{and} \ \ \ \lb A_M(t,\omega) x^*, y^*\rb = \lb A_M(t,\omega) y^*, x^*\rb.\]

\begin{proposition}[Polar decomposition]\label{Q_M}
For each $M\in \mathcal M_{\rm var}^{\rm loc}(X)$ there exists a process $Q_M: \mathbb R_+ \times \Omega \to \mathcal L(X^*,X^{**})$
such that almost surely for all $t>0$
\begin{equation}\label{eq:derivativeAM}
\langle A_M(t) x^*,y^*\rangle = \int_0^t \langle Q_M(s)x^*,y^*\rangle \ud [[M]]_s
,\;\;\; x^*,y^* \in X^*.
\end{equation}
Moreover, the following properties hold:
\begin{enumerate}
\item For all $x^*,y^*\in X^*$, $(t,\omega)\mapsto \lb Q_M(t,\omega)x^*, y^*\rb$ is predictable.
\item $Q$ is self-adjoint and positive $\mu_M$-a.e.
\item $\|Q_M(t)\| =1$ for $\mu_{[[M]]}$-a.e. $t$ on $\mathbb R_+$. In particular, $\|Q_M(t,\omega)\|=1$, $\mu_M$-a.s.\ on $\mathbb R_+\times \Omega$.
\end{enumerate}
\end{proposition}
In \eqref{eq:derivativeAM} the Lebesgue-Stieltjes integral is considered.
In the proof we use the following fact which  is closely related to \cite[Theorem 5.8.8]{Bog} and \cite[Theorem 3.21]{Fol}. In the statement and its proof we use the convention that $\frac{0}{0}=0$.
\begin{lemma}\label{lem:Lebesguediff}
Let $\mu$ be a positive non-atomic $\sigma$-finite measure on $\mathbb R_+$ and let $f\in L^1_{\rm loc}(\mathbb R_+,\mu)$. Define the measure $\nu$ by
$\ud\nu = f\ud\mu$. Then for $\mu$-almost all $t>0$,
\begin{equation*}
 \lim_{\varepsilon\downarrow 0}\frac{\nu((t-\varepsilon, t])}{\mu((t-\varepsilon, t])} = f(t).
\end{equation*}
\end{lemma}

\begin{proof}
It is enough to show this lemma given $\mu \geq \lambda$. If it is shown for $\mu \geq \lambda$, then in general situation one can use $\mu+\lambda$: due to the fact that $\mu\ll \mu + \lambda$ one has that there exists $g:\mathbb R_+ \to \mathbb R_+$ such that $\ud\mu = g\ud(\mu + \lambda)$ and $\ud\nu = fg\ud(\mu + \lambda)$, so for $\mu$-a.a. $t\geq 0$
 \begin{align*}
  \lim_{\varepsilon\downarrow 0}\frac{\nu((t-\varepsilon, t])}{\mu((t-\varepsilon, t])}
  &=\lim_{\varepsilon\downarrow 0}\frac{\nu((t-\varepsilon, t])}{(\mu+\lambda)((t-\varepsilon, t])}\Big{\slash}\lim_{\varepsilon\downarrow 0}\frac{\mu((t-\varepsilon, t])}{(\mu+\lambda)((t-\varepsilon, t])}\\
  &= \frac{(fg)(t)}{g(t)} = f(t).
 \end{align*}

Now let $\mu \geq \lambda$, and define $\tau:\mathbb R_+\to \mathbb R_+$ by $\tau(t) = \inf\{s:\mu([0,s)) > t\}$. Then $\mu\circ\tau = \lambda$ is the Lebesgue measure on $\mathbb R_+$, $\ud(\nu\circ\tau) = f\circ\tau\ud\lambda$. By the Lebesgue differentiation theorem (see \cite[Theorem 3.21]{Fol}) one has
\begin{equation}\label{eq:Lebesguediff1}
  \lim_{\varepsilon\downarrow 0}\frac{\nu\circ\tau((t-\varepsilon, t])}{\mu\circ\tau((t-\varepsilon, t])} = f\circ \tau(t),
\end{equation}
 for $\lambda$-almost all $t$. Define $F:\mathbb R_+ \to \mathbb R_+$ by $F(s) =\mu([0,s))$. Then $F$ is strictly increasing and continuous since $\mu$ is nonatomic. Therefore $\tau\circ F(s) = s$ for all $s\in \mathbb R_+$, and it follows from \eqref{eq:Lebesguediff1} that for $\mu$-a.a. $t\in \mathbb R_+$
 \begin{align*}
  \lim_{\varepsilon\downarrow 0}\frac{\nu((t-\varepsilon, t])}{\mu((t-\varepsilon, t])} = \lim_{\varepsilon\downarrow 0}\frac{\nu\circ\tau\circ F((t-\varepsilon, t])}{\mu\circ\tau\circ F((t-\varepsilon, t])}&=\lim_{\varepsilon\downarrow 0}\frac{\nu\circ\tau((F(t-\varepsilon), F(t)])}{\mu\circ\tau((F(t-\varepsilon), F(t)])} \\
  &= f\circ \tau (F(t)) = f(t).
 \end{align*}
 \end{proof}

\begin{proof}[Proof of Proposition \ref{Q_M}]
Let $\Omega_0\subset \Omega$ be a set of a full measure such that $G(t)$ from \eqref{eq:funcG} is finite for all $t\geq 0$ in $\Omega_0$.
Then pointwise on $\Omega_0$, for all $x^*,y^* \in X^*$, $\langle A_Mx^*,y^*\rangle$
is absolutely continuous with respect to $[[M]]$. Let $(e^*_n)_{n \geq 1}\subseteq X^*$ be
a set of linearly independent vectors, such that its linear span
$F$ is dense in $X^*$.
Then there exists a process $Q_M\colon  \Omega\times \mathbb R_+ \to \mathcal L(F,X^{**})$
such that $\langle Q_Me^*_n,e^*_m\rangle$ is predictable for each $n,m \geq 1$ and
$\int_0^t\langle Q_M(s) e^*_n,e^*_m\rangle \ud [[M]]_s = \langle A_M(t) e^*_n,e^*_m\rangle$.
To check the predictability, note that by Lemma \ref{lem:Lebesguediff} a.s.\ for $\mu_{[[M]]}$-a.a.\ $t\geq 0$,
\[\frac{\lb A_M(t)e_n^*, e_m^*\rb-\lb A_M(t-1/k)e_n^*, e_m^*\rb}{[[M]]_{t} - [[M]]_{t-\frac 1k}}\to \lb Q_M(t)e_n^*, e_m^*\rb \]
as $k\to \infty$. Since the left-hand side is predictable, the right-hand side has a predictable version.

Let $(f_n^*)_{n \geq 1}$ in $F$ of length one be dense in $\{x^*\in X^*: \|x^*\|=1\}$.
Then by the definition of $[[M]]$, on $\Omega_0$ it holds that $|\mu_{\vv_{M}(\cdot)(x^*, x^*)} |\leq \mu_{[[M]]}$ for all $m,n\geq 1$. Therefore on $\Omega_0$, we find that for all $m,n\geq 1$,
$|\langle Q_M(s) f^*_n,f^*_m\rangle|\leq 1$ for $\mu_{[[M]]}$-a.a.\ $t\geq 0$. Let $S\subseteq\mathbb R_+\times\Omega_0$ be the set where $|\langle Q_M(s) f^*_n,f^*_m\rangle|\leq 1$ for all $m,n\geq 1$. Then $S$ is predictable and for each $\omega\in \Omega_0$, $\mu_{[[M]]}(\mathbb R\setminus S_{\omega})$ = 0, where $S_{\omega}$ denotes its section. Taking the supremum over all $n, m\geq 1$, it follows that  $\|Q_M\|\leq 1$ on $S$. On the complement of $S$ we let $Q_M = 0$.
Since $F$ is dense in $X^*$, $Q_M$ has a unique continuous extension to a mapping in $\calL(X^*,X^{**})$.

Fix $t > 0$. Let $x^*, y^* \in X^*$, $(x_n^*)_{n\geq 1}, (y_n^*)_{n \geq 1} \subseteq F$
be such that $x^* =\lim_{n \to \infty} x_n^*$, $y^* = \lim_{n \to \infty} y_n^*$.
Since on $\Omega_0$ for all $t\geq 0$,
\[\langle A_M(t) x_n^*,y_m^*\rangle = \int_0^t \langle Q_M(s)x_n^*,y_m^*\rangle \ud [[M]]_s,\]
letting as $m,n\to \infty$, \eqref{eq:derivativeAM}  follows by the dominated convergence theorem.

We claim that for all $\omega\in \Omega_0$,  $\|Q_M\|=1$, $\mu_{[[M]]}$-a.e.\ on $\mathbb R_+$.
Since $\mu_{[[M]]}$ is a maximum for the measures $\mu_{\vv_{M}(\cdot)(f^*_n, f^*_n)}$ (where the $f_n^*$ are as before)
it follows that $\|Q_M\| = \sup_n \langle Q_Mf^*_n,f^*_n\rangle = 1$, $\mu_{[[M]]}$-a.e.\ on $\mathbb R_+$. Indeed, otherwise there exists an $\alpha\in (0,1)$ such that $C = \{t\in \mathbb R_+: \|Q(t)\|<\alpha\}$ satisfies $\mu_{[[M]]}(C)>0$. Then it follows that for the maximal measure and all measurable $B\subseteq C$
\[ \mu_{a_M(\cdot)(f_n^*, f_n^*)}(B) = \int_{B} \langle Q_M(s)f_n^*,f_n^*\rangle \ud [[M]]_s\leq \alpha \mu_{[[M]]}(B).\]
This contradicts the fact that the supremum measure on the left equals $\mu_{[[M]]}$ as well. Thus $\mu_{[[M]]}(C) = 0$ and hence the claim follows.

It follows from the construction that $Q_M$ is self-adjoint and positive $\mu_{[[M]]}$-a.s.
\end{proof}

\begin{remark}
Assume that $X^{**}$ is also separable (e.g.\ $X$ is reflexive).
In this case it follows from the Pettis measurability theorem that the functions $A_Mx^*$ and $Q_M x^*$
are strongly progressively measurable for each $x^* \in X^*$ (see e.g. \cite{HNVW1}).
Moreover, if $\phi:\mathbb R_+ \times \Omega \to X^*$ is strongly progressively measurable,
then $A_M\phi$ and $Q_M\phi$ are strongly progressively measurable as well.
\end{remark}

\begin{remark}
Let $H$ be a separable Hilbert space and $X$ be a separable Banach space.
In \cite{MR1, Ond, Ond1} cylindrical continuous martingales are considered for which
the quadratic variation operator has the form
\[\lb A_M(t)x^*, y^*\rb = \int_0^t (g^*x^*, g^*y^*)_H\ud s,\]
where $g: \mathbb R_+\times \Omega \to \mathcal L(H,X)$ is such that for all $x^*\in X^*$, $g^*x^*\in L^2_{\rm loc}(\mathbb R_+;H)$. In this case $[[M]]_t = \int_0^t \|gg^*\|\ud s$. Indeed,
\[\vv_{M}(b)(x^*, x^*) - \vv_{M}(a)(x^*, x^*)= \int_{(a,b]}\|g(s)^*x^*\|^2_H \ud s
\]
and hence the identity follows from Lemma \ref{lem:maxmeasuref}, Remark \ref{maxbormes}, Theorem \ref{thm:aMdef} and the separability of $X^*$.
\end{remark}

\subsection{Quadratic variation for local martingales}

In this section we will study the case where the cylindrical local martingale actually comes from a local martingale on $X$. We discuss several examples and compare our definition quadratic variation from Definition \ref{def:cylmart} to other definitions. In order to distinguish between martingales and cylindrical martingales we use the notation $\wt M$ for an $X$-valued martingale.

For a continuous local martingale $\widetilde{M}\colon \mathbb R_+ \times \Omega \to X$
we define the  \textit{associated}  cylindrical continuous martingale
$M:\mathbb R_+ \times X^* \to L^0(\Omega)$ by
\[
Mx^* = \langle\widetilde M,x^*\rangle,\,\,\,x^* \in X^*.
\]
It is a cylindrical continuous local martingale since if $(x^*_n)_{n\geq 1}\subseteq X^*$ vanishes as $n \to \infty$, then for all $t\geq 0$ almost all $\omega$
\[
\sup_{0\leq s\leq t}|\langle\widetilde M_s(\omega),x^*_n\rangle|\leq\|x^*_n\|\sup_{0\leq s\leq t}\|\widetilde M_s(\omega)\|\to 0\;\;\; n\to\infty,
\]
so $\langle\widetilde M,x^*_n\rangle\to 0$ in the ucp topology.

Below we explain several situations where one can check that the associated cylindrical continuous local martingale $M$ is an element of $\mathcal M_{\rm var}^{\rm loc}(X)$. In general this is not true (see Example \ref{ex:countexbanval}).

First we recall some standard notation in the case $H$ is a separable Hilbert space. Let $\wt{M}:\mathbb R_+\times\Omega\to H$ be a continuous local martingale. Then the quadratic variation is defined by
\begin{equation}\label{eq:Mtclassic}
[\wt M]_t = \mathbb P-\lim_{{\rm mesh}\to 0}  \sum_{n=1}^N \|M_{t_n}- M_{t_{n-1}}\|^2.
\end{equation}
where $0=t_0<t_1<\ldots<t_N=t$.
It is well known that this limit exists in the ucp sense (see \cite[2.6 and 3.2]{MP}) and the limit coincides with the unique increasing and continuous process starting at zero such that $\|\wt M\|^2 - [\wt M]$ is a continuous local martingale. Moreover, one can always choose a sequence of partitions with ${\rm mesh}\to 0$ for which a.s.\ uniform convergence on compact intervals holds.

Observe that for an orthonormal basis $(h_n)_{n\geq1}$, letting $M^n_t = (\wt M_t, h_n)_H$ we find that almost surely for all $t\geq 0$
\[\wt M_t = \sum_{n\geq 1} M_t^n h_n\]
with convergence in $H$. Moreover, the following identity for the quadratic variation  holds (see \cite[Chapter 14.3]{MP}):
a.s.
\begin{align}\label{eq:quadtrace}
[\wt M]_t = \sum_{n\geq 1} [M^n]_t, \ \ \text{for all $t\geq 0$}.
\end{align}

Next we first consider two finite dimensional examples before returning to the infinite dimensional setting.
\begin{example}
Let $M \in \mathcal M_{\rm var}^{\rm loc}(\mathbb R)$. Then $\wt M = M 1$ is a continuous
real-valued local martingale, $[[M]] = [\wt M]$ and $Q_M=1$ (where $Q_M$ is as in Proposition \ref{Q_M}).
\end{example}

\begin{example}\label{ex:findim}
Let $d \geq 1$ and $H = \mathbb R^d$.
Again let $M\in \mathcal M_{\rm var}^{\rm loc}(H)$.
Let $h_1,\ldots,h_d$ be an orthonormal basis in $H$. Then $\wt{M} = \sum_{n=1}^d M h_n\otimes h_n$ defines a continuous $H$-valued local martingale.
Moreover, its quadratic variation satisfies
\[[\wt M]_t = \sum_{n=1}^d [Mh_n]_t, \ \ t\geq 0.\]
and in particular the right-hand side does not depend on the choice of the orthonormal basis.
It follows that
\[
\frac{[\wt M]_t -[\wt M]_s}{d} \leq \sup_{\|h\|=1} ([Mh]_t - [Mh]_s) \leq  [\wt M]_t -[\wt M]_s,\,\,\, t>s\geq0,
\]
and hence from the definition of $[[M]]$ we see that
\[
\frac{[\wt M]_t -[\wt M]_s}{d} \leq [[M]]_t - [[M]]_s \leq  [\wt M]_t -[\wt M]_s,\,\,\, t>s\geq0,
\]
which means that the Lebesgue-Stieltjes measures $\mu_{[\wt M]}$ and $\mu_{[[M]]}$ are equivalent a.s.
\end{example}

\begin{example}\label{ex:martingaleassociated}
Let $H$ be a separable Hilbert space again and let $\widetilde M$ be an $H$-valued continuous local martingale.
The  \emph{quadratic variation operator} (see \cite[Chapter 14.3]{MP}) $\langle\widetilde M \rangle:\mathbb R_+ \times \Omega \to \mathcal L(H)$ is defined by
\[
\langle\langle\widetilde M \rangle_t h,g\rangle = [\langle\widetilde M,h\rangle,\langle\widetilde M , g\rangle]_t,
\;\;\;\; \omega\in \Omega, t\geq 0.
\]
To see that this is well-defined and bounded of norm at most $[M]_t$, choose partitions with decreasing mesh sizes such that the convergence in \eqref{eq:Mtclassic} holds on a set of full measure $\Omega_0$. Then a polarization argument shows that pointwise on $\Omega_0$,
\begin{align*}
|[\langle\widetilde M,h\rangle,\langle\widetilde M , g\rangle]_t| & = \lim \Big|\sum_{n=1}^N \langle\widetilde \Delta M_{t_n}, h\rangle  \langle\widetilde \Delta M_{t_n} , g\rangle\Big|
\\ & \leq \lim  \sum_{n=1}^N \|\Delta M_{t_n}\|^2 \|g\| \, \|h\| = [M]_t \|g\| \, \|h\|.
\end{align*}
The operator $\langle\widetilde M \rangle_t$ is positive and it follows from \eqref{eq:quadtrace} that for any orthonormal basis $(h_n)_{n\geq 1}$  of $H$, pointwise on $\Omega_0$ for all $t\geq 0$,
\[\sum_{n\geq 1}\lb \langle\widetilde M \rangle_t h_n, h_n\rb = \sum_{n\geq 1}[(\wt M, h_n)]_t = [M]_t.\]
Hence a.s.\ for all $t\geq 0$, $\langle\widetilde M \rangle_t$ a trace class operator and  $\text{Tr}\langle\widetilde M \rangle_t = [M]_t$.

As in Proposition \ref{Q_M} one sees that there is a $q_{\widetilde M}:\mathbb R_+ \times \Omega
\to \mathcal L(H)$ such that for all $g,h\in H$,  $\lb q_{\wt M} g, h\rb $ is predictable and a.s.\
\[
\langle \widetilde M\rangle_t = \int_0^t q_{\widetilde M}(s)\ud[\widetilde M]_s, \;\;\; t > 0.
\]
Moreover, a.s.\ $q_{\widetilde M}$ is a trace class operator $\mu_{[\widetilde M]}$-a.a.,
and $\text{Tr} (q_{\widetilde M}(t)) = 1$~a.s. for all $t\geq 0$.

Define $M:\mathbb R_+ \times H\to L^0(\Omega)$ by the formula
\[
Mh := \langle\widetilde M,h\rangle,\,\,\,h \in H.
\]
We claim that $M\in M_{\rm var}^{\rm loc}(H)$. As before a.s.\ for $\forall t>s>0$,
$\sup_{\|h\| = 1}[Mh]_t - [Mh]_s \leq [\widetilde M]_t - [\widetilde M]_s$, so
$[[M]]_t - [[M]]_s \leq [\widetilde M]_t - [\widetilde M]_s$, which means that a.s.\ $[[M]]_t$ is continuous in $t$. Such $M$ is called the  \textit{associated} local $H$-cylindrical martingale.

Now we find that almost surely, for all $h,g \in H$ and $t \geq 0$
\[
\int_0^t \langle Q_M(s) g,h\rangle \ud [[M]]_s = [Mh,Mg]_t =
[\langle \widetilde M,h\rangle,\langle \widetilde M,g\rangle]_t = \int_0^t \langle q_{\widetilde M}(s)
g,h\rangle\ud [\widetilde M]_s.
\]

Moreover, an approximation argument yields that for all elementary progressive processes $\phi,\psi:\mathbb R_+\times\Omega\to H$
\begin{equation}\label{cylhil}
\int_0^{\infty}\langle Q_{M}(s) \phi(s),\psi(s)\rangle \ud [[M]]_s =
\int_0^{\infty}\langle q_{\widetilde M}(s) \phi(s),\psi(s)\rangle\ud [\widetilde M]_s.
\end{equation}
\end{example}

\begin{remark}
Example \ref{ex:martingaleassociated} illustrates some of the advantages using $[[M]]$ instead of $[M]$.
Indeed, $[M]$ is rather large and in order to compensate for this $q_M$ has to be small (of trace class). On the other hand $[[M]]$  is so small that only the boundedness of $Q_M$ is needed. The above becomes even more clear in the cylindrical case, where $[M]$ and $q_M$ are not defined at all.
\end{remark}

Let $X$ be a Banach space, $\widetilde{M}\colon \mathbb R_+ \times \Omega \to X$
be a continuous local martingale. Then we say that $\widetilde{M}$ has
a  \textit{scalar quadratic variation} (see \cite[Definition 4.1]{DGFR}),
if for any $t >0$
\begin{equation}\label{banachvalued}
 [\widetilde{M}]^{\varepsilon}_t :=
\int_0^t \frac{\|\widetilde{M}_{s+\varepsilon}-\widetilde{M}_s\|^2}{\varepsilon}\ud s
\end{equation}
has a ucp limit as $\varepsilon \to 0$. In this case the limit will be denoted by $[\widetilde{M}]_t:= \mathbb P-\lim_{\varepsilon\to 0}[\widetilde{M}]_t^{\varepsilon}$.
Since in the Hilbert space case the above limit coincides with the previously defined quadratic variation, there is no risk of confusion here (see \cite[Remark~4.3.3-4.3.4]{DGFR}).

\smallskip

Outside the Hilbert space setting it is not so simple to determine whether the scalar quadratic variation exists. Also note that the definition can not be extended to cylindrical (local) martingales. In the next example we show that the existence of $[M]$ implies the existence of $[[M]]$.

\begin{example}
Let $\widetilde M$ be an $X$-valued continuous local martingale
with a scalar quadratic variation. Then the associated cylindrical continuous local martingale
$Mx^* := \langle\widetilde M,x^*\rangle$ for $x^* \in X^*$ is in $\mathcal M_{\rm var}^{\rm loc}(X)$.
Indeed, choose a sequence $\varepsilon_n\to 0$ such that the limit in \eqref{banachvalued} converges uniformly on compact intervals on a set of full measure $\Omega_0$.
Then for every $\omega\in \Omega_0$, $t>s\geq 0$, $x^*\in X^*$,
\begin{equation*}
 [Mx^*]_t - [Mx^*]_s  = \lim_{n\to \infty}\int_s^t\frac{|(\widetilde Mx^*)_{ r+\varepsilon_n}-(\widetilde Mx^*)_{r}|^2}{\varepsilon_n}\ud { r}
\leq ([\widetilde M]_t - [\widetilde M]_s) \|x^*\|.
\end{equation*}
Therefore, $[[M]]$ exists and for all $\omega_0\in \Omega$, $t> s\geq 0$, $[[M]]_t - [[M]]_s \leq [\widetilde M]_t - [\widetilde M]_s$.  With a similar argument one sees that the existence of the tensor quadratic variations of \cite{DGFR} implies the existence of $[[M]]$.

It follows from Example \ref{ex:countexbanval} that there are martingales which do not admit a scalar (or tensor) quadratic variation. We do not know if the existence of $[[M]]$ implies that $[M]$ (or its tensor quadratic variation) exists in general.
\end{example}

\subsection{Cylindrical martingales and stochastic integrals}
Let $X, Y$ be two Banach spaces, $x^* \in X^*, y \in Y$. We denote by $x^* \otimes y\in\mathcal L(X,Y)$
the following linear operator: $x^* \otimes y: x \mapsto \langle x^*,x\rangle y$.

Let $X$ be a Banach space. The process $\Phi: \mathbb R_+ \times \Omega \to \mathcal L(H,X)$ is called  \textit{elementary progressive}
with respect to the filtration $\mathbb F = (\mathcal F_t)_{t \in \mathbb R_+}$ if it is of the form
\[
\Phi(t,\omega) = \sum_{n=1}^N\sum_{m=1}^M \mathbf 1_{(t_{n-1},t_n]\times B_{mn}}(t,\omega)
\sum_{k=1}^K h_k \otimes x_{kmn},
\]
where $0 \leq t_0 < \ldots < t_n <\infty$, for each $n = 1,\ldots, N$ the sets
$B_{1n},\ldots,B_{Mn}\in \mathcal F_{t_{n-1}}$ and vectors $h_1,\ldots,h_K$ are orthogonal.
For each elementary progressive
$\Phi$ we define the stochastic integral with respect to $M\in \mathcal M_{\rm var}^{\rm loc}(H)$
as an element of $L^0(\Omega; C_b(\mathbb R_+;X))$ as
\begin{equation}\label{intnorm}
 \int_0^t \Phi(s) \ud M(s) = \sum_{n=1}^N\sum_{m=1}^M \mathbf 1_{B_{mn}}
\sum_{k=1}^K (M(t_n\wedge t)h_k - M(t_{n-1}\wedge t)h_k)x_{kmn}.
\end{equation}
Often we will write $\Phi \cdot M$ for the process $\int_0^\cdot \Phi(s) \ud M(s)$.

\begin{remark}\label{realvalued}
For all progressively measurable processes $\phi:\mathbb R_+\times\Omega\to \calL(H, \mathbb R)$ with $\phi Q_M^{1/2}\in L^2(\mathbb R_+,[[M]];\calL(H,\mathbb R))$ one has
\begin{equation}\label{eq:quadvarsimpleint}
 \Big[\int_0^{\cdot} \phi \ud M\Big]_t = \int_0^t \phi(s) Q_M(s) \phi^*(s) \ud [[M]]_s.
\end{equation}
This can be proved analogously to \cite[(14.7.4)]{MP}.

One can also prove that in the situation above for each stopping time $\tau:\Omega \to \mathbb R_+$ a.s.\ for all $t\geq 0$
\begin{equation}\label{eq:optstoppingcylcase}
 \Bigl(\int_0^{\cdot} \phi \ud M\Bigr)^{t\wedge\tau} = \int_0^{t} \phi(s) \mathbf 1_{s\leq\tau } \ud M_s = \int_0^{t} \phi \ud M^{\tau}.
\end{equation}
If the domain of $\phi$ is in a fixed finite dimensional subspace $H_0\subseteq H$, then \eqref{eq:optstoppingcylcase} is an obvious multidimensional corollary of \cite[Proposition~17.15]{Kal}. For general $\phi$ it follows from an approximation argument.
Indeed, let $\phi_n:\mathbb R_+\times\Omega\to \calL(H_n, \mathbb R)$, where $H_n\subseteq H$ is fixed finite dimensional for each $n\geq 0$, be such that $\phi_nQ_M^{1/2} \to\phi Q_M^{1/2}$ in $L^2(\mathbb R_+,[[M]];\calL(H,\mathbb R))$ a.s.\ Then thanks to Lemma \ref{lemma:stoppingtimequadvar} $\phi_n Q_M^{1/2}\to\phi Q_M^{1/2}$ in $L^2(\mathbb R_+,[[M^{\tau}]];\calL(H,\mathbb R))$ a.s.\ and $\phi_n\mathbf 1_{\cdot\leq\tau } Q_M^{1/2}\to\phi 1_{\cdot\leq\tau } Q_M^{1/2}$ in $L^2(\mathbb R_+,[[M]];\calL(H,\mathbb R))$ a.s.\ So, using \eqref{eq:optstoppingcylcase} for $\phi_n$, \eqref{eq:quadvarsimpleint} and Remark \ref{rem:Kal176} one obtains \eqref{eq:optstoppingcylcase} for general $\phi$.

\end{remark}

\begin{remark}\label{rem3}
It follows from Remark \ref{rem:Kal176}
that for each finite dimensional subspaces $X_0\subseteq X$ the definition of the stochastic integral can be extended
to all strongly progressively measurable processes
$\Phi\colon \mathbb R_+\times \Omega \to \mathcal L(H,X)$ that take values in $\mathcal L(H,X_0)$,
and satisfy $\Phi\,Q_M^{1/2}\in L^2(\mathbb R_+,[[M]]; \mathcal L(H,X))$ a.s.\ (or equivalently $\Phi\,Q_M^{1/2}$
is scalarly in $L^2(\mathbb R_+,[[M]]; H)$ a.s.).  In order to deduce this result from the one-dimensional case
one can approximate $\Phi$ by a process which is supported on a finite dimensional subspace of $H$ and use Remark \ref{rem:Kal176} together with \eqref{eq:quadvarsimpleint} and the fact that $X_0$ is isomorphic to $\mathbb R^d$ for some $d\geq 1$ since it is finite dimensional.
The space of stochastic integrable $\Phi$ will be characterized in Theorem \ref{main}.
\end{remark}

\begin{proposition}\label{prop:cylmartBM}
Let $H$ be a Hilbert space. Let $N\in \mathcal M_{\rm var}^{\rm loc}(H)$.
Let $\Phi:\mathbb R_+ \times\Omega \to \mathcal L(H,X)$ be such that for each $x^*\in X^*$, $\Phi^* x^*$ is progressively measurable and assume that for all $x^*\in X^*$, $\omega \in \Omega_0$, $\lb \Phi(\omega) Q_N(\omega) \Phi^*(\omega) x^*, x^*\rb \in L^1_{\rm loc}(\mathbb R_+,[[N]](\omega))$.
Define a cylindrical continuous local martingale $M:=\int \Phi \ud N$ by
\begin{align}\label{intcyl}
M x^*(t) := \int_0^t \Phi^*x^*\, \ud N,\;\;\; x^*\in X^*.
\end{align}
Then $M\in \mathcal M_{\rm var}^{\rm loc}(X)$ if and only if $\|\Phi Q_N \Phi^*\|\in L^1_{\rm loc}(\mathbb R_+, [[N]])$ a.s.\
In this case,
\begin{equation}\label{eq:quadvarint}
 [[M]]_t = \int_0^t \|\Phi(s) Q_N\Phi^*(s)\| \ud [[N]],  \ \ t\geq 0,
\end{equation}
\begin{align*}
 \langle A_M(t) x^*,y^*\rangle &  = \int_0^t \langle \Phi(s) Q_N \Phi^*(s)x^*,y^*\rangle \ud [[N]]_s, \ \ t\geq 0, x^*, y^*\in X^*,
\\ Q_M(s) &= \frac{\Phi(s) Q_N(s) \Phi^*(s)}{\|\Phi(s)Q_N(s)\Phi^*(s)\|}, \ \ \text{for $\mu_{[[N]]}$-almost all $s\in \mathbb R_+$.}
\end{align*}
\end{proposition}

In this section there are two definitions of a stochastic integral (see \eqref{intnorm} and \eqref{intcyl}).
One can check that both integrals coincide in the sense that \eqref{intcyl} would be the cylindrical continuous martingale associated to the one given in \eqref{intnorm}.

\begin{proof}
We first show that $M$ is a cylindrical continuous local martingale. Clearly, each $M x^*$ is a continuous local martingale. It remains to prove the continuity of $x^*\mapsto M x^*$ in the ucp topology. Fix $T>0$. Let $\Omega_0$ be a set of full measure such that  for $\omega\in \Omega_0$, $t\mapsto \lb \Phi(t,\omega) Q_N(t,\omega)^*\Phi(t,\omega)^* x^*, x^*\rb\in L^1(0,T)$.
By the closed graph theorem for each $\omega \in \Omega_0$ there is a constant $C_T(\omega)$ such that
\[\|\lb \Phi(\cdot,\omega) Q_N(\cdot, \omega)^*\Phi(\cdot,\omega)^* x^*, y^*\rb\|_{L^1(0,T,[[N]](\omega))} \leq C_T(\omega) \|x^*\|\, \|y^*\|.\]
Also note that $[Mx^*]_t = \int_0^t\langle\Phi(s) { Q_N} \Phi^*(s)x^*,x^*\rangle \ud [[N]]$ for all $x^*\in X^*$.
Now if $x_n^*\to x^*$ as $n\to \infty$, it follows from the above estimate and identity that $[M x^*_n]_T\to [M x^*]_T$ on $\Omega_0$, and hence by the remarks in Subsection \ref{sec:defcyllocalmart} also $Mx^*_n\to Mx^*$ uniformly on $[0,T]$ in probability. Since $T>0$ was arbitrary, we find that $M$ is a cylindrical continuous local martingale.

To prove the equivalence it suffices to observe that
\begin{equation}\label{qvint}
\begin{split}
[[M]]_t &= \lim_{{\rm mesh}\to 0}\sum_{j=1}^J \sup_{x^*\in X^*, \|x^*\|=1}( [Nx^*]_{t_j}-[Nx^*]_{t_{j-1}})\\
&= \lim_{{\rm mesh}\to 0}\sum_{j=1}^J \sup_{x^*\in X^*, \|x^*\|=1}
\int_{t_{j-1}}^{t_j}\langle\Phi(s) Q_N(s) \Phi^*(s)x^*,x^*\rangle\ud [[N]]_s
\\ &  = \int_0^t \|\Phi(s) Q_N({s}) \Phi^*(s)\| \ud [[N]]_s,
\end{split}
\end{equation}
where the last equality holds true thanks to Lemma \ref{lem:maxmeasuref}, Remark \ref{maxbormes} and the separability of $X^*$. At the same time this proves the required formula for $[[M]]_t$. In order to find $A_M$ it suffices to note that for all $x^*,y^* \in X^*$:
\begin{align*}
\langle A_M(t) x^*,y^*\rangle_H & = [Mx^*,My^*]_t  = \int_0^t \langle  Q_N(s) \Phi^*(s)x^*,\Phi^*(s)y^*\rangle \ud [[N]]_s
\\ & =\int_0^t \langle \Phi(s) Q_N(s) \Phi^*(s)x^*,y^*\rangle \ud [[N]]_s.
\end{align*}
Since $\ud [[M]]_s = \|\Phi(s) Q_N(s) \Phi^*(s)\|\ud [[N]]_s$ the required identity for $Q_M$ follows from Proposition \ref{Q_M}.
\end{proof}

Next we present an example of a situation where $\widetilde M$ is a continuous martingale which associated cylindrical continuous local martingale $M$ is not in $M_{\rm var}^{\rm loc}(X)$.
\begin{example}\label{ex:countexbanval}
Let $X = \ell^p$ with $p\in (2, \infty)$ and let $W$ be a one-dimensional Brownian motion. It follows from \cite[Example 3.4]{RS} that
there exists a continuous martingale $\wt M:\mathbb R_+\times\Omega\to X$ such that
\[\lb \wt M_t, x^*\rb = \int_0^t \lb \phi(s), x^*\rb \, \ud W(s),\]
where $\phi:\mathbb R_+\to X$ is such that $\lb \phi, x^*\rb\in L^2(\mathbb R_+)$ for all $x^*\in X^*$,
but on the other hand $\|\phi\|_{L^2(0,1;X)}= \infty$.
Therefore, by Proposition \ref{prop:cylmartBM} the associated cylindrical martingale satisfies $[[M]]_1 = \infty$ a.s.,
and hence $M\notin \mathcal M_{\rm var}^{\rm loc}(X)$.

The same construction can be done for any Banach space $X$ which does not have cotype $2$ (see \cite[Proposition 6.2]{RS} and \cite[Theorem 11.6]{Ngamma}).
\end{example}

In the next example we construct a cylindrical continuous martingale in a Hilbert space which is not in $\mathcal M_{\rm var}^{\rm loc}(H)$.
\begin{example}\label{countex}
Let $H$ be a separable Hilbert space with an orthonormal basis $(h_n)_{n \geq 1}$ and $W$ be an one-dimensional Brownian motion.
Let $[0,1] = \cup_{n=1}^{\infty} A_n$ be a partition of $[0,1]$ into pairwise disjoint sets. Let $\psi:\mathbb R_+\times \Omega \to H$ be a deterministic function such that $\psi(t) = \sum_{n=1}^{\infty}|A_n|^{-1/2}\mathbf 1_{A_n}(t)h_n$. For each $h\in H$ one has that
\begin{equation}\label{eq:countex}
 \int_{\mathbb R_+} \langle\psi(s), h\rangle^2\ud s =
\int_{\mathbb R_+} \sum_{n=1}^{\infty}|A_n|^{-1}\mathbf 1_{A_n}(t)\langle h_n, h\rangle^2\ud s= \sum_{n=1}^{\infty} \langle h_n,h\rangle^2 = \|h\|^2,
\end{equation}
therefore $\langle\psi, h\rangle$ is stochastically integrable with respect to $W$ and one can define $M:H\to \mathcal M^{\rm loc}$ by $Mh = \langle\psi, h\rangle\cdot W$. Obviously $M$ is linear. Moreover, $Mh$ is an $L^2$-martingale for each $h\in H$ and thanks to \eqref{eq:countex} and the It\^o isometry, $\|(Mh)_{\infty}\|_{L^2(\Omega)} = \Big\|\int_0^1 \lb \psi, h\rb \ud W\Big\|_{L^2(\Omega)} = \|h\|$. So $Mh \to 0$ as $h \to 0$ in the ucp topology by Remark \ref{rem:Kal176}, hence $M$ is a cylindrical continuous local martingale. On the other hand due to \eqref{eq:quadvarint} one concludes that
\[
[[M]]_1 = \int_{0}^1 \|\psi(s)\|^2\ud s
= \int_{0}^1\sum_{n=1}^{\infty}|A_n|^{-1}\mathbf 1_{A_n}(s) \|h_n\|^2\ud s
=\sum_{n=1}^{\infty}\|h_n\|^2 = \infty.
\]
Consequently, $M\notin \mathcal M_{\rm var}^{\rm loc}(H)$.
\end{example}

\subsection{Quadratic Dol\'eans measure}

Recall from Definition \ref{def:cylDoleansmeasure} that $\mu_M$ is the cylindrical Dol\'eans measure associated with $M$. Since it only depends on $[[M]]$ sometimes the information get lost. In the next definition we define a bilinear-valued measure associated to $M$ (see \cite[Section 15.3]{MP}).

\begin{definition}\label{def:quaddoleansmeasure}
Let $M$ be a cylindrical continuous martingale such that $M(t)x^*\in L^2(\Omega)$ for all $t\geq 0$. Define the  \textit{quadratic Dol\'eans measure} $\bar{\mu}_M:\mathcal P \to \mathbf B( X^{*}, X^{*}) $ by
\[
\langle\bar{\mu}_M(F \times (s,t]),x^* \otimes y^*\rangle = \mathbb E[\mathbf 1_{F} ([Mx^*, My^*]_t - [Mx^*,My^*]_s)]
\]
for every predictable rectangle $F \times (s,t]$ and for every $x^*,y^* \in X^*$.
\end{definition}
A disadvantage of the quadratic Dol\'eans measure is that it can only be considered if $\lb M, x^*\rb_t\in L^2(\Omega)$. Such a problem does not occur for $\mu_{[[M]]}$, $A_M$ and $Q_M$ as in Proposition \ref{Q_M}.

Note that $\bar{\mu}_M$ defines a vector measure with variation (see \cite{DU,VTC}) given by
\begin{equation}\label{eq:varvectormeasure}
 |\bar{\mu}_M|(A) = \sup\sum_{n=1}^N\|\bar{\mu}_M(A_n)\|,
\end{equation}
where the supremum is taken over all the partitions $A=\bigcup_{n=1}^N A_n$. If $|\bar{\mu}_M|([0,\infty)\times\Omega)<\infty$, then it is a standard fact that the variation $|\bar{\mu}_M|$ defines a measure again and $|\bar{\mu}_M| \ll \mu_M$ (see \cite{DU}). Under the assumption that $\bar{\mu}_M$ has bounded variation a stochastic integration theory was developed in \cite[Chapter 16]{MP}.
The next result connects the measure $\mu_M$ from Definition \ref{def:cylDoleansmeasure} the operator $Q_M$ from Proposition \ref{Q_M} and the above vector measure $\bar{\mu}_{M}$. It provides a bridge between the theory in \cite[Chapter 16]{MP} and our setting.
\begin{proposition}\label{prop:dolquadvar}
Assume $M$ is a cylindrical continuous martingale such that $\lb M, x^*\rb_t\in L^2(\Omega)$ for all $t\geq 0$. Then
the following assertions are equivalent
\begin{enumerate}
\item $M\in \mathcal M_{\rm var}^{\rm loc}(X)$ and $\mu_M([0,\infty)\times\Omega)<\infty$
\item $\bar{\mu}_M$ has bounded variation.
\end{enumerate}
In that case  $d\bar{\mu}_M = Q_M d \mu_M$ in a weak sense, namely
\begin{equation}\label{dolquadvar}
\langle \bar{\mu}_M(A), x^*\otimes y^*\rangle = \int_{A}\langle Q_Mx^*,y^* \rangle \ud\mu_M,\;\;\;
x^*,y^* \in X^*, A \in \mathcal P.
\end{equation}
Moreover, $|\bar{\mu}_M|=\mu_M$.
\end{proposition}

The identity \eqref{dolquadvar} coincides with \cite[(16.1.1)]{MP}. To prove the above result we will need a technical lemma.
Let $f:\mathbb R_+\times\Omega\to [0,\infty]$ be an a.s.\ continuous increasing predictable process.
With slight abuse of terminology we say that the Dol\'eans measure of $f$ exists if
$C \mapsto \mathbb E \int_0^{\infty} \mathbf 1_C \ud f$ defines a finite measure on $\mathcal{P}$.

\begin{lemma}\label{lemma:technicalquaddol}
 Let $(f^n)_{n\geq 1}$ be a sequence of continuous predictable increasing processes on $\mathbb R_+$. Suppose that for all $n\geq 1$ the corresponding Dol\'eans measure $\mu_{n}$ of $f^n$ exists. Assume also that $\mu = \sup_{n\geq 1}\mu_n$ is of bounded variation. Then $F:\mathbb R_+\times\Omega \to \mathbb R_+$ defined by
 \begin{equation}\label{eq:maxquaddolmes3}
 F(t) = \lim_{\text{mesh}\to 0}\sum_{k=1}^{K} \sup_{n}(f^n(t_{k})-f^n(t_{k-1})),
 \end{equation}
 where the limit is taken over all partitions $0=t_0<\ldots<t_K = t$,
 is a predictable continuous increasing process and its Dol\'eans measure exists and equals $\mu$.
\end{lemma}

\begin{proof}
 For each $N \geq 1$ define $F^N:\mathbb R_+\times\Omega \to \mathbb R_+$ as
  \[
 F^N(t) = \lim_{\text{mesh}\to 0}\sum_{k=1}^{K} \sup_{1\leq n\leq N}(f^n(t_{k})-f^n(t_{k-1})),\;\;\; t\geq 0,
 \]
 where the limit is taken over all partitions $0=t_0<\ldots<t_K = t$. Then $F^N$ is a predictable process by Remark \ref{maxbormes}. Moreover, it is continuous since the corresponding Lebesgue-Stieltjes measure is nonatomic by \eqref{eq:supofmeas}. Let us consider the corresponding Dol\'eans measure $\nu_N$ of $F^N$. We claim that
 \begin{equation}\label{eq:claimtechnicalsup}
 \nu_N =\sup_{1\leq n \leq N}\mu_n.
 \end{equation}
 Since $\nu_N\geq \mu_n$ for each given $n\leq N$, we have $\nu_N \geq \sup_{1\leq n \leq N}\mu_n$.
Also notice that $\nu_N \leq \sum_{1\leq n\leq N}\mu_n$.

It remains to show $``\leq"$ in \eqref{eq:claimtechnicalsup}. First of all by Remark \ref{maxbormes} a.s.\ $\mu_{F^N}(\omega) = \sup_{1\leq n\leq N} \mu_{f^n}(\omega)$. By Lemma \ref{lem:maxmeasuref} a.s.\ the maximum of the Radon-Nikodym derivatives satisfies $\max_{1\leq n\leq N} \frac {\ud\mu_{f^n}}{\ud\mu_{F^N}}(t) = 1$ for $\mu_{F^N}$-a.a.\ $t \in \mathbb R_+$. So by Lemma \ref{lem:Lebesguediff} a.s.\ for $\mu_{F^N}$-a.a. $t> 0$
 \begin{equation}\label{eq:maxquaddolmes}
   1=\max_{1\leq n\leq N} \frac {\ud\mu_{f^n}}{\ud\mu_{F^N}}(t)
 = \max_{1\leq n\leq N} \lim_{\eps \to 0} \frac{f^n(t) - f^n(t-\eps\wedge t)}{F^N(t)-F^N(t-\eps\wedge t)}.
 \end{equation}
 Notice, that for each $n\leq N$ the processes $t\mapsto f^n(t) - f^n(t-\eps\wedge t)$ and $t\mapsto F^N(t)-F^N(t-\eps\wedge t)$ are predictable and continuous. Therefore, the sets
 \[
 A_n := \{(t,\omega)\in \mathbb R_+\times \Omega: \lim_{\eps \to 0}\frac {f^n(t) - f^n(t-\eps\wedge t)}{F^N(t)-F^N(t-\eps\wedge t)} = 1\},\;\;\; 1\leq n\leq N,
 \]
 are in the predictable $\sigma$-algebra $\mathcal{P}$. Redefine these sets to make them disjoint: $A_n := A_n \setminus (\cup_{1\leq k < n} A_k)$. Then by \eqref{eq:maxquaddolmes} for each predictable rectangle $B\in \mathcal P$ we have that $\nu_N(A_n\cap B) = \mu_n(A_n\cap B)$. Clearly this extends to all $B\in \mathcal{P}$.
Now it follows that for all $B \in\mathcal P$
 \[
 {\nu}_N (B) = \sum_{1\leq n \leq N}\nu_{N}(B\cap A_n)=\sum_{1\leq n \leq N}\mu_{n}(B\cap A_n)
 \leq (\sup_{1\leq n \leq N}\mu_{n}) (B),
 \]
and hence \eqref{eq:claimtechnicalsup} holds. Letting $N\to\infty$ in \eqref{eq:claimtechnicalsup} by Lemma \ref{lemma:limmaxmes} we obtain
 \begin{equation}\label{eq:maxquaddolmes2}
\lim_{N\to \infty} \nu_N(A) = \lim_{N\to \infty}(\sup_{1\leq n\leq N}\mu_{n})(A) = \mu(A), \ \ \ A \in \mathcal P.
\end{equation}

By Lemma \ref{lemma:limmaxmes}, pointwise on $\mathbb R_+\times\Omega$, $F^N\to F$, where $F$ is as in \eqref{eq:maxquaddolmes3}. Notice that $\mathbb E F^N(t) = \nu^N(\Omega\times [0,t])\nearrow\mu(\Omega\times [0,t])< \infty$, and since $F^N(t) \nearrow F(t)$ we have that $\mu(\Omega\times [0,t]) = \mathbb E F(t)$, so $F(t)$ finite a.s. Moreover, $F$ is predictable as it is the pointwise limit of the predictable processes $F^N$.
By the monotone convergence theorem and~\eqref{eq:maxquaddolmes2} we find that for all $0\leq s<t$ and $A\in \mathcal F_s$,
\[\mathbb E \one_A (F(t) - F(s)) = \lim_{N\to \infty}\mathbb E \one_A(F^N(t) - F^{N}(s)) = \lim_{N\to \infty}\nu_N((s,t]\times A) = \mu((s,t]\times A),\]
which completes the proof.
\end{proof}

\begin{proof}[Proof of Proposition \ref{prop:dolquadvar}]
(1)$\Rightarrow$(2): Assume (1). Let $x^*,y^* \in X^*$. Then for $A = (a,b]\times F$ with $b>a\geq 0$ and $F \in \mathcal F_a$, it follows from Proposition \ref{Q_M} that
\begin{align*}
\langle \bar{\mu}_M(A), x^*\otimes y^*\rangle &= \mathbb E 1_{F} ([Mx^*,My^*]_b - [Mx^*,My^*]_a)\\
&=\int_{F}\int_a^b d\langle A_M(s)x^*,y^* \rangle\ud\mathbb P
\\ &= \int_{F}\int_a^b \langle Q_Mx^*,y^* \rangle \ud [[M]]_s\ud\mathbb P =
\int_{A}\langle Q_Mx^*,y^* \rangle\ud\mu_M.
\end{align*}
As in \cite[Chapter 16.1]{MP} this extends to each $A \in \mathcal P$. This proves \eqref{dolquadvar} and since $\|Q_M\|=1$ $\mu_M$-a.e.\ it follows that
\[|\bar{\mu}_M|([0,\infty)\times \Omega) \leq \int_{[0,\infty)\times\Omega}\ud\mu_M = \mu_M([0,\infty)\times\Omega)<\infty.\]

(2)$\Rightarrow$(1): Assume (2). Let $(x_n^*)_{n\geq 1}$ be such that its $\mathbb Q$-linear span $E$ is dense in $X^*$ and $(x_1^*, \ldots, x_n^*)$ are linear independent for any $n\geq 1$. By a standard argument one can construct a $\mathbb Q$-bilinear mapping $\vv_M:\Omega{\times} [0,\infty)\to \mathbf B_{\mathbb Q}(E,E)$ such for all $x^*, y^*\in E$ and all $t\geq 0$, a.s.\ $\vv_M(t,\omega)(x^*, y^*) = [\lb M, x^*\rb, \lb M, y^*\rb]_t$.

Let $(y_n^*)_{n\geq 1}\subseteq X^*$ be equal to the intersection of $E$ and the unit ball in $X^*$. Then by Definition \ref{def:quaddoleansmeasure} and \eqref{eq:varvectormeasure} $|{\bar{\mu}_M}| =\sup_{n}\mu_{My_n^*}$, where $\mu_{Mx^*}$ is the Dol\'eans measure of $Mx^*$ for a given $x^* \in X^*$.
Now by Lemma \ref{lemma:technicalquaddol} one derives that there exists a predictable continuous increasing process ${F}:\mathbb R_+\times \Omega \to \mathbb R$ such that a.s.\
\begin{equation*}
 F(t) = \lim_{\text{mesh}\to 0}\sum_{k=1}^{K} \sup_{n}(a_M(t_k)(y_n^*,y_n^*)-a_M(t_{k-1})(y_n^*,y_n^*)),
 \end{equation*}
 where the limit is taken over all partitions $0=t_0<\ldots<t_K = t$. In particular, $a_M(t) (y_n^*, y_n^*)\leq F(t)$ a.s.\ and hence as in the first part of the proof of Theorem \ref{thm:aMdef} one sees that $a_M(t)$ extends to a bounded bilinear form on $X^*\times X^*$ a.s.\ and thanks to Remark \ref{rem:Kal176} and the fact that $M$ is a cylindrical continuous local martingale one obtains that for each $x^*,y^*\in X^*$, $\vv_M(x^*,y^*)$ and $[Mx^*,My^*]$ are indistinguishable. Then
 \begin{equation*}
 F(t) = \lim_{\text{mesh}\to 0}\sum_{k=1}^{K} \sup_{x^*\in X^*,\|x^*\|=1}(a_M(t_k)(x^*,x^*)-a_M(t_{k-1})(x^*,x^*)),
 \end{equation*}
and thanks to Theorem \ref{thm:aMdef} we conclude that the quadratic variation of $M$ exists.

The final identity $|\bar{\mu}_M|=\mu_M$ follows from Lemma \ref{lem:maxmeasuref}, \eqref{dolquadvar} and the fact that
\[\sup_{\|x^*\|=\|y^*\|=1} \lb Q_M x^*, y^*\rb = \|Q_M\| = 1.\]
which was proved in Proposition \ref{Q_M}.
\end{proof}

\subsection{Covariation operators}

In this subsection we assume that both $X$ and $Y$ have a separable dual space. In this section we introduce a covariation operator for $M_1\in \mathcal M_{\rm var}^{\rm loc}(X),M_2\in \mathcal M_{\rm var}^{\rm loc}(Y)$ and develop some calculus results for them.

\begin{prop}
Let $M_1\in \mathcal M_{\rm var}^{\rm loc}(X),M_2\in \mathcal M_{\rm var}^{\rm loc}(Y)$ be defined on the
same probability space. Then there exists a  \textit{covariation operator}
$A_{M_1,M_2}:\mathbb R_+ \times \Omega \to \mathcal L(X^*,Y^{**})$ such that for each $x^*\in X^*$, $y^*\in Y^*$ a.s.
\begin{equation*}
 \langle A_{M_1,M_2}(t)x^*,y^* \rangle = [M_1x^*,M_2y^*]_t,\;\;\;t\geq 0.
\end{equation*}
\end{prop}

\begin{proof}
 Let $a_{M_1,M_2}:\mathbb R_+ \times \Omega \to \mathbf B(X^*,Y^*)$ be defined as a version of $(t,\omega)(x^*,y^*) \mapsto [M_1(\omega)x^*,M_2(\omega)y^*]_t$ such that a.s.\ for each $t\in \mathbb R_+$
 \begin{equation}\label{proplinaerity}
  \begin{split}
 |a_{M_1,M_2}(t)(x^*,y^*)|&\leq \sqrt{a_{M_1}(t)(x^*,x^*)a_{M_2}(t)(y^*,y^*)}\\
 &\leq
 \sqrt{[[M_1]]_t[[M_2]]_t}\|x^*\|\|y^*\|\;\;\;\forall x^*\in X^*, y^*\in Y^*,
  \end{split}
 \end{equation}
To construct such a version we can argue as in the first part of the proof of Theorem \ref{thm:aMdef}.
\end{proof}

\begin{proposition}\label{prop:linearity}
The space $\mathcal M_{\rm var}^{\rm loc}(X)$ is a vector space and equipped with the (metric) topology of ucp convergence of the quadratic variation $[[\cdot]]$ it becomes a complete metric space
with the translation invariant metric given by
\[
\|M\|_{\mathcal M_{\rm var}^{\rm loc}(X)} := \sum_{n=1}^{\infty} 2^{-n}\mathbb E[1 \wedge  [[M]]^{1/2}_n] + \sup_{\|x^*\|\leq 1}\mathbb E[1 \wedge |(Mx^*)_0|].
\]
Moreover, for $M_1,M_2\in\mathcal M_{\rm var}^{\rm loc}(X)$ a.s.\ for all $t\geq 0$ the triangle inequality holds:
\begin{equation}\label{eq:triangeineqquadvar}
 [[M_1+M_2]]_t^{\frac 12}\leq [[M_1]]_t^{\frac 12}+[[M_2]]_t^{\frac 12}.
\end{equation}
\end{proposition}

The above metric does not necessarily turn $\mathcal M_{\rm var}^{\rm loc}(X)$ into a topological vector space in the case $X$ is infinite dimensional. However, if the martingales are assumed to start at zero then it becomes a topological vector space.

\begin{proof}
Now for $M_1, M_2 \in \mathcal M_{\rm var}^{\rm loc}(X)$ one can easily prove, that $M_1+M_2 \in \mathcal M_{\rm var}^{\rm loc}{(X)}$. Indeed, by the definition of the quadratic (co)variation operator and linearity for all $x^*,y^*\in X^*, t\geq 0$, a.s.
\begin{multline*}
 [(M_1 + M_2) x^*, (M_1+M_2)y^*]_t = \langle (A_{M_1}(t) + A_{M_1,M_2}(t) + A_{M_2,M_1}(t) + A_{M_2}(t))x^*,y^*\rangle,
\end{multline*}
and so by \eqref{proplinaerity} and Definition \ref{def:cylmart} $[[M_1+M_2]]$ exists and a.s.\
\begin{equation*}
 [[M_1+M_2]]_t \leq [[M_1]]_t + [[M_2]]_t + 2\sqrt{[[M_1]]_t[[M_2]]_t},\;\;\; t\geq 0,
\end{equation*}
which proves \eqref{eq:triangeineqquadvar}. Since it is clear that $\mathcal M_{\rm var}^{\rm loc}(X)$ is closed under multiplication by scalars, it follows that $\mathcal M_{\rm var}^{\rm loc}(X)$ is a vector space.

To prove the completeness let $(M^n)_{n\geq 1}\subseteq \mathcal M_{\rm var}^{\rm loc}(X)$ be a Cauchy sequence, then $(M^nx^*)_{n\geq 1}$ is a Cauchy sequence in $\mathcal M^{\rm loc}$ for all $x^* \in X^*$, and so by Remark \ref{rem:Kal176} and completeness it converges to a continuous local martingale $Mx^*$ in the ucp topology. Let $(x_m^*)_{m=1}^{\infty}\subset X^*$ be a dense subset of $X^*$. Then due to a diagonalization argument there exists a subsequence $(n_k)_{k\geq 1}$ such that $[M^{n_k}x_m^*]_t$ converges a.s.\ for any $t\geq 0$ and $m\geq 1$, and $[[M^{n_k}]]_t$ has an a.s.\ limit for all $t\geq 0$ (recall that due to \eqref{eq:triangeineqquadvar}, $[[\cdot]]_t^{1/2}$ obeys a triangle inequality for each $t\geq 0$). Then a.s. for all $t\geq s\geq 0, m\geq 1,$
\[
[Mx_m^*]_t-[Mx_m^*]_s= \lim_{k\to \infty} ([M^{n_k}x_m^*]_t-[M^{n_k}x_m^*]_s)\leq \varliminf_{k\to \infty} ([[M^{n_k}]]_t-[[M^{n_k}]]_s) \|x_m^*\|^2.
\]
By Proposition \ref{prop:[[M]]andF} we find $M \in \mathcal M_{\rm var}^{\rm loc}(X)$ and $[[M]] \leq \lim_{n\to \infty} [[M^n]]$, where the last limit is taken in the ucp topology. Now fix $t>0$. To prove that a.s. \ $\lim_{k\to \infty}[[M-M^{n_k}]]_t=0$
one has firstly to consider a sequence $(c_k^D)_{k,D = 1}^{\infty}$, such that for all $k,D>0$
\begin{equation*}
c_k^D = \biggl(\lim_{\text{mesh}\to 0}\sum_{l=1}^{L} \sup_{1\leq d\leq D} [(M^{n_k}-M)x_{d}^*]_{t_l}-[(M^{n_k}-M)x_{d}^*]_{t_{l-1}}\biggr)^{\frac 12},
 \end{equation*}
 where the limit is taken over all partitions $0=t_0<\ldots<t_L = t$. Then by Lemma~\ref{lemma:limmaxmes} a.s.\ $c_k^D\to[[M^{n_k}-M]]^{\frac 12}_t$ as $D \to \infty$, and consequently $c_k :=(c_k^D)_{D=1}^{\infty}\in \ell^{\infty}$ for all $k\geq 1$, where $\ell^{\infty}$ is the space of bounded sequences. Then obviously by \eqref{eq:triangeineqquadvar} a.s.
 \[
 \sup_{D\geq 1}|c_k^D - c_l^D|\leq [[M^k-M^l]]_t^{\frac 12},\;\;\;k,l\geq 1,
 \]
 which yields that $(c_k)_{k=1}^{\infty}$ is a Cauchy sequence in $\ell^{\infty}$. Now one can easily show that $c_k^{D} \to 0$ as $k\to \infty$, so $c_k \to 0$, and a.s.\ $\sup_D (c_k^D)^2 = [[M-M^{n_k}]]_t\to 0$.
\end{proof}

As a positive definite bilinear form the covariation operator has the following properties a.s.\
$\forall t>s\geq 0, x^*\in X^*$:
\[
A_{M_1,M_2}(t,\omega)=\frac{A_{M_1+M_2}(t,\omega)-A_{M_1-M_2}(t,\omega)}{4},
\]
\begin{multline}\label{cov}
\langle (A_{M_1,M_2}(t) - A_{M_1,M_2}(s))x^*,x^*\rangle\\ \leq \sqrt{\langle(A_{M_1}(t) - A_{M_1}(s))x^*,x^*\rangle\langle(A_{M_2}(t) - A_{M_2}(s))x^*,x^*\rangle}.
\end{multline}

\begin{remark}
One can also define a  \textit{covariation process} $[[M_1,M_2]]$ by the formula
\[
[[M_1,M_2]]_t:= \lim_{{\rm mesh}\to 0}\sum_{n=1}^N \|A_{M_1,M_2}({t_n})-A_{M_1,M_2}({t_{n-1}})\|.
\]
 The limit exists a.s.\ thanks to the Cauchy-Schwartz inequality and the fact that a.s.\ for each $0\leq s<t$
 \[
 \|A_{M_1,M_2}({t})-A_{M_1,M_2}({s}) \|\leq \sqrt {\|A_{M_1}({t})-A_{M_1}({s})\|}\sqrt {\|A_{M_2}({t})-A_{M_2}({s})\|},
 \]
 where the last is an easy consequence of \eqref{cov}.

 The process $[[M_1,M_2]]$
 is continuous a.s.\ and has some properties of a covariation process of real-valued martingales.
 For instance, one can prove by the formula \eqref{cov} that for all $t>s\geq 0$
\begin{equation}\label{triangle}
|[[M_1,M_2]]_t - [[M_1,M_2]]_s| \leq \sqrt{([[M_1]]_t-[[M_1]]_s)([[M_2]]_t -[[M_2]]_s)} \;\;\; a.s.
\end{equation}
Unfortunately, in general $[[\cdot]]_t$ is not a quadratic form (except in the one-dimensional case).

\end{remark}
Thanks to the continuity of covariation process one can consider the Lebesgue-Stieltjes measure
$\mu_{[[M_1,M_2]]}$ for a.a. $\omega$.
By the same technique as it was mentioned before one can also construct $Q_{M_1,M_2}:\mathbb R_+\times \Omega \to \mathcal L(X^*,Y^{**})$:
\[
\langle A_{M_1,M_2}(t) x^*,y^*\rangle=\int_0^t\langle Q_{M_1,M_2}(s) x^*,y^*\rangle\ud [[M_1,M_2]],\;\;\; t\geq 0,
\omega \in \Omega.
\]

Note, that $\|Q_{M_1,M_2}(t)\|\leq 1$ a.s.\ and for $\mu_{[[M_1,M_2]]}$-a.a. $t>0$ by the same argument, as in
Proposition \ref{Q_M}. Also evidently
$Q_{M_1,M_2} = Q_{M_2,M_1}^*$.
One can derive the following result:
\begin{proposition}[Kunita-Watanabe inequality, cylindrical case] Let
$M_1\in \mathcal M_{\rm var}^{\rm loc}(X)$, $M_2\in \mathcal M_{\rm var}^{\rm loc}(Y)$ defined on the
same probability space $(\Omega, \mathbb F, \mathbb P)$, $f\colon \mathbb R_+ \times \Omega \to X^*$, $g \colon \mathbb R_+ \times \Omega \to Y^*$
be two strongly $\mathcal B_{\mathbb R_+} \otimes \mathcal F$-measurable
bounded functions. Then for all $t>0$ and for almost all $\omega\in \Omega$
\begin{multline*}
\left|\int_0^t \langle Q_{M_1,M_2}(s)f(s,\omega),g(s,\omega)\rangle\ud [[M_1,M_2]]_s(\omega)\right|^2\\
\leq\int_0^t \langle Q_{M_1}(s)f(s,\omega),f(s,\omega)\rangle \ud [[M_1]]_s(\omega)
\int_0^t \langle Q_{M_2}g(s,\omega),g(s,\omega)\rangle\ud [[M_2]]_s(\omega).
\end{multline*}
\end{proposition}
The proof is analogous to the proof of \cite[Theorem~II.25]{Prot}, for which one has to apply
inequalities of the form \eqref{cov}.

Recall from \eqref{intcyl} that for suitable $\Phi$ and $M\in M_{\rm var}^{\rm loc}(H)$, $(\Phi \cdot M) \in M_{\rm var}^{\rm loc}(X)$ given by
\[(\Phi \cdot M) x^*:= \int_0^{\cdot}\Phi^* x^* \, d M,\;\;\;x^*\in X^*\]
is well-defined.

\begin{theorem}
Let $H$ be a separable Hilbert space, $M_1\in \mathcal M_{\rm var}^{\rm loc}(H)$, $M_2\in \mathcal M_{\rm var}^{\rm loc}(Y)$,
$\Phi:\mathbb R_+ \times\Omega \to \mathcal L(H,X)$ be such that
$\Phi^* x^*$ is a strongly progressively measurable process for each $x^* \in X^*$
and let $\|\Phi Q_{M_1} \Phi^*\|\in L^1_{\rm loc}(\mathbb R_+,[[M_1]])$ a.s.
Then for all $t \geq 0$ and for all $x^*\in X^*$, $y^* \in Y^*$ one has
\[
\langle A_{\Phi \cdot M_1,M_2}(t) x^*,y^*\rangle
= \int_0^t\langle Q_{M_1,M_2} \Phi^* x^* , y^*\rangle\ud [[M_1,M_2]]\;\;\; a.s.
\]

\end{theorem}

\begin{proof}
 Fix $t\geq 0$ and $x^* \in X^*$, $y \in Y^*$. Put $\phi = \Phi^* x^*$.
 Firstly suppose that
 there exists $n>0$ such that $\phi$ takes its values in a finite-dimensional
 subspace $\text{span}(h_1,\ldots,h_n) \subseteq H$.
 Then by bilinearity of covariation process, the definition of $Q_{M_1,M_2}$,
 and thanks to \cite[Theorem~17.11]{Kal}
 \begin{align*}
  \langle A_{\Phi \cdot M_1,M_2}(t) x^*,y^*\rangle &= \Bigl[\int_0^{\cdot}\phi \ud M_1, M_2y^* \Bigr]_t
  =\sum_{i=1}^n \Bigl[\int_0^{\cdot} \langle \phi,h_i\rangle \ud (M_1 h_i), M_2 y^* \Bigr]_t\\
  &=\sum_{i=1}^n \int_0^{t} \langle \phi,h_i\rangle\ud [ M_1 h_i, M_2 y^* ]_t
\\ &  =\sum_{i=1}^n \int_0^{t} \langle \phi,h_i\rangle\langle h_i, Q_{M_2,M_1}y^*\rangle\ud [[ M_1, M_2 ]]_t\\
  &= \int_0^{t} \langle \phi, Q_{M_2,M_1}y^*\rangle\ud [[ M_1, M_2 ]]_t
\\ &  = \int_0^{t} \langle Q_{M_1,M_2}\phi, y^*\rangle\ud [[ M_1, M_2 ]]_t.
  \end{align*}
In the general case one can approximate {$\phi$ by $P_n \phi$, where $P_n\in \mathcal L(H)$ is an orthogonal projection on $\text{span}(h_1,\ldots,h_n)$,} and derive the desired
by using {\eqref{eq:quadvarsimpleint} and} inequalities of the type \eqref{cov}-\eqref{triangle}.
\end{proof}

One can prove the full analogues of \cite[Lemma 17.10]{Kal} and \cite[Theorem~17.11]{Kal}
using the same methods as in the proof above:
\begin{theorem}[Covariation of integrals] Let $H$ be a separable Hilbert space,
$M_1, M_2\in \mathcal M_{\rm var}^{\rm loc}(H)$, $\Phi_1:\mathbb R_+ \times\Omega \to \mathcal L(H,X)$, $\Phi_2:\mathbb R_+ \times\Omega \to \mathcal L(H,Y)$ be such that
$\Phi_1^* x^*$, $\Phi_2^* y^*$ are strongly progressively measurable processes for each $x^* \in X^*$, $y^*\in Y^*$
and assume that for $j\in \{1, 2\}$, $\|\Phi_j Q_{M_j} \Phi_j^*\|\in L^1_{\rm loc}(\mathbb R_+,[[M_j]])$ a.s.\
Then for all $t \geq 0$ and for all $x^*\in X^*$, $y^*\in Y^*$ one has
\[
\langle A_{\Phi_1 \cdot M_1,\Phi_2\cdot M_2}(t) x^*,y^*\rangle
= \int_0^t\langle Q_{M_1,M_2} \Phi_1^* x^* , \Phi_2^*y^*\rangle\ud [[M_1,M_2]]\;\;\; a.s.
\]

\end{theorem}

\begin{remark} To construct the analogy one has to see due to the equation above that in a
weak sense
\[
 A_{\Phi_1 \cdot M_1,\Phi_2\cdot M_2}(t) = \int_0^t \Phi_2 Q_{M_1,M_2} \Phi_2^*\ud [[M_1,M_2]]_s=
\int_0^t \Phi_2 \ud A_{M_1,M_2}(s) \Phi_2^*,
\]
which extends the scalar case.
\end{remark}

\section{Stochastic integration with respect to cylindrical continuous local martingales}\label{sec:char}

Let $H$ be a separable Hilbert space and let $X$ be a separable Banach space with a separable dual space. In the previous section we have introduced stochastic integrals as cylindrical continuous local martingales. Often one wants the stochastic integral to be an actual local martingale instead of a cylindrical one. In this section we will characterize when this is the case we prove two-sided estimates
for the stochastic integral
\[(\Phi \cdot M)_t = \int_0^t \Phi(s) \ud M(s),\]
where $\Phi$ is an $\calL(H,X)$-valued $H$-strongly progressively measurable processes.
Here $M\in \mathcal M_{\rm var}^{\rm loc}(H)$ (see Definition \ref{def:cylmart}).

For this characterization we need the language of $\gamma$-radonifying operators and the geometric condition UMD on the Banach space $X$. Both will be introduced in the next two subsection.

\subsection{$\gamma$-radonifying operators\label{subsec:gamma}}
We refer to \cite{HNVW2}, \cite{Ngamma} and \cite{KW} and references therein for further details.
Let $(\gamma_n')_{n \geq 1}$ be a sequence of independent standard Gaussian random variables
on a probability space $(\Omega', {\mathcal F'}, \mathbb P')$
(we reserve the notation $(\Omega, {\mathcal F}, \mathbb P)$ for the probability
space on which our processes live) and let $H$ be a separable real Hilbert
space. A bounded operator $R \in\mathcal L (H, X)$ is said to be $\gamma$-radonifying
if for some (or equivalently for each) orthonormal basis $(h_n)_{n\geq 1}$ of $H$ the Gaussian series
$\sum_{n\geq 1} \gamma_n' Rh_n$
converges in $L^2 (\Omega'; X)$. We then define
\[
\|R\|_{\gamma(H,X)} :=\Bigl(\mathbb E'\Bigl\|\sum_{n\geq 1} \gamma_n' Rh_n \Bigr\|^2\Bigr)^{\frac 12}.
\]
This number does not depend on the sequence $(\gamma_n')_{n\geq 1}$ and the basis $(h_n)_{n\geq 1}$, and
defines a norm on the space $\gamma(H, X)$ of all $\gamma$-radonifying operators from $H$ into
$X$. Endowed with this norm, $\gamma(H, X)$ is a Banach space, which is separable if $X$
is separable.
For later reference we note that the convergence of $\sum_{n\geq 1} \gamma_n' Rh_n$ in $L^p(\Omega';X)$ with $p\in(0, \infty)$, in probability and a.s.\ can  all be shown to be equivalent.

If $R \in\gamma(H, X)$, then $\|R\| \leq \|R\|_{\gamma(H, X)}$. If $X$ is a Hilbert space, then
$\gamma(H, X) =\mathcal L^2 (H, X)$ isometrically.
Let $G$ be another Hilbert space, $X$ be another Banach space. Then by the so-called  \textit{ideal property}
(see \cite{HNVW2}) the following holds true: for all $S \in\mathcal L(G,H)$ and all $T \in \mathcal L(X,Y)$
we have $TRS \in \gamma(G,Y)$ and
\begin{equation}\label{ideal}
\|TRS\|_{\gamma(G,Y)}\leq \|T\| \|R\|_{\gamma(H,X)}\|S\|.
\end{equation}

Let $\mu$ be a measure on a Borel set $J\subseteq \mathbb R_+$ {with a $\sigma$-field $\mathcal A$} such that $L^2(J,\mu)$ is separable and $p\in [1,\infty)$. We say that a function $\Phi : J \to \calL(H,X)$ belongs to $L^p (J, \mu;H)$  \textit{scalarly} if for all $x^* \in X^*$,
$\Phi^* x^* \in L^p (J, \mu;H)$. A function
$\Phi : J\to \calL(H,X)$ is said to  \textit{represent} an operator $R \in \gamma(L^2(J, \mu;H), X)$ if $\Phi$ belongs to $L^2(J, \mu;H)$ scalarly and for all $x^* \in X^*$ and $f \in L^2 (J, \mu;H)$ we have
\[
\langle Rf, x^*\rangle = \int_J f(s) \Phi(s)^* x^* \ud \mu(s).
\]
The above notion will be abbreviated by $\Phi\in \gamma(J, \mu; H,X)$. In the case $X$ is a Hilbert space, one has $\gamma(J, \mu; H,X) = L^2(J,\mu;\calL_2(H,X))$ isometrically, where $\calL_2(H,X)$ denotes the Hilbert-Schmidt operators from $H$ to $X$.

If $\mu$ is the Lebesgue measure we will also write $\gamma(L^2(J;H), X)$ and $\gamma(J; H, X)$
for $\gamma(L^2 (J, \mu; H), X)$ and
$\gamma(J, \mu; H, X)$ respectively.

Let $\nu:{\mathcal A\times \Omega}\to [0,\infty]$ be a random measure. Typically, $\nu$ will be the Lebesgue-Stieltjes measure associated $[[M]]$ for $M\in M_{\rm var}^{\rm loc}(H)$. In this case we will also identify $[[M]]$ and $\nu$.
We say that $\Phi : J \times \Omega \to \calL(H,X)$  \textit{is scalarly} in
$L^2 (J,\nu;H)$ a.s.\ if for all $x^* \in X^*$,
for almost all $\omega \in \Omega$, $\Phi(\cdot, \omega)^*x^* \in L^2 (J, \nu(\cdot, \omega);H))$.

For such a process $\Phi$ and a family $R = (R(\omega) : \omega \in \Omega)$ with
$R(\omega) \in \gamma(L^2 (J, \nu(\cdot, \omega); X)$ for almost all
$\omega \in \Omega$, we say that $\Phi$  \textit{represents} $R$ if for all $x^* \in X^*$, for almost all
$\omega \in \Omega$,
$\Phi(\cdot, \omega)^* x^* = R^* (\omega)x^*$ in $L^2 (J, \nu(\cdot, \omega);H)$. As before this will be abbreviated by $\Phi\in \gamma(J,\nu;H,X)$ a.s.

In the case that $\nu$ is the Lebesgue measure the above notion of representability reduces to the one given in \cite{NVW}.

\subsection{The UMD property}
The results will be stated for the important class of UMD Banach spaces and we refer to \cite{Burk01}, \cite{HNVW1}, \cite{Rubio86} for details. A Banach space $X$ is called a  \textit{UMD space} if for some (or equivalently, for all)
$p \in (1,\infty)$ there exists a constant $\beta>0$ such that
for every $n \geq 1$, every martingale
difference sequence $(d_j)^n_{j=1}$ in $L^p(\Omega; X)$, and every $\{-1, 1\}$-valued sequence
$(\varepsilon_j)^n_{j=1}$
we have
\[
\Bigl(\mathbb E \Bigl\| \sum^n_{j=1} \varepsilon_j d_j\Bigr\|^p\Bigr )^{\frac 1p}
\leq \beta \Bigl(\mathbb E \Bigl \| \sum^n_{j=1}d_j\Bigr\|^p\Bigr )^{\frac 1p}.
\]
The infimum over all admissible constants $\beta$ is denoted by $\beta_{p,X}$.

UMD spaces are always reflexive. Examples of UMD space include, the reflexive range of $L^q$-spaces, Besov spaces, Sobolev spaces. Example of spaces without the UMD property include all nonreflexive spaces, e.g. $L^1(0,1)$ and $C([0,1])$.

\subsection{Characterization of stochastic integrability}

The next result is the main result of this section.
\begin{theorem}\label{main}
Let $X$ be a UMD space, $M \in \mathcal M_{\rm var}^{\rm loc}(H)$. For a strongly progressively measurable process
$\Phi \colon \mathbb R_+ \times \Omega\rightarrow \mathcal L(H,X)$
such that $\Phi\,Q_M^{1/2}$ is scalarly in $ L^2(\mathbb R_+, [[M]];H)$~a.s.\
the following assertions are equivalent:
\begin{enumerate}

 \item[(1)] There exists  elementary progressive processes $(\Phi_n)_{n \geq 1}$ such that:
 \begin{enumerate}
  \item[(i)] for all $x^* \in X^*$, $\displaystyle \limn Q_M^{1/2}\Phi_n^*x^* =
  Q_M^{1/2}\Phi^*x^*$ in $L^0(\Omega; L^2(\mathbb R_+,[[M]];H))$;
  \item[(ii)] there exists a process $\zeta \in L^0(\Omega; C_b(\mathbb R_+; X))$ such that
  \[
  \zeta = \lim_{n \to \infty} \int_0^{\cdot} \Phi_n(t)\ud M(t) \;\;\; \text{in}\; L^0(\Omega; C_b(\mathbb R_+;X)).
  \]
 \end{enumerate}
 \item[(2)] There exists an a.s.\ bounded process $\zeta\colon \mathbb R_+\times \Omega\to X$ such that for all $x^* \in X^*$
 we have
  \[
  \langle\zeta,x^*\rangle = \int_0^{\cdot}\Phi^*(t)x^* \ud M(t) \;\;\;
  \text{in}\: L^0(\Omega; C_b(\mathbb R_+)).
  \]

 \item[(3)] $\Phi\;Q_M^{1/2} \in \gamma(L^2(\mathbb R_+,[[M]];H), X)$ almost surely;
\end{enumerate}
In this case $\zeta$ in (1) and (2) coincide and for all $p \in (0, \infty)$ we have
\[
\mathbb E \sup_{t \in \mathbb R_+}\|\zeta (t)\|^p \eqsim_{p,X} \mathbb E \|\Phi\: Q_M^{1/2}\|^p_{\gamma (L^2(\mathbb R_+,[[M]];H), X)}.
\]
\end{theorem}

A process $\Phi:\mathbb R_+\times\Omega\to \calL(H,X)$ which satisfies the above conditions and the assertions (1)--(3) will be called  \emph{stochastically integrable with respect to $M$}.

\begin{remark}
The case of scalar-valued continuous local martingales of Theorem \ref{main} was considered in \cite{Ver}, where the Dambis, Dubins-Schwarz result is applied to write the continuous local martingale as a time changed Brownian motion. Unfortunately, in the vector-valued setting, this technique breaks down as one cannot do a different time change in infinitely many direction. The proof of Theorem \ref{main} will be given in Subsection \ref{subsec:proofmain} after we have introduced some techniques we will use.
\end{remark}

\subsection{Time transformations}
A nondecreasing, right-continuous family of stopping times $\tau_s:\Omega\to [0,\infty]$, $s \geq 0$, will be called
a  \textit{random time-change}.
If additionally $\tau_s:\Omega\to [0,\infty)$ then $\tau_s$, $s \geq 0$, will be called a  \textit{finite random time-change}. If $\mathbb F$ is right-continuous, then
according to \cite[Lemma~7.3]{Kal} the same holds true
for the  \textit{induced filtration} $\mathbb G = (\mathcal G_s)_{s \geq 0} = (\mathcal F_{\tau_s})_{s\geq 0}$
(see \cite[Chapter~7]{Kal}).
An $M\in\mathcal M_{\rm var}^{\rm loc}(X)$ is said to be  \textit{$\tau$-continuous} if for each $x^*\in X^*$, $Mx^*$ is an a.s. constant on every interval $[\tau_{s-}, \tau_s]$, $s \geq 0$, where we let $\tau_{0-} = 0$. Notice that if $M$ is $\tau$-continuous, then $[[M]]$ is $\tau$-continuous as well by \cite[Exercise 17.3]{Kal} and by using Proposition \ref{prop:[[M]]andF}. A vector-valued process $F$ is $\tau$-continuous if $F$ is an a.s. constant on every interval $[\tau_{s-}, \tau_s]$, $s \geq 0$.

\begin{proposition}[Kazamaki]\label{Kazamaki}
Let $\tau$ be a finite random time-change and let $M\in M_{\rm var}^{\rm loc}(H)$ with respect to~$\mathbb F$.
Let $X_0$ be a finite dimensional Banach space. Assume also that $M$ is $\tau$-continuous.
Let $\Phi: \mathbb R_+ \times \Omega \to \calL(H,X_0)$ be $\mathbb F$-progressively measurable and
assume
\[
\int_0^{\infty} \|\Phi Q_M^{1/2}\|^2_{\calL(H,X_0)} \ud [[M]] < \infty \textit{ a.s.}
\]
Define the process $\Psi:\mathbb R_+\times\Omega\to \calL(H,X_0)$ by $\Psi(s) =\Phi(\tau_s)$.
Then following assertions hold:
\begin{enumerate}
\item $N = M\circ \tau:H \to \mathcal{M}^{\rm loc}$ given by
\[
N h := (Mh)\circ\tau,\;\;\;h \in H.
\]
is in $\mathcal M_{\rm var}^{\rm loc}(H)$ with respect to~$\mathbb G$;
\item $[[N]]=[[M\circ \tau]] = [[M]]\circ \tau$~a.s.;
\item $Q_N = Q_M \circ \tau$;
\item $\Psi$ is $\mathbb G$-progressively measurable and
\begin{equation}\label{eq:subruleKaz}
 \int_0^\infty \|\Psi Q_N^{1/2}\|^2_{\calL(H,X_0)} \ud [[N]]< \infty \textit{ a.s.},
\end{equation}
\begin{equation}\label{Kaz}
(\Phi\circ \tau)\cdot (M \circ\tau) = (\Phi\cdot M)\circ \tau\textit{ a.s.}
\end{equation}
\end{enumerate}
\end{proposition}
Note that the stochastic integrals are well-defined by Remark \ref{rem3}.

\begin{proof}
(1):\ By \cite[Theorem 17.24]{Kal}
for each $h \in H$ the process $Nh = (Mh)\circ\tau$ is a continuous $\mathbb G$-local martingale
and $[M h \circ \tau] = [Mh]\circ \tau$.
Thus by Proposition \ref{prop:[[M]]andF} $M\circ \tau:H \to \mathcal{M}^{\rm loc}$ given by
\[
(M \circ \tau)h := (Mh)\circ\tau,\;\;\;h \in H.
\]
is in $\mathcal M_{\rm var}^{\rm loc}(H)$, since for any $h\in H$ one has that $\mu_{[Mh]\circ \tau}\leq \mu_{[[M]]\circ \tau}\|h\|^2$ a.s.\ (notice that thanks to $\tau$-continuity both $[Mh]\circ \tau$ and $[[M]]\circ \tau$ are a.s.\ continuous).

(2):\  Let $(x_m^*)_{m\geq 1}$ be a dense subset of the unit ball in $X^*$. Since $M$ is $\tau$-continuous, one has that a.s. $[[M]]$ and $[Mx_m^*]$ are $\tau$-continuous for each $m\geq 1$. Now by Proposition \ref{prop:[[M]]andF} we find that a.s.
\[
\mu_{[[N]]} = \sup_{m\geq 1}\mu_{[Nx_m^*]} = \sup_{m\geq 1}\mu_{[Mx_m^*]\circ\tau}=\mu_{[[M]]\circ\tau},
\]
and therefore, $[[M]]\circ \tau$ is a version of $[[N]]$.

(3): \ This follows from a substitution argument:
\[
\langle Q_N h_1, h_2\rangle = \frac{\ud(\langle A_M\circ \tau h_1, h_2\rangle)}{\ud([[M]]\circ\tau)}
= \frac{\ud\langle A_Mh_1,h_2\rangle}{\ud [[M]]}\circ \tau = \langle Q_Mh_1,h_2\rangle \circ \tau,
\;\;\; h_1,h_2 \in H.
\]

(4): \ The $\mathbb G$-progressive measurability of $\Psi$ can be proven in the same way as in the proof of \cite[Proposition 2]{Kaz}.  Assertion \eqref{eq:subruleKaz} can be obtained by (2), (3) and the general version of the substitution rule \eqref{rule}.

The existence of the left hand side of \eqref{Kaz} can be proved via \eqref{eq:subruleKaz} and Remark~\ref{rem3}.
The equation \eqref{Kaz} is obvious for elementary progressively measurable $\Phi$ and follows by an approximation argument as in Remark \ref{rem3}.
\end{proof}

We now prove a version of Proposition \ref{Kazamaki} for a special class of random time changes which are not necessarily finite.
\begin{corollary}\label{cor:inversekaz}
Let  $M\in \mathcal M_{\rm var}^{\rm loc}(H)$. Suppose that $(\tau_s)_{s\geq 0}$ has the following form:
\[\tau_s =
\begin{cases}
\inf\{t\geq 0:[[M]]_t>s\}, &\text{if}\;\; 0\leq s<S;\\
 \infty,&\text{otherwise},
\end{cases}
\]
where $S = \sup_{t\geq 0} [[M]]_t$. Then for each $h\in H$, $M_{\infty}h = \lim_{t\to \infty}M_th$ a.s.\ exists if $S<\infty$ and Proposition \ref{Kazamaki} holds true for $N := M\circ \tau$ defined as follows
\[N_s =
\begin{cases}
M_{\tau _s}, &\text{if}\;\; 0\leq s<S;\\
 M_{\infty},&\text{otherwise},
\end{cases}
\]
Moreover, if  $\Psi:\mathbb R_+\times\Omega\to\mathcal L(H,X_0)$ is stochastically integrable with respect to $N$, $\Phi:= \Psi\circ[[M]]$, then also $\Phi\circ\tau$ is stochastically integrable with respect to $N$ and a.s.\
\begin{equation}\label{eq:inversekaz}
 \int_0^{[[M]]_t} \Psi\ud N =\int_0^{[[M]]_t}\Phi\circ\tau\ud N = \int_0^t \Phi \ud M,\;\;t\geq 0.
\end{equation}
\end{corollary}

Recall the substitution rule: for a strongly measurable $f \colon \mathbb R_+ \to X$ we have
$f \in L^1(\mathbb R_+,\mu_{[[M]]};X)$ if and only if $f\circ \tau \in L^1(0,S;X)$, and in that case
\begin{equation}\label{rule}
\int_{\mathbb R_+} f(t)\ud [[M]] = \int_{[0,S)}f(\tau(s))\ud s.
\end{equation}

\begin{proof}[Proof of Corollary \ref{cor:inversekaz}]
According to \cite[Proposition IV.1.26]{RY} and the fact that
for each $h\in H$, $[[M]]\|h\|\geq[Mh]$ a.s., one can define $M_{\infty}h$ if $S<\infty$, so $N$ is well-defined. Now we prove that $N \in \mathcal M_{\rm var}^{\rm loc}(H)$.

Define $\tau_s^n:= \inf\{t\geq 0:[[M]]_t>s\}\wedge n$ for each $n\geq 1$. Then $\tau_s^n$ is a finite random time change. Let $N^n :=M \circ \tau^n$. Then by Proposition \ref{Kazamaki} $(N^n)_{n\geq 1}\subset \mathcal M_{\rm var}^{\rm loc}(H)$ and $[[N^n]]_t = t\wedge[[M]]_n$ for all $n\geq 0$ a.s.\ for all $t\geq 0$. Also notice that $[[N^n]]_t \to t\wedge[[M]]_{\infty}$ as $n\to \infty$. Therefore $N^n$ is a Cauchy sequence in the ucp topology, and thanks to Proposition \ref{prop:linearity} there exists a limit $\tilde{N}\in \mathcal M_{\rm var}^{\rm loc}(H)$. Obviously $\tilde N_s = N_s$ a.s.\ for all $s< S$. If $s\geq S$, then $\tilde N_sh = \lim_{t\to \infty} M_th = M_{\infty} h = Nh$ a.s.\ for each $h\in H$. So, $N=\tilde N\in\mathcal M_{\rm var}^{\rm loc}(H)$.

By the same argument
\[
[[N]]_t = \lim_{n\to \infty}[[N^n]]_t = \lim_{n\to \infty}t\wedge[[M]]_{n} = t\wedge[[M]]_{\infty} = [[M]]_{\tau_t},
\]
which proves Proposition \ref{Kazamaki}(2). To prove Proposition \ref{Kazamaki}(3) note that since the measure $\ud[[N]]$ vanishes on $[S,\infty)$, one can put $Q_N(s) = 0$ if $\tau_s = \infty$, and for $\tau_s < \infty$ one has that
\[
Q_N(s) = \lim_{n\to \infty}Q_{N^n}(s) = \lim_{n\to \infty}Q_{M}(\tau^n_s) = Q_M(\tau_s).
\]
The proof of Proposition \ref{Kazamaki}(4) is analogous to one in the main proof.

Now let us prove the last statement of the corollary.

 Since a.s.\ $\tau\circ[[M]](s) = s$ for $\mu_{[[M]]}$-a.a.\ $s$, we find that a.s.\
 \[
 (\Phi\circ\tau - \Psi)\circ [[M]] = \Phi\circ\tau \circ [[M]] - \Psi \circ [[M]] = 0
 \]
 $\mu_{[[M]]}$-a.e. Therefore according to \eqref{rule}, Proposition \ref{Kazamaki}(2) a.s.\
 \[
 (\Phi\circ\tau - \Psi)\circ [[M]]\circ \tau = \Phi\circ\tau - \Psi = 0
 \]
 $\mu_{[[N]]}$-a.e., which means that a.s.\ $\int_0^{\infty} \|(\Phi\circ\tau - \Psi)Q_N^{1/2}\|^2d[[N]] = 0$, which yields stochastic integrability of $\Phi\circ\tau$ and the first equality of \eqref{eq:inversekaz} thanks to \cite[Exercise~17.3]{Kal}. The last equality of \eqref{eq:inversekaz} is nothing more than formula \eqref{Kaz}.
\end{proof}

The next lemma is a $\gamma$-version of this substitution result and can
be proved as in \cite[Lemma 3.5]{Ver} where the case $H = \mathbb R$ was considered.
\begin{lemma}\label{timechange}
Let $X$ be a Banach space, $H$ be a separable Hilbert space.
Let $F:\mathbb R_+ \to \mathbb R_+$ be increasing
and continuous with $F(0) = 0$ and let $\mu$ be the Lebesgue-Stieltjes measure corresponding to $F$.
Let $S:= \lim_{t \to \infty}F(t)\leq \infty$ and define $\tau:\mathbb{R}_+ \to [0,\infty]$ as
\[
\tau(s) =\begin{cases}
\inf\{t \geq 0:F(t)>s\}, &\text{for}\; 0\leq s<S;
\\
\infty,&\text{for}\; s\geq S.
\end{cases}
\]
Let $\Phi\colon\mathbb R_+ \to \mathcal L(H,X)$ be strongly measurable and define
$\Psi\colon\mathbb R_+ \to \mathcal L(H,X)$ by
\[
\Psi(s) = \begin{cases}
\Phi(\tau_s), &\text{for}\; 0\leq s<S;
\\
0,&\text{for}\; s\geq S.
\end{cases}
\]
Then $\Phi \in \gamma(L^2(\mathbb R_+,\mu;H),X)$ if and only if
$\Psi \in \gamma(L^2(\mathbb R_+;H),X)$. In that case
\begin{equation}\label{eqtimechange}
 \|\Phi\|_{\gamma(L^2(\mathbb R_+,\mu;H),X)} = \|\Psi\|_{\gamma(L^2(\mathbb R_+;H),X)}.
\end{equation}
\end{lemma}

\subsection{Representation and cylindrical Brownian motion}

The next theorem is an infinite time interval version of \cite[Theorem 3.6]{NVW}, while the second part
is modified thanks to \cite[Theorem 5.1]{OV} and the last part modified by \cite[Theorem 4.4]{NVW} and
\cite[Theorem 5.4]{CV}. It will play an important role in the proof of Theorem \ref{main}. It might be instructive for the reader to check that it is exactly Theorem \ref{main} in the special case that $M$ is a cylindrical Brownian motion.

\begin{theorem}\label{Wienercase}
Let $X$ be a UMD space. For a strongly measurable and adapted process
$\Phi \colon \mathbb R_+ \times \Omega\to \mathcal L(H,X)$ which is scalarly in $L^2(\mathbb R_+;H)$ a.s.\
the following assertions are equivalent:
\begin{enumerate}

 \item[(1)] There exists a sequence $(\Phi_n)_{n \geq 1}$ of elementary progressive processes such that:
 \begin{enumerate}
  \item[(i)] for all $x^* \in X^*$ we have
  $\lim\limits_{n \to \infty}\Phi_n^*x^* =
  \Phi^*x^*$ in $L^0(\Omega; L^2(\mathbb R_+, H))$,
  \item[(ii)] there exists a process $\zeta \in L^0(\Omega; C_b(\mathbb R_+; X))$ such that
  \[
  \zeta = \lim_{n \to \infty} \int_0^{\cdot} \Phi_n(t)\ud W_H(t) \;\;\; \text{in}\; L^0(\Omega; C_b(\mathbb R_+;X)).
  \]
 \end{enumerate}
 \item[(2)] There exists an a.s.\ bounded process $\zeta\colon \mathbb R_+\times \Omega\to X$
 such that for all $x^* \in X^*$
 we have
  \[
  \langle\zeta,x^*\rangle = \int_0^{\cdot}\langle \Phi(t),x^*\rangle \ud W_H(t) \;\;\;
  \text{in}\; L^0(\Omega; C_b(\mathbb R_+)).
  \]

 \item[(3)] $\Phi \in \gamma(L^2(\mathbb R_+;H), X)$ almost surely;
\end{enumerate}
In this case $\zeta$ in (1) and (2) coincide and is in $\mathcal M_{\rm var}^{\rm loc}(X)$. Furthermore, for all $p \in (0, \infty)$ we have

\begin{equation}\label{eqwiener}
\mathbb E \sup_{t \in \mathbb R_+}\|\zeta (t)\|^p \eqsim_{p,X} \mathbb E \|\Phi\|^p_{\gamma (L^2(\mathbb R_+;H), X)}.
\end{equation}
\end{theorem}

For the proof of Theorem \ref{main} we will also need the following result which is a simple consequence of \cite[Theorem 2]{Ond}.
\begin{proposition}\label{Th5}
Let $X$ be a reflexive separable Banach space and let
$M\in\mathcal M_{\rm var}^{\rm loc}(X)$. If $[[M]]$ is absolutely continuous
with respect to the Lebesgue measure, then there exists
a separable Hilbert space $H$, an $H$-cylindrical Brownian motion $W_H$ on an enlarged
probability space $(\overline{\Omega}, \overline{\mathbb F}, \overline{\mathbb P})$,
a progressively measurable process $z \colon \mathbb R_+ \times \Omega\to \mathbb R_+$
and a~scalarly progressively measurable process $Q_M^{1/2}: \mathbb R_+ \times \Omega \to \mathcal L(X^*,H)$
which satisfies $Q_M^{1/2*}Q_M^{1/2} = Q_M$ a.s.\ and
$z^{1/2}Q_M^{1/2} \in L^0(\Omega; L^2_{\rm loc}(\mathbb R_+; \mathcal L(X^*,H)))$ such
that a.s.
\[
M_t x^* = \int_0^t z^{1/2}(s) (Q_M^{1/2}(s) x^*)^*\ud W_H(s), \;\;\; t \in \mathbb R_+,x^* \in X^*.
\]
Moreover, if $X$ is a Hilbert space, then for each progressively strongly measurable
$\Phi\colon \mathbb R_+ \times\Omega \to  X^*$ such that
$\int_0^{\infty}\langle Q_M\Phi,\Phi\rangle \ud [[M]] <~\infty$ a.s.\ one has
\begin{equation}\label{eq:intTh5}
 \int_0^{t}\Phi(s) \ud M(s) = \int_0^t z^{1/2}(s) Q_M^{1/2}(s) \Phi(s)\ud W_H(s),\;\;\;\;t \in \mathbb R_+.
\end{equation}
\end{proposition}

\begin{remark}\label{rem:existintMabscont}
 The integral in the left hand side of \eqref{eq:intTh5} exists for the special $M$ with absolutely continuous quadratic variation thanks to the isometry given in \cite[Remark~30]{Ond} and the construction given in \cite[p.1022]{Ond}.
\end{remark}

\begin{proof}
Since $[[M]]$ is absolutely continuous with respect to the Lebesgue measure, one can find $z \colon \mathbb R_+ \times\Omega\to \mathbb R_+$
such that $[[M]]_t = \int_0^t z(s)ds$ for each $t \in \mathbb R_+$ a.s.\
Define $H$ and $Q_M^{1/2}: \mathbb R_+ \times \Omega \to \mathcal L(X^*,H)$ as in Lemma \ref{sqroot}.
By Remark \ref{sqrootmes} the process $Q_M^{1/2}$ is scalarly progressively measurable.
Then for all $x^*,y^* \in X^*$
\begin{align*}
[Mx^*,My^*]_t & = \int_0^t\ud [Mx^*,My^*]_s =
\int_0^t \langle Q_Mx^*,y^*\rangle \ud [[M]]_s
\\ & = \int_0^t \langle z(s)Q_M(s)x^*,y^*\rangle\ud s
\\ &=\int_0^t \langle (z(s)^{1/2}Q_M(s)^{1/2})x^*,(z(s)^{1/2}Q_M(s)^{1/2})y^*\rangle\ud s,
\end{align*}
and the rest follows from \cite[Theorem 2]{Ond}.

The last equation is evident for elementary functions, and the general case
follows from a density argument, Remark \ref{rem:existintMabscont} and the isometry, mentioned in \cite[Remark~30]{Ond}.
\end{proof}

\subsection{Proof of the main characterization Theorem \ref{main}\label{subsec:proofmain}}

To prove the result we will reduce to  Theorem \ref{Wienercase} by using the time transformation from Corollary \ref{cor:inversekaz} and the representation of Proposition \ref{Th5}.
\begin{proof}[Proof of Theorem \ref{main}]

Define $\tau: \Omega\times {\mathbb R_+} \to [0,\infty]$ as follows:
\begin{equation}\label{proofmain1}
 \tau_s = \begin{cases}
\inf\{t\geq 0:[[M]]_t>s\}, &\text{for}\; 0\leq s<[[M]]_{\infty};
\\
\infty,&\text{for}\; s\geq [[M]]_{\infty}.
\end{cases}
\end{equation}
Put
\begin{equation}
 \Psi(s) = \begin{cases}\label{proofmain2}
\Phi(\tau_s), &\text{for}\; 0\leq s<[[M]]_{\infty};
\\
0,&\text{for}\; s\geq [[M]]_{\infty}.
\end{cases}
\end{equation}

For each $s \geq 0$ it holds true that $[[M]]_{\tau_s} - [[M]]_{\tau_{s-}} = 0$
a.s.\ So, since for fixed $h\in H$, $\mu_{[[M]]}\geq \mu_{[Mh]}$, then
also $[Mh]_{\tau_s} - [Mh]_{\tau_{s-}} = 0$ a.s.\ Therefore thanks to the fact that
$\tau_{s-}$ is a stopping time, so $(Mh)^{\tau_s} -(Mh)^{\tau_{s-}}$
is a continuous local martingale with zero quadratic variation (see \cite[Theorem~I.18]{Prot}),
and by Remark \ref{rem:Kal176} and
\cite[Problem 1.5.12]{KS} one concludes that
$Mh$ is $\tau$-continuous.

It also follows that $([[M]]\circ \tau)_s = s$ for $s < [[M]]_{\infty}$.
Let $\mathbb G$ be as in Corollary \ref{cor:inversekaz}.
By Corollary \ref{cor:inversekaz} one can define a local $H$-cylindrical continuous $\mathbb{G}$-martingale
$N:\mathbb R_+ \times H\to L^0(\Omega)$ such that $N = M \circ \tau$, $[[N]]_s =s$ for $s <[[M]]_{\infty}$, and
$Q_N = Q_M \circ \tau$.

Let
$W_H$ and $(\overline{\Omega}, \overline{\mathbb F}, \overline{\mathbb P})$ be as in
Proposition \ref{Th5}. We will prove the result by showing that (1), (2) and (3)
for $\Phi$ are equivalent with (1), (2) and (3) in Theorem \ref{Wienercase} for~$\Psi\,Q_N^{1/2}$.
(Notation $(k,\Phi)\Leftrightarrow(k,\Psi)$ for $k=1,2,3$).

(1, $\Phi$) $\Rightarrow$ (1, $\Psi$): Assume (1) holds for a sequence
of elementary progressive processes $(\Phi_n)_{n \geq 1}$. For all $n \geq 1$
define $\Psi_n:\mathbb R_+\times \Omega \to \mathcal L(H,X)$ as

\[
\Psi_n(s) = \begin{cases}
\Phi_n(\tau_s), &\text{for}\; 0\leq s<[[M]]_{\infty},
\\
0,&\text{for}\; s\geq [[M]]_{\infty}.
\end{cases}
\]

Then it follows from the Pettis measurability theorem and Corollary \ref{cor:inversekaz}
that each $\Psi_n$ is strongly progressively measurable with respect to the time transformed filtration, and the same holds true for each
$\Psi_n\,Q_N^{1/2}$, because $\Phi_n$ takes their values in finite dimensional subspace of $X$. So since
$\Phi_n$ is elementary progressive it follows from \eqref{Kaz}, Corollary \ref{cor:inversekaz}, Proposition \ref{Th5}, Remark \ref{rem3}
that for all $n \geq 1$ for all $s\in \mathbb R_+$ we have a.s.
\[
\zeta_{\Psi_n\,Q_N^{1/2}}(s) = \int_{0}^s\Psi_n(r)\,Q_N^{1/2}(r)\ud W_H(r) = \int_0^s\Psi_n(r)\ud N(r)
=\int_0^{\tau_s}\Phi_n(r)\ud M(r)
\]
(recall that $z(s) = [[N]]_s' = 1$ for $s < [[M]]_{\infty}$).

Therefore, it follows that $(\zeta_{\Psi_n\,Q_N^{1/2}})_{n\geq 1}$ is a Cauchy sequence in
$L^0(\Omega;C_b(\mathbb R_+;X))$, and hence it converges to some
$\zeta_{\chi} \in L^0(\Omega;C_b(\mathbb R_+;X))$. By \eqref{rule}, Theorem \ref{main}
(1) (i), by the special choice of $\overline{\Omega}$
and by Fubini's theorem it follows that for every $x^* \in X^*$ we have
$\lim_{n \to \infty}Q_N^{1/2}\Psi_n^*x^* =
Q_N^{1/2}\Psi^*x^*$ in $L^0(\overline{\Omega}; L^2(\mathbb R_+;H))$.
Since $\Psi_n\,Q_N^{1/2}h$
take values in finite dimensional subspace of $X$ for each $h \in H$,
one can approximate $(\Psi_n\, Q_N^{1/2})_{n \geq 1}$ to obtain
a sequence of elementary progressive processes $(\hat{\chi}_n)_{n \geq 1}$ that satisfies
Theorem \ref{Wienercase} (1) (i) and (ii).

(1, $\Psi$) $\Rightarrow$ (1, $\Phi$): Let Theorem \ref{Wienercase} (1)
be satisfied for $\Psi\,Q_N^{1/2}$ on the enlarged probability space
$(\overline{\Omega}, \overline{\mathcal F}, \overline{\mathbb P})$. Then it follows from
Theorem \ref{Wienercase} that $\Psi\,Q_N^{1/2} \in \gamma(L^2(\mathbb R_+;H), X)$
$\overline{\mathbb P}$-a.s.\ By special choice of $\overline{\Omega}$
and by Fubini's theorem we may conclude that $\Psi \,Q_N^{1/2} \in \gamma(L^2(\mathbb R_+;H), X)$
$\mathbb P$-almost everywhere. By \cite[Remark 2.8]{NVW}
$\Psi \,Q_N^{1/2} \in L^0(\Omega;\gamma(L^2(\mathbb R_+;H), X))$.
Then by \cite[Lemma 3.2]{Ver}, \cite[Proposition 2.10]{NVW} and \cite[Proposition 2.12]{NVW}
there exist elementary progressive processes $(\chi_n)_{n\geq 1}$ in
$L^0(\Omega;\gamma(L^2(\mathbb R_+;H), X))$ such that $\Psi\,Q_N^{1/2} = \lim_{n\to \infty} \chi_n$
in $L^0(\Omega;\gamma(L^2(\mathbb R_+;H), X))$.

\medskip
Let $n$ be fixed. Without loss of generality
one can suppose that $\chi_n$ has the following form:
\[
\chi_n = \sum_{i=1}^{I} \sum_{j=1}^J \mathbf 1_{(t_{i-1},t_i]\times B_{ij}}\sum_{k=1}^K h_k\otimes x_{ijk}.
\]
Fix $\omega \in \Omega$. Let $P_0:\mathbb R_+ \times \Omega \to \mathcal L(H)$
be the projection onto
$\overline{\text{ran } Q_N^{1/2}(t,\omega)}$.
It is easy to check that $P_0$ is scalarly progressively measurable and
$\|P_0\|\leq 1$. By the ideal property \eqref{ideal} one has $\P$-a.e.
\begin{align*}
\|\Psi Q_N^{1/2} - \chi_n P_0\|_{\gamma(L^2(\mathbb R_+;H),X)}
&= \|\Psi Q_N^{1/2}P_0 - \chi_n P_0\|_{\gamma(L^2(\mathbb R_+;H),X)}\\
&\leq\|\Psi Q_N^{1/2} - \chi_n \|_{\gamma(L^2(\mathbb R_+;H),X)},
\end{align*}
thanks to $P_0Q_N^{1/2}=Q_N^{1/2}P_0 = Q_N^{1/2}$.

Now for each $k \geq 1$ define $P_k \in \mathcal L(H)$ in the same way as $P_0$,
but by taking projections onto $Q_N^{1/2}(\text{span }(h_1,\ldots,h_k))$.
Note that $P_k$ is a scalarly measurable operator. By \cite[Proposition 2.4]{NVW},
pointwise on $\Omega$ we have $\|\chi_n P_k - \chi_n P_0\|_{\gamma(L^2(\mathbb R_+;H),X)} \to 0$
as $k \to \infty$.

Fix $k\geq 1$. By Lemma \ref{technical} (applied with $F = Q_N^{1/2}$) we can find $H$-strongly progressive $\tilde{P}_k, L_k:\mathbb R_+\times\Omega \to \mathcal L(H)$ such that
\begin{equation}\label{eq:niceidenti}
\tilde P_k Q_N^{1/2} = Q_N^{1/2} P_k \ \ \text{and} \ \  L_k Q_N^{1/2} = P_k.
\end{equation}
For each $n,k \geq 1$ one let $\Psi_{nk} = \chi_{n} L_k\in L^0(\Omega; \gamma(L^2(\mathbb R_+;H),X))$.
Then by \eqref{eq:niceidenti} $\Psi_{nk}Q_N^{1/2} = \chi_nP_k$. Since $\Psi_{nk}Q_N^{1/2} \to \chi_nP_0$ as $k \to \infty$, we can choose a subsequence $(k_n)_{n\geq 0}$
and define $\Psi_n := \Psi_{nk_n}$ such that $\Psi Q_N^{1/2} = \lim_{n\to \infty} \Psi_n Q_N^{1/2}$ in $L^0(\Omega;\gamma(L^2(\mathbb R_+;H), X))$.

Without loss of generality assume
that $\Psi_n \,Q_N^{1/2}(s) = 0$ for $s \geq [[M]]_{\infty}$. For each $n \geq 1$
define $\Phi_n \colon \mathbb R_+ \times \Omega \to \mathcal L(H,X)$ as
${\Phi}_n =\Psi_n \circ [[M]]$. It is easy to see that ${\Phi}_n\,Q_M^{1/2} = (\Psi_n\,Q_N^{1/2})\circ [[M]]$ for each $n>0$.
Then $\Phi_n\,Q_M^{1/2}$ is a sequence of strongly progressively measurable processes,
and $(\Phi_n\,Q_M^{1/2})\circ \tau = \Psi_n\,Q_N^{1/2}$.

By the substitution rule \eqref{rule} for all $x^* \in X^*$ one has
\[
\left\|Q_M^{1/2}\Phi^*x^*- Q_M^{1/2}\Phi_n^*x^*\right\|_{L^2(\mathbb R_+,[[M]];H )}=
\left\|Q_N^{1/2}\Psi^*x^*- Q_N^{1/2}\Psi_n^*x^*\right\|_{L^2(\mathbb R_+;H)},
\]
and we derive (1) (i) because the last expression converges to 0 in probability. By the It\^{o} homeomorphism
\cite[Theorem 5.5]{NVW} and the fact that $\Psi_n Q_N^{1/2} \to \Psi Q_N^{1/2}$ in
$L^0(\Omega;\gamma(L^2(\mathbb R_+;H), X))$ one obtains
\[
\int_0^{\cdot} \Psi(t)Q_N^{1/2}(t) \ud W_H(t) = \lim_{n \to \infty}\int_0^{\cdot} \Psi_n(t)Q_N^{1/2}(t) \ud W_H(t)
\,\,\,\,\,\text{in}\,\,\, L^0(\Omega; C_b(\mathbb R_+;X)).
\]
Since $\Psi_n\,Q_N^{1/2}$ are progressively strongly measurable processes one concludes from Proposition
\ref{Th5} and the fact that $[[N]]_s = s$ for $s \leq [[M]]_{\infty}$ (and so $z(s) = [[N]]_s' = 1$) that,
almost surely for all $t\in \mathbb R_+$ and for all $n \geq 1$
\begin{equation}\label{eq:intfortimechange}
\int_0^{[[M]]_t} \Psi_n(s)Q_N^{1/2}(s)\ud W_H(s) = \int_0^{[[M]]_t} \Psi_n(s)\ud N(s)=\int_0^t \Phi_n(s) \ud M(s).
\end{equation}
Here the second identity follows from Corollary \ref{cor:inversekaz}.

It follows that $(\int_0^t \Phi_n(s) \ud M(s))_{n\geq 1}$ is a Cauchy sequence in $L^0(\Omega; C_b(\mathbb R_+;X))$.
Now as in the proof of the previous step one may conclude (1) (ii) via an approximation argument.

(2, $\Phi$) $\Rightarrow$ (2, $\Psi$): Let $\zeta: \mathbb R_+ \times\Omega \to X$
be the given stochastic integral process. Let $\zeta_{\Psi}:\mathbb R_+ \times\Omega \to X$ be defined as
\[
\zeta_{\Psi}(s) =
\begin{cases}
\zeta(\tau_s), &\text{for}\; 0\leq s<[[M]]_{\infty},
\\
\text{weak}-\lim_{t \to \infty}\zeta(t),&\text{for}\; s\geq [[M]]_{\infty}.
\end{cases}
\]
The weak limit exists a.e. and it is strongly measurable by \cite[Lemma 3.8]{Ver}.

Moreover, by Corollary \ref{cor:inversekaz} and Proposition \ref{Th5}
\[
\langle\zeta_{\Psi},x^*\rangle = \int_0^{\cdot}\Psi(t)^* x^* \ud N(t)
= \int_0^{\cdot}\langle Q_N^{1/2} \Psi(t)^* (t) x^* \ud W_H(t)\;\;\;
  \text{in}\; L^0(\Omega; C_b(\mathbb R_+)).
\]
On the other hand since $\zeta$ is a.s.\ bounded, the same holds for $\zeta_{\Psi}$. Therefore, Theorem \ref{Wienercase} (2) holds for $\Psi\, Q_N^{1/2}$ and ${\zeta}_{\Psi}$.

(2, $\Psi$) $\Rightarrow$ (2, $\Phi$): Let ${\zeta}_{\Psi}$ be the stochastic
integral process of $\Psi\,Q_N^{1/2}$ with respect to $W_H$. Let $\zeta:\mathbb R_+ \times \Omega \to X$
be defined as $\zeta=\zeta_{\Psi}\circ [[M]]$. Then $\zeta \in L^0(\Omega; C_b(\mathbb R_+;X))$
and it follows from Proposition \ref{Th5} that for all $x^* \in X^*$, for all $t \in \mathbb R_+$ a.s.\ we have
\begin{align*}
 \langle\zeta(t),x^*\rangle =\langle\zeta_{\Psi}([[M]]_t),x^*\rangle &= (\langle\zeta_{\Psi},x^*\rangle)([[M]]_t)
 =\int_0^{[[M]]_t} Q_N^{1/2} \Psi^* x^* \ud W_H(r)\\
 &= \int_0^{[[M]]_t} \Psi^* x^* \ud N(r) =\int_0^{t} \Phi^* x^* \ud M(r).
\end{align*}
Here the last identity follows from Corollary \ref{cor:inversekaz}.

(3, $\Phi$) $\Leftrightarrow$ (3, $\Psi$): This statement is obvious
by Lemma \ref{timechange}. Furthermore, from \eqref{eqtimechange} it follows that
$\P$-a.s.\ we have
\begin{equation}\label{eq7}
 \|\Phi\,Q_M^{1/2}\|_{\gamma(L^2(\mathbb R_+,[[M]];H),X)} = \|\Psi\,Q_N^{1/2}\|_{\gamma(L^2(\mathbb R_+;H),X)}.
\end{equation}

Therefore $\|\Phi\,Q_M^{1/2}\|_{\gamma(L^2(\mathbb R_+,[[M]];H),X)}$ is a measurable
function on $\Omega$. Since $\zeta(t) = \zeta_{\Psi}([[M]]_t)$ and by using Proposition \ref{Th5}, \eqref{eqwiener} and \eqref{eq7} one derives
for $p\in(0,\infty)$
\begin{align*}
\mathbb E \sup_{t\in\mathbb R_+} \|\zeta(t)\|^p & =
\mathbb E \sup_{t\in\mathbb R_+}\left\|\int_0^t \Phi \ud M\right\|^p =
\mathbb E \sup_{t\in\mathbb R_+}\left\|\int_0^{t} \Psi \ud N\right\|^p
\\ &=\mathbb E \sup_{t\in\mathbb R_+}\left\|\int_0^{t} \Psi\,Q_N^{1/2} \ud W_H\right\|^p
\\ & \eqsim_{p,X}\mathbb E \|\Psi\, Q_N^{1/2}\|^p_{\gamma (L^2(\mathbb R_+;H), X)}
=\mathbb E \|\Phi\, Q_M^{1/2}\|^p_{\gamma (L^2(\mathbb R_+, [[M]];H), X)},
\end{align*}
which proves the last part of Theorem \ref{main}.
\end{proof}

By the above proof and a limiting argument in $L_0(\Omega; C_b(\mathbb R_+;X))$ one obtains the following theorem, which can be seen as a vector-valued generalization of the famous Dambis-Dubins-Schwarz theorem (see \cite[Theorem 18.4]{Kal} for the isotropic case in finite dimensions).

\begin{theorem}\label{rem:intfortimechange}
Let $H$ be a Hilbert space, $X$ be a UMD Banach space, $M \in \mathcal M_{\rm var}^{\rm loc}(H)$, $(\tau_s)_{s\geq 0}$ be the time change defined as in \eqref{proofmain1}. Then we have that there exists an $H$-cylindrical Brownian motion $W_H$ that does not depend on $X$ such that for any $\Phi:\mathbb R_+\times \Omega \to \mathcal L(H,X)$ which is stochastically integrable with respect to $M$, one has a.s.
\[
\int_0^{t}\Phi(s)\ud M = \int_0^{[[M]]_t}(\Phi(s)Q_M(s))\circ \tau \ud W_H, \ \ t\geq 0.
\]
\end{theorem}

\subsection{Further consequences}
During the proof of Theorem \ref{main} we have obtained the following corollary, which is absolutely
analogous to \cite[Corollary~3.9]{Ver}:
\begin{corollary}[Kazamaki, infinite dimensional case]\label{cor3.9}
Assume the conditions of Theorem \ref{main} hold and formula \eqref{proofmain1}.
If $\Phi:\mathbb R_+\times \Omega \to \mathcal L(H,X)$ is scalarly $\mathbb F$-measurable
and satisfies $\Phi Q_M^{1/2} \in \gamma(L^2(\mathbb R_+, [[M]]; H),X)$ a.s., then the process
$\Psi :\mathbb R_+ \times \Omega \to \mathcal L(H,X)$ defined as in \eqref{proofmain2}
is $\mathbb G$-adapted and satisfies $\Psi Q_N^{1/2} \in \gamma(L^2(\mathbb R_+; H),X)$ a.s.,
and the $X$-valued version of \eqref{Kaz} holds.
\end{corollary}

Using this corollary one can prove the following analogue of \cite[Corollary~3.10]{Ver}:
\begin{corollary}\label{cor3.10}
Let $X$ be a UMD space. For each $n\geq$ let $\Phi_n:\mathbb R_+ \times \Omega\to \mathcal L(H,X)$
be stochastically integrable and let $\zeta_n \in L^0(\Omega, C_b(\mathbb R_+,X))$ denote its stochastic
integral. Then we have $\Phi_n Q_M^{1/2} \to 0$ in $L^0(\Omega;\gamma(L^2(\mathbb R_+,[[M]];H),X))$
if and only if $\zeta_n \to 0$ in $L^0(\Omega; C_b(\mathbb R_+;X))$.
\end{corollary}

\begin{corollary}[Local property]\label{cor:whenintis0}
 Let $X$ be a UMD space, $\Phi:\mathbb R_+ \times \Omega\to \mathcal L(H,X)$ be stochastically integrable. Suppose that there exists $A\in \mathcal F$ such that for all $x^* \in X^*$ a.s.\ for all $t\geq 0$, $\Phi^*(t)x^* = 0$. Then a.s.\ in $A$ for all $t\geq 0$
 \[
 \int_0^t \Phi \ud M = 0.
 \]
\end{corollary}

\begin{proof}
By Hahn--Banach and strong measurability, it is enough to show that for each $x^* \in X^*$ a.s.\ in $A$ for all $t\geq 0$
\[
N_t := \int_0^t \Phi^*x^*\ud M = 0.
\]
But we know that by Remark \ref{realvalued} a.s.\ on $A$
\[
[N]_{\infty} = \int_0^{\infty} (\Phi^*x^*) Q_M (\Phi^*x^*)^* \ud [[M]] = 0,
\]
what yields the desired by \cite[Exercise~17.3]{Kal}.
\end{proof}

\begin{remark}
Due to \cite[Proposition 3.2]{NVW} the implication $(1)\Rightarrow(2)$ can be proven
for any Banach space $X$, because in the proofs of $(1,\Phi)\Rightarrow (1,\Psi)\Rightarrow(2,\Psi)\Rightarrow(2,\Phi)$
one does need the UMD property. The same holds true for $(3,\Phi)\Leftrightarrow (3,\Psi)$
because there is no restriction on $X$ in Lemma \ref{timechange}.
\end{remark}

The next corollary is a generalization of both \cite[Corollary 4.1]{Ver} and \cite[Proposition 6.1]{NVW07c}.
Let $\mathcal{P}$ denote the progressive measurable $\sigma$-algebra in the result below.
\begin{corollary}
Let $X$ be a UMD Banach function space over a $\sigma$-finite measure space $(S,\Sigma,\mu)$
and let $p \in (0,\infty)$. Let $\Phi: \mathbb R_+ \times \Omega \to \mathcal L(H,X)$
be scalarly progressive and assume that there exists a $\mathcal{P}\times \Sigma$-measurable
process $\phi:\mathbb R_+ \times \Omega \times S \to H$ such that for all $h \in H$
and $t \geq 0$
\[
(\Phi(t)h)(\cdot) = \langle \phi(t,\cdot),h\rangle,
\]
where the equality holds in $X$. Then $\Phi$ is stochastically integrable
with respect to $M$ if and only if almost surely
\[
\left\|\left( \int_{\mathbb R_+} \|Q_M^{1/2}(t)\phi(t,\cdot)\|_{H}\ud [[M]]_t \right)^{\frac12} \right\|_X < \infty.
\]
In this case
\[
\mathbb E \sup_{t\geq 0} \Bigl\|\int_{\mathbb R_+}\Phi(t)\ud M(t)\Bigr\|_X^p\eqsim_{p,X}
\mathbb E \Bigl\| \Bigl( \int_{\mathbb R_+}\|Q_M^{1/2}(t)\phi(t,\cdot)\|_H\ud [[M]]_t \Bigr)^{\frac 12} \Bigr\|_X^p.
\]
\end{corollary}
\begin{proof}
To prove this statement note that as in  \cite[Proposition 6.1]{NVW07c}
\[\Bigl\| \Bigl( \int_{\mathbb R_+}\|Q_M^{1/2}(t)\phi(t,\cdot)\|_H\ud [[M]]_t \Bigr)^{\frac 12} \Bigr\|_X \eqsim \|\Phi\, Q_M^{1/2}\|_{\gamma (L^2(\mathbb R_+, [[M]];H), X)}\]
and hence the results follows from Theorem \ref{main}.
\end{proof}

Due to the canonical embedding $L^2(\mathbb R_+,\mu;\gamma(H,X))\hookrightarrow \gamma(L^2(\mathbb R_+,\mu;H),X)$\
for~a~measure $\mu$ for type 2 spaces, and the reversed embedding for cotype 2 spaces, stated
in \cite[Theorem 11.6]{Ngamma}, one obtains the full analogue of \cite[Corollary~4.2]{Ver}:
\begin{corollary} Let $X$ be a UMD space, $p\in (0,\infty)$ and $M \in \mathcal M_{\rm var}^{\rm loc}(H)$.
 \begin{itemize}
  \item[(1)] If $X$ has type 2, then every scalarly progressively measurable process $\Phi:\mathbb R_+\times\Omega\to \calL(H,X)$ such that
  $\Phi Q_M^{1/2} \in L^2(\mathbb R_+,[[M]];\gamma(H,X))$ almost surely is stochastically integrable with respect to $M$
  and we have
  \[
  \mathbb E \sup_{t\geq 0} \Bigl\| \int_{\mathbb R_+} \Phi(t) \ud M(t) \Bigr\|^p \lesssim_{p,X}
  \mathbb E\|\Phi Q_M^{1/2}\|^p_{L^2(\mathbb R_+, [[M]];\gamma(H,X))}.
  \]
  \item[(2)] If $X$ has cotype 2, then every scalarly progressively measurable process
  $\Phi$ which is integrable with respect to $M$ satisfies $\Phi Q_M^{1/2}\in L^2(\mathbb R_+, [[M]];\gamma(H,X))$ almost surely and we have
  \[
   \mathbb E\|\Phi Q_M^{1/2}\|^p_{L^2(\mathbb R_+, [[M]];\gamma(H,X))} \lesssim_{p,X}
  \mathbb E \sup_{t\geq 0} \Bigl\| \int_{\mathbb R_+} \Phi(t) \ud M(t) \Bigr\|^p.
  \]
  \item[(3)] If $X$ is a Hilbert space, then $\Phi$ is integrable with respect to $M$ if and only if $\Phi Q_M^{1/2}\in L^2(\mathbb R_+, [[M]];\calL_2(H,X))$ almost surely, and we have
  \[
  \mathbb E \sup_{t\geq 0} \Bigl\| \int_{\mathbb R_+} \Phi(t) \ud M(t) \Bigr\|^p \eqsim_{p} \mathbb E\|\Phi Q_M^{1/2}\|^p_{L^2(\mathbb R_+, [[M]];\calL(H,X))}.
  \]
 \end{itemize}
\end{corollary}

\subsection{It\^{o}'s formula}

We will say that $\Phi\in \gamma_{\rm loc}(L^2(\mathbb R_+,[[M]];H),X)$ a.s.\ if for every $T>0$,
$\Phi \mathbf 1_{[0,T]}\in \gamma(L^2(\mathbb R_+,[[M]];H),X)$ a.s.\ It is an easy consequence of Theorem~\ref{main}
that $\Phi$ is  \emph{locally stochastically integrable} if and only if
$\Phi\,Q_M^{1/2}\in \gamma_{\rm loc}(L^2(\mathbb R_+,[[M]];H),X)$.

A function $f: \mathbb R_+ \times X \to Y$ is said to be of  \emph{class $C^{1,2}$} if it is differentiable
in the first variable and twice Fr\'echet differentiable in the second variable and the
functions $f$, $D_1f$, $D_2f$ and $D^2_2f$ are continuous on $\mathbb R_+\times X$

For $R\in \gamma(H,X)$ and $T\in \calL(X,X^*) = \mathbf B(X,X)$,
\[{\rm Tr}_R(T) = \sum_{n\geq 1} T (Rh_n, Rh_n),\]
where $(h_n)_{n\geq 1}$ is any orthonormal basis for $H$ (see \cite[Lemma 2.3]{BNVW} for details). The following version of It\^o's formula holds
\begin{theorem}
 Let $H$ be a Hilbert space, $X$ and $Y$ be UMD Banach spaces, $M \in \mathcal M_{\rm var}^{\rm loc}(H)$ and let $A: \mathbb R_+\times \Omega \to \mathbb R$ be adapted, a.s. continuous and locally of finite variation. Assume that $f: \mathbb R_+ \times X \to Y$ is of class $C^{1,2}$. Let $\Phi: \mathbb R_+ \times \Omega \to \mathcal L(H,X)$ be an $H$-strongly progressively measurable which is stochastically integrable with respect to $M$ and assume that $\Phi Q_M^{1/2}$ belongs to $L^2_{loc}(\mathbb R_+, [[M]]; \gamma(H,X))$. Let $\psi:\mathbb R_+ \times\Omega \to X$ be strongly progressively measurable with paths in $L^1_{loc}(\mathbb R_+,A;X)$ a.s.\ Let $\xi:\Omega \to X$ be strongly $\mathcal F_0$-measurable. Define $\zeta:\mathbb R_+ \times \Omega \to X$ as
 \[
 \zeta = \xi + \int_0^{\cdot}\psi(s)\ud A(s) + \int_0^{\cdot} \Phi(s)\ud M(s).
 \]
 Then $s \mapsto D_2f(s,\zeta(s))\Phi(s)$ is locally stochastically integrable with respect to $M$ and almost sure we have for all $t\geq 0$
 \begin{equation}\label{eq:Itoformula}
 \begin{split}
  f(t,\zeta(t))-f(0,\zeta(0)) &= \int_0^t D_1 f(s,\zeta(s))\ud s + \int_0^t D_2f(s,\zeta(s))\psi\ud A(s)\\
  &+ \int_0^t D_2 f(s,\zeta(s))\Phi(s)\ud M(s)\\
  &+ \frac 12 \int_0^t {{\rm Tr}}_{\Phi(s){Q_M^{1/2}(s)}} (D_2^2f(s,\zeta(s)))\ud [[M]]_s.\
  \end{split}
\end{equation}
\end{theorem}
A typical application of this formula are the case where $f:X\to \mathbb R$ is given by $f(x) = \|x\|^p$ whenever this two time Fr\'echet differentiable and satisfies appropriate estimates (e.g. $X = L^p$ with $p\geq 2$).  Another application is $f:X\times X^*\to \mathbb R$ given by $f(x, x^*) = \lb x, x^*\rb$.

To prove this result we can reduce to the case $Y = \mathbb R$ in a similar way as in \cite[Theorem 2.4]{BNVW} step 1. Indeed, if the formula holds true for $F = \mathbb R$, then we can apply the result to $\langle f,y^*\rangle$ for each $y^*\in Y$. After that we can apply Theorem \ref{main} (2) to derive the stochastic integrability of $s\mapsto D_2 f(s,\zeta(s))\Phi(s)$. The identity \eqref{eq:Itoformula} then follows from the Hahn-Banach theorem.

The next step is to reduce the proof to the case where $\xi$ is simple and both $\psi$ and $\Phi$ have finite dimensional range (see \cite[Theorem 2.4]{BNVW} step 2).  As soon as we have this reduction, then there exists a fixed finite dimensional subspace $H_0 \subset H$ such that $H = H_0 \oplus \ker \Phi$. Then one can restrict $M$ onto this subspace, and thanks to Example \ref{ex:findim} one can use the usual finite-dimensional It\^o formula to derive the required result (see e.g.\ \cite[Section 3.3]{MP}).

\begin{lemma}\label{lemma:foritoproof1}
Let $X$ be a UMD Banach space, $H$ be a Hilbert space, $M \in M_{\rm var}^{\rm loc}(H)$. Let $\Phi:\mathbb R_+\times \Omega \to \mathcal L(H,X)$ be stochastically integrable with respect to $M$. Assume that its paths are in $L^2(\mathbb R_+,[[M]];\gamma(H,X))$ almost surely. Then there exists a sequence of progressive processes $(\Phi_n)_{n\geq 1}$ such that each $\Phi_n$ takes values in a finite dimensional subspace of $X$ and is supported on a finite dimensional subspace of $H$ and
\[
\Phi_n Q_M^{1/2} \to \Phi Q_M^{1/2}\; in\;\; L^2(\mathbb R_+,[[M]];\gamma(H,X))\cap \gamma(L^2(\mathbb R_+,[[M]];H),X) \ \ \text{in probability}.
\]
\end{lemma}
\begin{proof}
Let $(\tilde\Phi_n)_{n\geq 1}$ be constructed as in \eqref{eq:niceidenti}. Then $(\tilde\Phi_n Q_M^{1/2})_{n\geq 1}$ is an approximation of $\Phi Q_M^{1/2}$ in $ L^2(\mathbb R_+,[[M]];\gamma(H,X)))\cap \gamma(L^2(\mathbb R_+,[[M]];H),X)$ in probability. By \cite[Proposition 2.4]{NVW} $(\tilde\Phi_n P_k Q_M^{1/2})_{k\geq 1}$ approximates $\tilde\Phi_n Q_M^{1/2}$ for each $n\geq 1$, where $P_k$ is an orthogonal projection onto $\text{span}(h_1,\ldots,h_k)$. So, choosing a subsequence $\Phi_n := \tilde\Phi_n P_{k_n}$ one derives the desired.
\end{proof}

The next lemmas is taken from \cite[Lemma~2.8]{BNVW}:

\begin{lemma}\label{lemma:foritoproof2}
Let $X$ be a Banach space, $A:\mathbb R_+ \to \mathbb R_+$ be an increasing continuous function, and $\psi \in L^0(\Omega;L^1(\mathbb R_+,A;X))$ be a progressively measurable process. Then there exists a sequence of elementary progressive processes $(\psi_n)_{n\geq 1}$ such that $\psi = \lim_{n\to \infty} \psi_n$ in $L^0(\Omega; L^1(\mathbb R_+,A;X))$.
\end{lemma}

\section{Stochastic evolution equations and cylindrical noise\label{sec:SE}}

In this section we study existence and uniqueness of solutions to the stochastic evolution equation on a UMD space $X$:
\[\ud u = (Au(t)+ F(t,u))\ud t + G(t,u)\ud M, \ \  t\in [0,T],\]
where $u(0) = u_0$. Here $A$ is the generator of an analytic semigroup on $X$, $F$ and $G$ are nonlinearities and $M$ is a cylindrical continuous local martingale on a Hilbert space $H$ which admits a quadratic variations as introduced in Definition \ref{def:cylmart}. We will treat the above problem by semigroup methods. The case $M = W_H$ has been extensively studied in the literature (see \cite{Brz2,DPZ,NVW1}). Before we start we need some preliminaries from analysis.

\subsection{Analytic preliminaries}

Let $X$ and $Y$ be two Banach spaces, $(r_n)_{n \geq 1}$ be a Rademacher sequence,
i.e. a sequence of independent random variables satisfying $\mathbb P(r_n = 1) = \mathbb P(r_n = -1) = \frac 12$.
A family $\mathcal T \subseteq \mathcal L(X, Y)$ is called $R$-bounded if
there exists a constant $C$ such that for each $N>0$, $(x_n)_{n=1}^N \subseteq X$ and
$(T_n)_{n=1}^N \subseteq \mathcal T$ one has
\[
\Bigl(\mathbb E\Bigl\|\sum_{n=1}^Nr_nT_nx_n\Bigr\|^2\Bigr)^{\frac 12} \leq
C\Bigl(\mathbb E\Bigl\|\sum_{n=1}^Nr_nx_n\Bigr\|^2\Bigr)^{\frac 12}.
\]
The least such $C$ is called $R$-bound of $\mathcal T$, notation $R(\mathcal T)$.

If one replaces the Rademacher sequence by a sequence of independent Gaussian variables in the definition above,
then one obtains the notion of $\gamma$-bounded family of operators, whose $\gamma$-bound is denoted
by $\gamma(\mathcal T)$. A simple randomization argument shows that $R$-boundedness implies $\gamma$-boundedness, and in this case $\gamma(\mathcal T)\leq R(\mathcal T)$ and the converse fails in general (see \cite{KVW14}).

A set $(\Lambda,\leq)$ with an order $\leq$ is called a  \textit{set with a total order} if
for any $x,y \in \Lambda$ it holds true that $x \leq y$ or $y\leq x$. The next result is due to \cite{Bour} (for a proof see \cite{HNVW1}):

\begin{lemma}[Vector-valued Stein's inequality] \label{vectorstein}
 Let $(S, \mathcal A, \mu)$ be a probability space, $X$ be a UMD space. Let
 $\Lambda$ be a set with a total order. Then for all $1<p<\infty$ and every increasing set
 $\{\mathcal A_{\alpha}\}_{\alpha \in \Lambda}$ of sub-$\sigma$-algebras of $\mathcal A$
 one has that the family of conditional expectations
 \[
 \mathcal E_{p} = \{\mathbb E(\cdot|\mathcal A_{\alpha}),\alpha \in \Lambda\}\subseteq \mathcal L(L^p(\Omega;X))
 \]
 is $R$-bounded as a set of operators with an $R$-bound depending only on $p$ and $X$.
\end{lemma}

We will need the following technical lemma about $\gamma$-spaces:
\begin{lemma}\label{gammameasure}
Let $X$ be a Banach space, $T >0$. Let $\psi:(0,T)\to X$ be strongly measurable and let $\mu$ be a finite positive Borel measure on $[0,T]$.  Suppose that $\lb \psi, x^*\rb \in L^2(0,T;X)$ for each $x^* \in X^*$.
Then $\int_0^{\cdot}\psi dt \in \gamma(0,T,\mu;X)$
and
\[
\Bigl\| \int_0^{\cdot} \psi \ud t\Bigr\|_{\gamma(0,T,\mu;X)}
\leq \sup_{\|x^*\|\leq 1}\|\lb \psi, x^*\rb \|_{L^2(0,T)} \Big(\int_0^T t \ud \mu(t)\Big)^{\frac12}
\]
\end{lemma}

The integral $\int_0^{\cdot}\psi dt$ is defined as a Pettis integral (see \cite{HNVW1}). Note that the above supremum is finite by the closed graph theorem.

\begin{proof}
Let $\Psi(t) = \int_0^t \psi(s) \ud s$ for $t\in [0,T]$. Let $(\gamma_n')_{n\geq 1}$ be a sequence of standard independent Gaussian random variables on a probability space $(\Omega', \mathbb{F}', \P')$. Let $(\phi_n)_{n\geq 1}$ be an orthonormal basis for $L^2(0,T;\mu)$.
Then for a fixed $\phi\in L^2(0,T;\mu)$ and $n\geq 1$ we can write
\begin{align*}
\int_0^T \Psi(t) \phi(t) \ud \mu(t) & = \int_0^T \psi(s)\int_0^T \one_{(s,T)}(t) \phi(t) \ud \mu(t) \ud s=
\int_0^T \psi(s) \lb \one_{(s,T)}, \phi\rb_{L^2(\mu)} \ud s,
\end{align*}
where the latter is defined as a Pettis integral. By Parseval's identity we have
\begin{align*}
\sum_{n\geq 1}|\lb \one_{(s,T)}, \phi_n\rb_{L^2(\mu)}|^2 = \int_0^T \one_{(s,T)} \ud \mu.
\end{align*}
Therefore, defining $\xi:\Omega\to L^2(0,T)$ by
\begin{align*}
\xi(s) = \sum_{n\geq 1} \gamma_n' \lb \one_{(s,T)}, \phi_n\rb_{L^2(\mu)},
\end{align*}
by the previous estimate, the orthogonality of the $\gamma_n'$ and the three series theorem
(see \cite[p.\ 289]{VTC}) we find
\[\mathbb E\|\xi\|_{L^2(0,T)}^2 = \int_0^T  \int_0^T \one_{(s,T)} \ud \mu \ud s = \int_0^T  t \ud \mu=:C_T.\]
and the series defining $\xi$ converges a.s.\ in $L^2(0,T)$. It follows that
\begin{align*}
\sum_{n\geq 1} \gamma_n'  \int_0^T \Psi(t) \phi_n(t) \ud \mu(t)  = \int_0^T \psi(s) \xi(s) \ud s,
\end{align*}
converges a.s.\ in $X$ and
\begin{align*}
\Big\|\sum_{n\geq 1} \gamma_n'  \int_0^T \Psi(t) \phi_n(t) \ud \mu(t) \Big\|_{X} & = \Big\|\int_0^T \psi(s) \xi(s) \ud s\Big\| \leq C_{\psi}\|\xi\|_{L^2(0,T)},
\end{align*}
where $C_{\psi} = \sup_{\|x^*\|\leq 1}\|\lb \psi, x^*\rb \|_{L^2(0,T)}$.
Taking $L^2(\Omega')$ norms it follows from the definition of the $\gamma$-norm (note that a.s.\ convergence and convergence in $L^p(\Omega';X)$ are equivalent in this setting) that
\[\|\Psi\|_{\gamma(0,T,\mu; X)}\leq C_{\psi} \|\xi\|_{L^2(\Omega';L^2(0,T))} \leq C_{\psi} C_T.\]
\end{proof}

Let $X$ be a Banach space, $(S,\mathcal A,\mu)$ be a $\sigma$-finite measurable space, $1\leq p < \infty$
and $H$ be Hilbert space. Then one can prove that
\begin{equation}\label{fubis}
 L^p(S;\gamma(H,X))\simeq \gamma(H,L^p(S;X)).
\end{equation}
This relation is called  \textit{$\gamma$-Fubini isomorphism} (for more information see \cite{HNVW2}).

\medskip

A Banach space $X$ has  \textit{property $(\alpha)$} if for all $N\in \mathbb N$ and all sequences
$(x_{mn})_{m,n=1}^N \subseteq X$ it holds true that
\[
\mathbb E \Bigl\| \sum_{m,n=1}^Nr_{mn}x_{mn} \Bigr\|^2
\eqsim \mathbb E'\mathbb E''  \Bigl\| \sum_{m,n=1}^Nr_{m}'r_n''x_{mn} \Bigr\|^2,
\]
where $(r_{mn})_{m,n\geq 1}$, $(r_{m}')_{m\geq 1}$ and $(r_{n}'')_{n\geq 1}$ are independent Rademacher
sequences. This property was introduced in a slightly different manner in \cite{Pis1} (see \cite{HNVW2} for the proof of the equivalence).

\paragraph{Sectorial operators and $H^{\infty}$-calculus}
For each $\phi\in (0,\pi)$ let
\[
S_{\phi}:= \{\lambda \in \mathbb C\setminus \{0\}:\arg(\lambda) < \phi\}
\]
be an open sector of angle $\phi$ in the complex plane. A closed and densely defined operator
$A$ on $X$ is  \textit{sectorial of type $\phi \in [0,\pi)$} (see \cite{Haase1}) if $A$ is bijective with dense range,
$\sigma(A) \subseteq \overline {S}_{\phi}$ and for all $\omega \in (\phi,\pi)$
\[
\sup_{\lambda \notin \overline{S}_{\omega}}\|\lambda R(\lambda,A)\|<\infty.
\]

For details on $H^{\infty}$-calculus for sectorial operators we refer the reader to \cite{Haase1,KunWeis04}.

\subsection{Hypotheses and problem formulation}

Consider the following hypothesis.
\begin{itemize}
\item[(A0)] $H$ is a separable Hilbert space. $X$ is a separable Banach space which has UMD and satisfies property $(\alpha)$. $M\in \mathcal M_{\rm var}^{\rm loc}(H)$. The operator $A$ has a bounded $H^\infty$-calculus of angle $<\pi/2$.
\end{itemize}

Consider the following stochastic evolution equation:
\begin{equation}\label{stoceq}
  \begin{cases}
\ud u = (Au(t)+ F(t,u))\ud t + G(t,u)\ud M,
\\
u(0) = u_0,
\end{cases}
\end{equation}
where $A$ is the generator of an analytic $C_0$-semigroup $(S(t))_{t\geq 0}$
on $X$ (see \cite{EN, Par} for details).

We make the following assumption on $F$~and~$G$:
\begin{itemize}
 \item[(A1)] The function $F: \mathbb R_+\times\Omega\times X \to X$ is Lipschitz of
 linear growth uniformly in $\mathbb R_+ \times \Omega$, i.e., there are constants
 $L_F$ and $C_F$ such that for all $t \in \mathbb R_+$, $\omega \in \Omega$ and
 $x,y \in X$
 \begin{align*}
  \|F(t,\omega,x)-F(t,\omega,y)\|_X &\leq L_F\|x-y\|_X,\\
  \|F(t,\omega,x)\|_X &\leq C_F(1+\|x\|_X).
 \end{align*}
 Moreover, for all $x\in X$, $(t,\omega) \mapsto F(t,\omega,x)$ is strongly measurable
 and adapted in $X$.

\smallskip
 \item[(A2)] The function $G: \mathbb R_+\times \Omega\times X\to \mathcal L(H,X)$
 is Lipschitz of linear growth in a $\gamma$-sense uniformly in $\Omega$ and $T$, i.e., there
 are constants $L_G^{\gamma}$ and $C_G^{\gamma}$ s.t.\ for all $b\geq a \geq 0$
 and for all $\phi_1,\phi_2 :\mathbb R_+ \to X$ which are in
 $L^2(\mathbb R_+;X) \cap {\gamma}(\mathbb R_+,[[M]];X)$, a.s.
 \begin{align*}
  \|(G(\cdot,\phi_1)-& G(\cdot,\phi_2))Q_M^{1/2}\|_{\gamma(L^2(a,b, [[M]];H),X)}\\
& \leq L_G^{\gamma}(\|\phi_1-\phi_2\|_{L^2(a,b;X)}+
 \|\phi_1-\phi_2\|_{{\gamma}(a,b,[[M]];X)}),
\\  \|G(\cdot, \omega, \phi_1)& Q_M^{1/2}\|_{\gamma(L^2(a,b,[[M]];H),X)}
\\ & \leq  C_G^{\gamma} (1+\|\phi_1\|_{L^2(a,b;X)}+  \|\phi_1\|_{{\gamma}(a,b,[[M]];X)}).
 \end{align*}
 Moreover, for all $x \in X$, $(t,\omega) \mapsto G(t,\omega,x)$ is $H$-strongly progressively
 measurable.

\smallskip
 \item[(A3)] The initial value $u_0: \Omega \to X$ is strongly $\mathcal F_0$-measurable.

\end{itemize}

In the case $M = W_H$, the above Lipschitz assumptions reduce to the assumptions in \cite{NVW1}.
A key difference with \cite{NVW1} is that the nonlinearities can be defined on interpolation space between $X$ and $D(A)$,
but this cannot be done for general martingales except under additional assumptions on $[[M]]$.

\subsection{Existence and uniqueness result}

For deterministic and stochastic convolutions we will use the following
notations (see \cite{NVW1,VW1}):
\begin{align*}
 S*F(t) &:= \int_0^t S(t-s) F(s)\ud s,\\
 S\diamond G(t) &:= \int_0^t S(t-s) G(s) \ud M_s.
\end{align*}

We call a process $(u(t))_{t\in \mathbb R_+}$ a  \textit{mild solution} of $\eqref{stoceq}$ if

\begin{itemize}
 \item[(i)] $u:\mathbb R_+ \times\Omega \to X$ is strongly measurable and adapted,
 \item[(ii)] for all $t\in \mathbb R_+$, $s \mapsto S(t-s)F(s,u(s))$ is in $L^1(0,t;X)$ a.s.,
 \item[(iii)] for all $t \in \mathbb R_+$, $s \mapsto S(t-s)G(s,u(s))$ is $H$-strongly progressively measurable
and $GQ_M^{1/2}$ is in $\gamma(L^2(0,t,[[M]];H), X)$ a.s.,
 \item[(iv)] for all $t \in \mathbb R_+$, almost surely
 \begin{equation}\label{eq:mild}
   u(t) = S(t)u_0+S*F(\cdot,u)(t) + S\diamond G(\cdot,u)(t).
 \end{equation}
\end{itemize}

\begin{definition}
Fix $b\geq a \geq 0$ and $p \in (1,\infty)$.
\begin{enumerate}
\item We define $V^{p}(a,b,M;X)$ as the space of all
strongly progressively measurable processes $\phi : \mathbb R_+\times \Omega \to X$ for which
\begin{equation*}
 \|\phi\|_{V^p(a,b,M;X)}
 :=
(\mathbb E\|\phi\|_{L^2((a,b);X)}^p)^{\frac1p}+(\mathbb E\|\phi\|_{{\gamma}(a,b,[[M]];X)}^p)^{\frac1p}<\infty
\end{equation*}

\item We define $V(a,b,M;X)$ as the space of all
progressively measurable processes $\phi : \mathbb R_+\times \Omega \to X$ for which almost surely
\begin{equation*}
\|\phi\|_{L^2((a,b);X)} +\|\phi\|_{{\gamma}(L^2(a,b,[[M]]),X)} < \infty.
\end{equation*}
\end{enumerate}
\end{definition}

\begin{remark}\label{localvp}
 Due to the ideal property \eqref{ideal} one can show, that if $\tau$ is a stopping time and
 $\phi\in V^{p}(a,b,M;X)$,
 then $\phi\in V^{p}(a,b,M^{\tau};X)$ as well and $\|\phi\|_{V^{p}(a,b,M^{\tau};X)}\leq
 \|\phi\|_{V^{p}(a,b,M;X)}$.
\end{remark}

The following result is the main existence and uniqueness result:

\begin{theorem}[Existence and uniqueness]\label{exandun3}
Suppose that \textup{(A0)--(A3)} are satisfied. Then there exists a unique solution $U$ in $V(0,T,M;X)$ of \eqref{stoceq}.
\end{theorem}

Moreover, if the unbounded operator $A$ is omitted, then property $(\alpha)$ is not needed in the above result.
\begin{proposition}\label{prop:SEEwithoutA}
Let $H$ be a Hilbert space, $M\in \mathcal M_{\rm var}^{\rm loc}(H)$, $X$ be a UMD space.
Consider the equation:
\begin{equation}\label{stoceq1}
  \begin{cases}
\ud u = F(t,u)\ud t + G(t,u)\ud M,
\\
u(0) = u_0,
\end{cases}
\end{equation}
Suppose that \textup{(A1)--(A3)} are satisfied.
Then there exists a unique solution $U$ in $V(0,T,M;X)$ of \eqref{stoceq}.
\end{proposition}

Unlike in the Brownian case one cannot ensure $L^p(\Omega)$-integrability of the solution even if the initial value is constant in $\Omega$.

\subsection{The fix point argument}

Consider the fixed point operator
\[
L_T(\phi) =[t \mapsto S(t)u_0+S*F(\cdot,\phi)(t) + S\diamond G(\cdot,\phi)(t)].
\]

\begin{proposition}\label{stocheq}
Suppose that \textup{(A0)--(A3)} are satisfied.
If $u_0 \in L^p(\Omega,\mathcal F_0; X)$, $[[M]]_T\in L^{\infty}(\Omega)$,
then the operator $L_T$ is bounded and well-defined on $V^{p}(0,T,M;X)$
and there exists a constant $C_{T,M}$, with $\lim C_{T,M} = 0$
as $T \to 0$ and $T_{M,T}:=\|[[M]]_T\|_{L^{\infty}(\Omega)} \to 0$,
such that for all $\phi_1, \phi_2 \in V^{p}(0,T,M;X)$,
\begin{equation}\label{stocheq1}
  \| L_T(\phi_1)-L_T(\phi_2) \|_{V^{p}(0,T,M;X)}\leq C_{T,M} \|\phi_1-\phi_2\|_{V^{p}(0,T,M;X)}.
 \end{equation}
Moreover, for $T\in (0,1]$, there exists a constant $\tilde{C}$ independent of $T$ and $M$ such that
\begin{equation}\label{eq:CTM}
C_{T,M}\leq \tilde{C} \max\{T^{\frac12},T^{\frac12}_{M,T}\}.
\end{equation}
Furthermore, there is a constant $C\geq 0$, independent of $u_0$, such that
for all $\phi \in V^{p}(0,T,M;X)$
\begin{equation}\label{stocheq2eq}
 \|L_T(\phi)\|_{V^{p}(0,T,M;X)} \leq C(1+\|u_0\|_{L^p(\Omega;X)}) + C_{T,M} \|\phi\|_{V^{p}(0,T,M;X)},
\end{equation}
and $L_T(\phi)$ has a continuous version and
\begin{equation}\label{eq:stochcont}
 \|L_T(\phi)\|_{L^p(\Omega;C([0,T];X))} \leq C(1+\|u_0\|_{L^p(\Omega;X)}) + C_2 \|\phi\|_{V^{p}(0,T,M;X)}.
\end{equation}
\end{proposition}

\begin{proof}
Actually the assumptions even yields that $\{S(t):t\geq 0\}$ is $R$-bounded and hence $\gamma$-bounded by some constant $N$ (see \cite[Theorem 2.20 and 12.8]{KunWeis04}). In particular $\|S(t)\|\leq N$ for all $t\geq 0$.

Let $Y = \gamma(\mathbb R;X)$.
For the proof we use the following dilation result for the semigroup $S$ from \cite{FW}.
By the boundedness of the $H^{\infty}$-calculus with angle $<\frac{\pi}{2}$ yields that there exist
$J\in \mathcal L(X,Y)$, $P \in \mathcal L(Y)$ and $(\tilde{S}(t))_{t \in \mathbb R}\subseteq \mathcal L(Y)$
such that
\begin{itemize}
 \item[(i)] There are $c_J, C_J >0$ such that for all $x \in X$, one has $c_J \|x\| \leq \|Jx\|\leq C_J\|x\|$.
 \item[(ii)] $P$ is a projection onto $\text{ran }J$.
 \item[(iii)] $(\tilde{S}(t))_{t \in \mathbb R}$ is a strongly continuous group on $Y$ with
 $\|\tilde S(t)y\| = \|y\|$ for all $y \in Y$.
\item[(iv)] For all $t \geq 0$ it holds true that $JS(t) = P\tilde S(t)J$.
\end{itemize}
This dilation will be used to derive continuity of the stochastic convolution in a similar way as in \cite{HausSei}. Moreover, we use it to obtain estimates in the $\gamma$-norm.

Notice that by \cite[Lemma~2.3]{VW1} $Y$ is a UMD space. Also notice that since $X$ has property $(\alpha)$
then according to \cite[Theorem~3.18]{HK} family $(\tilde S(t))_{t \in \mathbb R}$ is $\gamma$-bounded
by some constant $\alpha_X$. Now we will proceed prove in 4 steps. Fix $T\geq 0$. Let $C_P = \|P\|$.

\medskip

 \textit{Step 1: Estimating the initial value part.} By the strong continuity and uniform boundedness of $S$ we derive:
\begin{equation}\label{step1continuity}
\|s\mapsto S(s)u_0\|_{L^2(0,T;X)}\leq T^{\frac12}\|s\mapsto S(s)u_0\|_{C([0,T];X)} \leq N T^{\frac12} \|u_0\|
\end{equation}
By the $\gamma$-boundedness
of $(S(t))_{t \in \mathbb R}$ and \cite[Proposition~4.11]{KW}:
\begin{align*}
\|s\mapsto S(s)u_0\|_{\gamma(0,T,[[M]];X)}\leq N \|s\mapsto u_0\|_{\gamma(0,T,[[M]];X)} = [[M]]_{T}^{1/2} \|u_0\|.
\end{align*}

 \textit{Step 2. Estimating the deterministic part.} We proceed in two steps.

(a): For fixed $\omega \in \Omega$ and $\psi \in L^2(0,T;X)$
we estimate the $L^p(0,T;X)$- and $\gamma(0,T,[[M]];X)$-norms
of $S*\psi$. One has
\begin{align}\label{contusint}
  \|S*\psi\|_{L^2(0,T;X)}=T^{\frac12}\|S*\psi\|_{C([0,T];X)}  \leq T N \|\psi\|_{L^2(0, T;X)},
\end{align}
where the continuity of $S*\psi$ is simple to check (see \cite[Corollary 4.2.4]{Lun}).

By the representation of $S$ as a group, the ideal property \eqref{ideal} and \cite[Proposition~4.11]{KW} we find that
\begin{align*}
  \| S*\psi\|_{\gamma(0,T,[[M]];X)} &\leq  \frac{1}{c_J} \| (J S)*\psi\|_{\gamma(0,T,[[M]];X)}
\\ & =  \frac{1}{c_J}  \Bigl\|P\tilde S(\cdot)\int_0^{\cdot} \tilde S(-s) J\psi(s)\ud s\Bigr\|_{\gamma(0,T,[[M]];X)}
\\ & \leq \frac{C_P\alpha_X}{c_J}   \Bigl\|\int_0^{\cdot} \tilde S(-s) J\psi(s)\ud s\Bigr\|_{\gamma(0,T,[[M]];X)}
\\ & \leq \frac{C_JC_P\alpha_X N}{c_J}  T^{\frac12}[[M]]_T^{\frac12} \|\psi\|_{L^2(0,T;X)},
\end{align*}
where in the last step we used Lemma \ref{gammameasure} and $\int_0^T t \ud [[M]]_t\leq T [[M]]_{T}$.

Now let $\Psi \in V^p(0,T,M;X)$. Then by applying the inequalities above to the paths
$\Psi(\cdot, \omega)$ one easily obtains that $S* \Psi \in V^p(0,T,M;X)$ and
\begin{equation*}
 \|S* \Psi\|_{V^p(0,T,M;X)}\leq C_T^{1}
 \|\Psi\|_{V^p(0,T,M;X)},
\end{equation*}
where
\[C_T^{1} = T N  + \frac{C_JC_P\alpha_X N}{c_J}  T^{\frac12}T_{M,T}^{\frac12}\ \ \text{and}  \ \ T_{M,T}:=\|[[M]]_T\|_{L^{\infty}(\Omega)}.\]

\medskip

(b): Let $\phi_1, \phi_2 \in V^p(0,T,M;X)$. Since $F$ is of linear growth,
$F(\cdot,\phi_1)$ and $F(\cdot,\phi_2)$ have a continuous version and belong to $V^p(0,T,M;X)$. Since $F$ is Lipschitz
in its $X$-variable, we deduce that
$S*F(\cdot,\phi_1)$ and $ S* F(s,\phi_2)$ are
in $V^p(0,T,M;X)$ and
\begin{align*}
\|S*F(s,\phi_1) - S* F(s,\phi_2)\|_{V^p(0,T,M;X)}
&\leq C_T^{1} \|F(\cdot,\phi_1)-F(\cdot,\phi_1)\|_{V^p(0,T,M;X)}\\
&\leq  C_T^{1} L_F \|\phi_1-\phi_2\|_{V^p(0,T,M;X)}.
\end{align*}

  \textit{Step 3. Estimating the stochastic part.}

(a): Let $\Psi:[0,T]\times \Omega \to \mathcal L(H,X)$
be scalarly strongly progressively measurable and suppose that
$\Psi Q_M^{1/2}$ is in $L^p(\Omega;\gamma(L^2(0,T,[[M]];H),X))$.
Then by \cite[Proposition~4.11]{KW}  and Theorem \ref{main} for each $t\in [0,T]$,
\[\zeta_M(t):=\int_0^t S(t-s)\Psi(s) \ud M(s)\]
is well-defined.
Now we estimate $\zeta_M$ pathwise in the space of continuous functions. As before one sees that $\int_0^{\cdot}\tilde S^{-1} J \Psi \ud M$ is well-defined and is a.s.\ continuous (here we use the fact that $Y$ is a UMD space and $\tilde{S}$ is $\gamma$-bounded). Therefore, by the representation of $S$ as a group it follows that we can write
\begin{align}\label{eq:represtochconv}
J \zeta_M(t) = P \tilde S(t) \int_0^{t}\tilde S(-s) J \Psi(s) \ud M
\end{align}
Since $P \tilde S(t)$ is strongly continuous, the continuity follows since $J$ is an isomorphic embedding.
Moreover, by Theorem \ref{main} and \cite[Proposition~4.11]{KW},
\begin{align*}
\|\zeta_M\|_{L^p(\Omega;L^2(0,T;X))} & \leq  T^{\frac12} \|\zeta_M\|_{L^p(\Omega;C([0,T];X))}
\\ & \leq  \frac{T^{\frac12}}{c_J} \Big\|t\mapsto P \tilde S(t) \int_0^t \tilde S^{-1}(s) J \Psi(s) \ud M(s)\Big\|_{L^p(\Omega;C([0,T];X))}
\\ & \leq  \frac{T^{\frac12} C_P N}{c_J} \Big\|t\mapsto \int_0^t \tilde S^{-1}(s) J \Psi(s) \ud M(s)\Big\|_{L^p(\Omega;C([0,T];X))}
\\ & \leq  \frac{T^{\frac12} C_P N C_{p,X}}{c_J} \|\tilde S^{-1} J \Psi Q_M^{\frac12}\|_{L^p(\Omega;\gamma(0,T,[[M]];H, Y))}
\\ & \leq  \frac{T^{\frac12} C_P N^2 C_{p,X} C_J}{c_J} \|\Psi Q_M^{\frac12}\|_{L^p(\Omega;\gamma(0,T,[[M]];H, X))}.
\end{align*}

To estimate the $\gamma (L^2(0,T,[[M]]), X)$-norm of $\zeta_M$ we can again use the representation \eqref{eq:represtochconv} and use \cite[Proposition~4.11]{KW} and the ideal property to estimate
\begin{align*}
\|\zeta_M\|_{L^p(\Omega;\gamma(0,T,[[M]];X))}
 & \leq \frac{1}{c_J} \Big\|t\mapsto P \tilde S(t) \int_0^t \tilde S^{-1}(s) J \Psi(s) \ud M(s)\Big\|_{L^p(\Omega;\gamma(0,T,[[M]];Y))}
\\ & \leq \frac{C_P N}{c_J} \Big\|t\mapsto \int_0^t \tilde S^{-1}(s) J \Psi(s) \ud M(s)\Big\|_{L^p(\Omega;\gamma(0,T,[[M]];Y))}
\end{align*}
To estimate the last term recall from Theorem \ref{rem:intfortimechange}
\begin{align*}
\tilde{\zeta}_M(t):=  \int_0^{t} \tilde S^{-1} J \Psi \ud M = \int_0^{[[M]]_t} (\one_{(0,T)}\tilde S^{-1} J \Psi Q_M^{\frac12}) \circ \tau \ud W_H=:\tilde{\zeta}_{W_H}([[M]]_t).
\end{align*}
Using this representation, we find
\begin{align*}
\|\tilde{\zeta}_M\|_{L^p(\Omega;\gamma(0,T,[[M]];Y))} &= \|\tilde{\zeta}_{W_H}\circ[[M]]\|_{L^p(\overline{\Omega};\gamma(0,T,[[M]];Y))}
\\ &\stackrel{(*)}{=}\|\tilde{\zeta}_{W_H}\|_{L^p(\overline{\Omega};\gamma(0,[[M]]_T;Y))}
\\ & \leq \|\mathbb E\Big(\tilde{\zeta}_{W_H}(T_{M,T})|\mathcal G_{t}\Bigr)\|_{L^p(\overline{\Omega};\gamma(0,T_{M,T};Y))}
\\ & \stackrel{(i)}{\leq} C_p\|\mathbb E\Big(\tilde{\zeta}_{W_H}(T_{M,T})|\mathcal G_{t}\Bigr)\|_{\gamma(0,T_{M,T};L^p(\overline{\Omega};Y))}
\\ & \stackrel{(ii)}{\leq} C_p \gamma(\overline{\mathcal E}_{p,T}) \|\tilde{\zeta}_{W_H}(T_{M,T})\|_{\gamma(0,T_{M,T};L^p(\overline{\Omega};Y))}
\\ & = C_p \gamma(\overline{\mathcal E}_{p,T}) T_{M,T}^{\frac12} \|\tilde{\zeta}_{W_H}(T_{M,T})\|_{L^p(\overline{\Omega};Y)}
\\ & \stackrel{(iii)}{\leq} C_p \gamma(\overline{\mathcal E}_{p,T}) T_{M,T}^{\frac12} C_{p,X} \|(\one_{(0,T)}\tilde S^{-1} J \Psi Q_M^{\frac12}) \circ \tau\|_{L^p(\Omega;\gamma(\mathbb R_+;Y))}
\\ & \stackrel{(*)}{\leq} C_p \gamma(\overline{\mathcal E}_{p,T}) T_{M,T}^{\frac12} C_{p,X} \|\tilde S^{-1} J \Psi Q_M^{\frac12}\|_{L^p(\Omega;\gamma(0,T,[[M]];Y))}
\end{align*}
In $(*)$ we used Lemma \ref{timechange} and $[[M]]_{\tau_s} = s$.
In $(i)$ we used \eqref{fubis}. In $(ii)$ we used Lemma \ref{vectorstein} for conditional expectations on $L^p(\overline{\Omega};Y)$ and \cite[Proposition~4.11]{KW}. In $(iii)$ we used \eqref{eqwiener}.
Therefore, combining both estimates it follows that
\begin{align*}
\|\zeta_M\|_{L^p(\Omega;\gamma(0,T,[[M]];X))}& \leq \frac{C_P N}{c_J}  \|\tilde{\zeta}_M\|_{L^p(\Omega;\gamma(0,T,[[M]];Y))}
\\ & \leq \frac{C_P N^2 C_J}{c_J}  C_p \gamma(\overline{\mathcal E}_{p,T}) T_{M,T}^{\frac12} C_{p,X}  \|\Psi Q_M^{\frac12}\|_{L^p(\Omega;\gamma(0,T,[[M]];X))}
\end{align*}
where in the last step we argue as below \eqref{eq:represtochconv}.

\medskip
Combining these estimates we conclude that
\begin{equation}\label{combstep3a}
\|\zeta_M\|_{V^p(0,T,M;Y)} \leq  C_T^{2}
 \|\Psi Q_M^{1/2}\|_{L^p(\Omega;\gamma (L^2(0,T,[[M]];H), X))},
\end{equation}
where
\[  C_T^{2}:= \frac{T^{\frac12} C_P N^2 C_{p,X} C_J}{c_J}  + \frac{C_P N^2 C_J}{c_J}  C_p \gamma(\overline{\mathcal E}_{p,T}) T_{M,T}^{\frac12} C_{p,X}  \]

\medskip

 (b): Let $\phi_1,\phi_2 \in V^p(0,T,M;X)$. It follows from the assumption on $G$ that
$S\diamond G(\cdot,\phi_1), S\diamond G(\cdot,\phi_2)$ have a continuous version and are $V^p(0,T,M;X)$ and
 \begin{align*}
   \Bigl\|&S\diamond G(\cdot,\phi_1)-S\diamond G(\cdot,\phi_2)\Bigr\|_{V^p(0,T,M;X)}
 \\ &\leq C_T^2  \|(G(\cdot,\phi_1)- G(\cdot,\phi_2)) Q_M^{1/2}\|_{L^p(\Omega;\gamma (L^2(0,T,[[M]];H), X))}\\
 &\leq L_G^{\gamma}C_T^2 \|\phi_1-\phi_2\|_{V^p(0,T,M;X)}.
 \end{align*}

  \textit{Step 4: Collecting the estimates.} It follows from the previous steps that
 $L_T$ is well-defined on $V^p(0,T,M;X)$ and
 \begin{equation}\label{mainest}
\|L_t(\phi_1)-L_t(\phi_2)\|_{V^p(0,T,M;X)}\leq C_{T,M}\|\phi_1-\phi_2\|_{V^p(0,T,M;X)},
 \end{equation}
where $C_{T,M} = L_F C_T^{1} + L_G C_T^{2}$ and one can check that \eqref{eq:CTM} holds.

To prove \eqref{stocheq2eq} one has to apply \eqref{mainest} and the fact that
for some positive constant $C$ it holds true that
\[
\|L_T(0)\|_{V^p(0,T,M;X)}\leq C(1+(\mathbb E\|u_0\|_{X}^p)^{\frac 1p}).
\]
The final continuity statement and \eqref{eq:stochcont} follows from the previous steps.
\end{proof}

\subsection{Existence and uniqueness when the variation is small}

\begin{theorem}[Existence and uniqueness]\label{exandun1}
Suppose that \textup{(A0)--(A3)} are satisfied and $\|[[M]]_T\|_{L^{\infty}(\Omega)}\leq (2\tilde C)^{-2}$, where $\tilde{C}$ is as in \eqref{eq:CTM}.
If $u_0 \in L^p(\Omega,\mathcal F_0; X)$,
then there exists a unique solution $U$ in $V^{p}(0,T,M;X)$ of \eqref{stoceq}.
Moreover, there exists a nonnegative
constant~$C$, independent of $u_0$ but depending on $T\vee 1$ and $\tilde{C}$ such that
\begin{equation}\label{exandunest}
 \|U\|_{V^p(0,T,M;X)}\leq C(1+(\mathbb E\|u_0\|_{X}^p)^{\frac 1p}).
\end{equation}
Furthermore, $U$ has a continuous version and there exists a constant $D$ independent of $u_0$ but depending on $T\vee 1$ and $\tilde{C}$ such that
\begin{equation}\label{exandunestcont}
\|U\|_{L^p(\Omega;C([0,T];X))}\leq C(1+(\mathbb E\|u_0\|_{X}^p)^{\frac 1p})
\end{equation}
\end{theorem}
\begin{proof} By Proposition \ref{stocheq} one can find $t \in [0,T\wedge 1]$, independent of $u_0$,
such that
$t \leq \|[[M]]_T\|_{L^{\infty}(\Omega)}$, so $C_{t,M}\leq \frac12$.
It follows from \eqref{stocheq1} and the Banach fixed point argument that $L_t$
has a unique fixed point $U \in V^p(0,t,M;X)$. This gives us a continuous progressively measurable process
$U:[0,t]\times \Omega \to X$ such that a.s.\ for all $s\in[0,t]$,
\begin{equation*}
 U(s) = S(s)u_0 + S* F(\cdot,U)(s) + S\diamond G(\cdot,U)(s).
\end{equation*}
Note that \eqref{stocheq2eq} implies that
\[
\|U\|_{V^p(0,t,M;X)} \leq C(1+(\mathbb E \|u_0\|_X^p)^{\frac 1p}) + C_{t,M}\|U\|_{V^p(0,t,M;X)},
\]
and since $C_{t,M} \leq 1/2$
\begin{equation}\label{exandunest2}
\|U\|_{V^p(0,t,M;X)} \leq 2C(1+(\mathbb E \|u_0\|_X^p)^{\frac 1p}).
\end{equation}
and by $U = L_t(U)$, \eqref{eq:stochcont} and \eqref{exandunest2} we find
\begin{equation}\label{exandunestconthelp}
\|U\|_{L^p(\Omega;C([0,t];X))} \leq C_2(1+\|u_0\|_{L^p(\Omega;X)}).
\end{equation}
Thanks to a standard induction argument one easily constructs a solution on each of intervals $[t,2t],\ldots,[nt,T]$, where $n = [\frac Tt]$. This solution $U$ on $[0,T]$
is the solution of \eqref{stoceq}. Moreover, according to \eqref{exandunest2}, \eqref{exandunestconthelp} and the induction one deduces~\eqref{exandunest} and \eqref{exandunestcont}.

For small $t\in [0,T]$ uniqueness on $[0,t]$ follows from the uniqueness of the fixed point of
$L_t$ in $V^p(0,t,M;X)$, and uniqueness on $[0,T]$ follows from the induction argument.
\end{proof}

\begin{lemma}\label{u1=u2}
Suppose that \textup{(A0)--(A3)} are satisfied both for $M$ and $N$ and
\[\|[[M]]_T\|_{L^{\infty}(\Omega)},\|[[N]]_T\|_{L^{\infty}(\Omega)}<\frac 1{4\tilde C^2}.\]
Let $U_1 \in V^p(0,T,M;X)$, $U_2 \in V^p(0,T,N;X)$
be the solutions of \eqref{stoceq} with initial values $u_1,u_2 \in L^p(\Omega, \mathcal F_0;X)$
and cylindrical martingales $M,N$ respectively. Finally suppose that $M\equiv N$ a.s.\ on the set $\{u_1=u_2\}$.
Then a.s.\ $U_1 \equiv U_2$ on $\{u_1=u_2\}$.
\end{lemma}

\begin{proof}
 Let $\Gamma = \{u_1=u_2\}$. Since $U_2 \in V^p(0,T,N;X)$, then $U_2 \mathbf 1_{\Gamma}\in V^p(0,T,M;X)$,
 because $M$ and $N$ coincides on $\Gamma$.
 Consider small $t$ as in the beginning of the proof of Theorem~\ref{exandun1}.
 Since $\Gamma$ is $\mathcal F_0$-measurable
\begin{align*}
 \|U_1 \mathbf 1_{\Gamma} - U_2 \mathbf 1_{\Gamma}\|_{V^p(0,t,M;X)}
 &=\|L_t(U_1)\mathbf 1_{\Gamma}- L_t(U_2) \mathbf 1_{\Gamma}\|_{V^p(0,t,M;X)}\\
 &= \|L_t(U_1\mathbf 1_{\Gamma})\mathbf 1_{\Gamma}- L_t(U_2\mathbf 1_{\Gamma}) \mathbf 1_{\Gamma}\|_{V^p(0,t,M;X)}\\
 &\leq C_{t,M}\|U_1 \mathbf 1_{\Gamma} - U_2 \mathbf 1_{\Gamma}\|_{V^p(0,t,M;X)},
\end{align*}
therefore almost surely $U_1|_{[0,t]\times \Gamma}\equiv U_2|_{[0,t]\times \Gamma}$.

To extend this result to the whole interval $[0,T]$ one has to apply the same induction argument as
in the end of the proof of Theorem~\ref{exandun1}.
\end{proof}

Let $b\geq a\geq 0$. We say that $\phi$ is  \textit{locally in $V^p(a,b,M;X)$}
(or simply $\phi \in V^p_{\rm loc}(a,b,M;X)$) if there exists a sequence of increasing stopping times $(\tau_n)_{n\geq 1}$
such that $\tau_n\nearrow\infty$ a.s.\ and $\phi \in V^p(a,b,M^{\tau_n};X)$ for each $n>0$.
It is evident by~Remark~\ref{localvp} that $\phi \in V^p(a,b,M;X)$ implies $\phi \in V^p_{\rm loc}(a,b,M;X)$.
Obviously $V^{p}_{\rm loc}(a,b,M;X)\subseteq V(a,b,M;X)$.

 \begin{lemma}\label{forun2}
Suppose that \textup{(A0)--(A3)} are satisfied. Let $\tau$ be a stopping time such that
  $ \|[[M^{\tau}]]_T\|_{L^{\infty}(\Omega)}<\frac 1{4\tilde C^2}$,
  $u_0 \in L^p(\Omega,\mathcal F_0; X)$, $v_0\in L^0(\Omega,\mathcal F_0; X)$, $U_M\in V^p_{\rm loc}(0,T,M;X)$, $U_{M^{\tau}}\in V^p(0,T,M^{\tau};X)$
 be solutions of \eqref{stoceq} with cylindrical martingales $M$, $M^{\tau}$ and initial values $u_0$, $v_0$ respectively.
Then on the set $\{u_0 = v_0\}$ and on the interval $[0,\tau\wedge T]$, one has $U_M\equiv  U_{M^{\tau}}$  a.s.
 \end{lemma}

 \begin{proof} Let $\Gamma = \{u_0=v_0\}$.
  Consider the localizing sequence $\{\tau_n\}_{n\geq 1}$ of stopping times for $U_M$, so
  $U_M \in V^p(0,T,M^{\tau_n};X)$ for each $n\geq 1$. Then by \eqref{ideal} both $U_M\mathbf 1_{[0,\tau\wedge\tau_n]}\mathbf 1_{\Gamma}$
  and $U_{M^{\tau}}\mathbf 1_{[0,\tau\wedge\tau_n]}\mathbf 1_{\Gamma}$ are in $V^p(0,T,M^{\tau\wedge\tau_n};X)$
  for each fixed $n\geq 1$. Let $t$ be such that
  $t^{\beta}< \|[[M^{\tau}]]_T\|_{L^{\infty}(\Omega)}^{\frac 12}$. Then for each fixed $n\geq 1$
 \begin{align*}
 \|U_M & \mathbf 1_{[0,\tau\wedge\tau_n]}\mathbf 1_{\Gamma}
 -  U_{M^{\tau}}\mathbf 1_{[0,\tau\wedge\tau_n]}\mathbf 1_{\Gamma} \|_{V^p(0,t,M^{\tau\wedge\tau_n};X)}
 \\ & = \|L_t(U_M)\mathbf 1_{[0,\tau\wedge\tau_n]}\mathbf 1_{\Gamma} -
 L_t(U_{M^{\tau}})\mathbf 1_{[0,\tau\wedge\tau_n]}\mathbf 1_{\Gamma}\|_{V^p(0,t,M^{\tau\wedge\tau_n};X)}\\
& = \|L_t(U_M\mathbf 1_{[0,\tau\wedge\tau_n]}\mathbf 1_{\Gamma})\mathbf 1_{[0,\tau\wedge\tau_n]}\mathbf 1_{\Gamma} -
 L_t(U_{M^{\tau}}\mathbf 1_{[0,\tau\wedge\tau_n]}\mathbf 1_{\Gamma})
 \mathbf 1_{[0,\tau\wedge\tau_n]}\mathbf 1_{\Gamma}\|_{V^p(0,t,M^{\tau\wedge\tau_n};X)}\\
& \leq C_{t,M^{\tau\wedge\tau_n}} \|U_M\mathbf 1_{[0,\tau\wedge\tau_n]}\mathbf 1_{\Gamma} -
 U_{M^{\tau}}\mathbf 1_{[0,\tau\wedge\tau_n]}\mathbf 1_{\Gamma}\|_{V^p(0,t,M^{\tau\wedge\tau_n};X)},
 \end{align*}
hence $U_M\mathbf 1_{[0,\tau\wedge\tau_n\wedge t]}\mathbf 1_{\Gamma} \equiv
 U_{M^{\tau}}\mathbf 1_{[0,\tau\wedge\tau_n\wedge t]}\mathbf 1_{\Gamma}$ a.s.\ Letting $n$ to infinity
 yields $U_M \equiv U_{M^{\tau}}$ on $[0,t\wedge \tau]$ for a.a.\ $\omega \in \Gamma$. Now by induction and the same technique
 as in Lemma \ref{u1=u2} one obtains the required result.
 \end{proof}

\subsection{Proof of the main existence and uniqueness result}

We first proof Theorem \ref{exandun3} under additional integrability assumptions on the initial value.

\begin{theorem}[Existence and uniqueness for integrable initial values]\label{exandun2}
 Suppose that \textup{(A0)--(A3)} are satisfied.
 If $u_0 \in L^p(\Omega,\mathcal F_0; X)$,
 then there exists a unique solution $U$ in $V^{p}_{\rm loc}(0,T,M;X)$ of \eqref{stoceq}.
\end{theorem}

\begin{proof}
 By Proposition \ref{stocheq} one can find $n\in \mathbb N$ large enough so that
 $\frac{T}{2^n} \leq \frac 1{4\tilde C^2}$ and $T\leq 2^n$.

Let $\rho = \tau_{\frac{1}{4\tilde C^2}}$, where $\tau_s$ is a stopping time introduced
in \eqref{proofmain1}.
Consider equation \eqref{stoceq} with the cylindrical martingale $M^{\rho}$ instead of $M$.
It follows from \eqref{mainest} that $C_{\frac T{2^n},M^{\rho}}<\frac 12$.
Using the Banach fixed point argument one derives that $L_{\frac T{2^n}}$
 has a unique fixed point $U_n \in V^p(0,{\frac T{2^n}},M^{\rho};X)$. This gives us
 a continuous progressive measurable process
 $U_n:[0,{\frac T{2^n}}]\times \Omega \to X$ such that for almost all $\omega \in \Omega$
 for all $s\in[0,{\frac T{2^n}}]$,
 \begin{equation*}
  U_n(t) = S(s)u_0 + S* F(\cdot,U_n) + \int_0^t S(t-s)G(s,U_n)\ud M^{\rho}_s.
 \end{equation*}
 Note that \eqref{stocheq2eq} implies that
 \[
 \|U_n\|_{V^p(0,\frac T{2^n},M^{\rho};X)} \leq C(1+(\mathbb E \|u_0\|_X^p)^{\frac 1p}) +
 C_{\frac T{2^n},M^{\rho}}\|U_n\|_{V^p(0,\frac T{2^n},M^{\rho};X)},
 \]
 and since $C_{\frac T{2^n},M^{\rho}} < \frac 12$
 \begin{equation}\label{exandunest1}
 \|U_n\|_{V^p(0,\frac T{2^n},M^{\rho};X)} \leq 2C(1+(\mathbb E \|u_0\|_X^p)^{\frac 1p}).
  \end{equation}

To go on with a standard induction argument  on each of intervals
$[\frac {(k-1)T}{2^n},\frac {kT}{2^n}]$ for $k\in \{2, \ldots, 2^n\}$ we introduce the following stopping times
for $k\in \{1, \ldots, 2^n\}$

\begin{equation}\label{rho_nk}
\rho_{nk} = \begin{cases}
\frac {(k-1)T}{2^n}+\inf\{t\geq 0:[[M]]_{t+\frac {(k-1)T}{2^n}}-[[M]]_{\frac {(k-1)T}{2^n}}>\frac {T}{2^n}\},
&\text{on the set}\; A ;
\\
\infty,&\text{on the set}\; \Omega\setminus A.
\end{cases}
\end{equation}
Here $A = \{0\leq \frac {T}{2^n}<[[M]]_{\infty}-[[M]]_{\frac {(k-1)T}{2^n}}\}$.
As one can notice, $\rho_{n1} = \rho$. By \cite[Theorem~I.18]{Prot}
and since the minimum of stopping times is a stopping time,
$M^{\rho_{n1}\wedge \ldots\wedge \rho_{nk}} \in \mathcal M_{\rm var}(H)$. Fix $k>1$. Then one can construct
solution of equation \eqref{stoceq} on the interval $[\frac {(k-1)t}{2^n},\frac {kt}{2^n}]$
with the cylindrical martingale
$M^{\rho_{n1}\wedge \ldots\wedge \rho_{nk}}$ instead of $M$ and with the initial value,
obtained on the previous interval $[\frac {(k-2)T}{2^n},\frac {(k-1)T}{2^n}]$.

Thanks to \eqref{exandunest1}, \eqref{eq:stochcont} and a standard induction argument
one may construct a solution on each of intervals
$[\frac T{2^n},2\frac T{2^n}],\ldots,[(2^n-1)\frac T{2^n},T]$. This solution $U_n$ on $[0,T]$
is the solution of \eqref{stoceq} with $M$ replaced by $M^{\rho_n}$.

Define $\rho_n := \rho_{n1}\wedge\ldots\wedge\rho_{n2^n}$ for each $n\in \mathbb N$.
Then by the fixed point argument, the induction argument and Lemma \ref{forun2},
$U_n = U_m$ on $[0,\rho_n\wedge\rho_m\wedge T]$ for all $m,n\in \mathbb N$.
Consequently, since $\rho_n \nearrow\infty$ a.s.\ there exists $U:[0,T]\times \Omega\to X$ such that
$U=U_n$ on $[0,\rho_n\wedge T]$ for each $n \geq 1$.

Now one has to show that $U$ is a solution of \eqref{stoceq}. First of all notice that for each fixed $t\geq 0$ we know that $(U-U_n)\mathbf 1_{t\leq \rho_n}=0$. Consequently $(S(t-s)G(s,U_n) - S(t-s)G(s,U))\mathbf 1_{t\leq\rho_n} = 0$. Then for each fixed $t\geq 0$ according to Corollary \ref{cor:whenintis0} one has that a.s.\ on $\{t\leq \rho_n\}$
\begin{align*}
 U(t) =U_n(t)&= S(t)u_0 + \int_0^t S(t-s) F(s,U_n)\ud s + \int_0^t S(t-s) G(s,U_n) \ud M_s\\
 &= S(t)u_0 + \int_0^t S(t-s) F(s,U)\ud s + \int_0^t S(t-s) G(s,U) \ud M_s.
\end{align*}
So, letting $n$ to infinity one can show that for each fixed $t\geq 0$ a.s.\
\begin{align*}
 U(t) = S(t)u_0 + \int_0^t S(t-s) F(s,U)\ud s + \int_0^t S(t-s) G(s,U) \ud M_s.
\end{align*}

Now assume that $V \in V^{p}_{\rm loc}(0,T,M;X)$ is another solution of \eqref{stoceq}.
Then by Lemma~\ref{forun2}, $V = U_n$ on $[0, \rho_{n1}\wedge \frac{T}{2^n}]$ for all $n\geq 1$. According to \eqref{exandunestcont} $U_n(\frac{T}{2^n})\in L^p(\Omega,\mathcal F_{\frac T{2^n}};X)$, so again by Lemma~\ref{forun2} on the set $\{\rho_{n1}\geq \frac{T}{2^n}\}$ $V = U_n$ on $[\frac{T}{2^n},\rho_{n2}\wedge \frac{2T}{2^n}]$ for all $n\geq 1$ (here we start our solutions from the point $\frac T{2^n}$). Continuing this procedure for $k=3,\ldots,2^n$ we have that $V = U_n$ on $[0,\rho_n\wedge T]$ for all positive $n$. But since $U = U_n$ on $[0,\rho_n\wedge T]$ for all $n\geq 1$, $V=U$ on $[0,\rho_n\wedge T]$, therefore on whole $[0,T]$.
\end{proof}

Finally we can prove Theorem \ref{exandun3} for general initial values.

\begin{proof}[Proof of Theorem \ref{exandun3}]
The structure of the proof is the same as in \cite[Theorem~7.1]{NVW1}.
 To prove existence define a sequence $(u_n)_{n \geq 1}$ in $L^p(\Omega, \mathcal F_0;X)$
 in the following way:
 \[
 u_n = \mathbf 1_{\|u_0\|\leq n}u_0.
 \]
Then by Theorem \ref{exandun2} for each $n \geq 1$ there exists a unique solution
$U_n \in V^p_{\rm loc}(0,T,M;X)$ of \eqref{stoceq} with initial value $u_n$.
By Lemma \ref{forun2} one can define $U :[0,T]\times \Omega \to X$ as $U(t) = \lim_{n\to \infty} U_n(t)$ if this limit exists and $0$ otherwise. Then $U$ is strongly progressive measurable,
and almost surely on $\{\|u_0\|\leq n\}$
for all $t\in [0,T]$ we have that $U(t) = U_n(t)$. Consequently, $U \in V(0,T,M;X)$
and one can check it is a solution of \eqref{stoceq}.

For uniqueness of the solution we will need the stopping times constructed
in the proof of Theorem \ref{exandun2}. Let $U,V\in V(0,T,M;X)$ be two solutions of \eqref{stoceq}.
First of all fix $n \geq 1$ and prove that $U\mathbf 1_{\|u_0\|\leq n} = V\mathbf 1_{\|u_0\|\leq n}$.
Let $U_n = U\mathbf 1_{\|u_0\|\leq n}$, $V_n = V\mathbf 1_{\|u_0\|\leq n}$.
Obviously $U_{n}$ and $V_{ n}$ are solutions of
\eqref{stoceq} with initial value $u_0\mathbf 1_{\|u_0\|\leq n}$.

Let $k$ be large enough such that
$\frac T{2^k} <\frac 1{2\tilde C}$. For each $l\in \mathbb N$ define a stopping time
$\sigma_{nl}$ as follows:
\begin{align*}
  \sigma_{nl} = \inf\{s\in [0,T]:&\|U_{n}\|_{L^2((0,s);X)}
+\|U_{n}\|_{{\gamma}(L^2(0,s,[[M]]),X)}\\
+&\|V_{n}\|_{L^2((0,s);X)}
+\|V_{n}\|_{{\gamma}(L^2(0,s,[[M]]),X)}\geq l\}.
\end{align*}

Then $U_n\mathbf 1_{[0,\sigma_{nl}]},V_n\mathbf1_{[0,\sigma_{nl}]} \in V^p(0,\frac T{2^k},M;X)$. Define $(\rho_{km})_{1\leq m\leq 2^k}$ in the same way as in \eqref{rho_nk}.
For fixed $k$ one has the following
\begin{align*}
 \|&U_n\mathbf 1_{[0,\sigma_{nl}\wedge\rho_{k1}]} -
 V_n\mathbf1_{[0,\sigma_{nl}\wedge\rho_{k1}]}\|_{V^p(0,\frac T{2^k},M;X)}
\\ &=\|U_n\mathbf 1_{[0,\sigma_{nl}\wedge\rho_{k1}]} -
 V_n\mathbf1_{[0,\sigma_{nl}\wedge\rho_{k1}]}\|_{V^p(0,\frac T{2^k},M^{\rho_{k1}};X)}\\
 &=\|L_t(U_n)\mathbf 1_{[0,\sigma_{nl}\wedge\rho_{k1}]} -
 L_t(V_n)\mathbf1_{[0,\sigma_{nl}\wedge\rho_{k1}]}\|_{V^p(0,\frac T{2^k},M^{\rho_{k1}};X)}\\
  &=\|L_t(U_n\mathbf 1_{[0,\sigma_{nl}\wedge\rho_{k1}]})\mathbf 1_{[0,\sigma_{nl}\wedge\rho_{k1}]} -
 L_t(V_n\mathbf 1_{[0,\sigma_{nl}\wedge\rho_{k1}]})
 \mathbf1_{[0,\sigma_{nl}\wedge\rho_{k1}]}\|_{V^p(0,\frac T{2^k},M^{\rho_{k1}};X)}\\
 &\leq C_{\frac T{2^k}, M^{\rho_{k1}}}\|U_n\mathbf 1_{[0,\sigma_{nl}\wedge\rho_{k1}]} -
 V_n\mathbf 1_{[0,\sigma_{nl}\wedge\rho_{k1}]}\|_{V^p(0,\frac T{2^k},M^{\rho_{k1}};X)}\\
  &\leq \frac 12\|U_n\mathbf 1_{[0,\sigma_{nl}\wedge\rho_{k1}]} -
 V_n\mathbf 1_{[0,\sigma_{nl}\wedge\rho_{k1}]}\|_{V^p(0,\frac T{2^k},M;X)},
\end{align*}
so a.s.\ $U_n\mathbf 1_{[0,\sigma_{nl}\wedge\rho_{k1}]}(s) = V_n\mathbf 1_{[0,\sigma_{nl}\wedge\rho_{k1}]}(s)$
for all $s \in (0,\frac T{2^k})$. Define again
\[
\rho_k := \rho_{k1}\wedge\ldots\wedge\rho_{k2^k},\;\;\; k\in \mathbb N.
\]
By the standard induction argument one derives that
a.s.\ $U_n\mathbf 1_{[0,\sigma_{nl}\wedge\rho_{k}]}\equiv V_n\mathbf 1_{[0,\sigma_{nl}\wedge\rho_{k}]}$ on $[0,T]$. Now taking $k$ and $l$ to infinity gives us the desired.

Since $U = \lim_{n\to \infty} U_n$ and $V = \lim_{n\to \infty} V_n$, then $U=V$ a.s.\ and uniqueness
is proved.
\end{proof}

\begin{proof}[Proof of Proposition \ref{prop:SEEwithoutA}]
This result follows with the same method as for Theorem \ref{exandun3}. Note that property $(\alpha)$ can be avoided since $A=0$ and hence we can take $S(t) = \tilde{S}(t) = I$ and the $\gamma$-boundedness is clear in this case.
\end{proof}

\begin{remark}
Using the time change result of Theorem \ref{rem:intfortimechange} one can turn the noise part of the problem \eqref{stoceq} into a cylindrical Brownian motion. Unfortunately, by using this technique the term $Au(t)\ud t$ becomes more involved. In particular, one has to use evolution families instead of semigroups, which complicates matters.
\end{remark}

\appendix

\section{A technical lemma on measurable selections}

In the next lemma we show that a certain projection valued function can be chosen in a measurable way. Moreover, we give a representation formula for its inverse which is used in the proof of Theorem \ref{main}. In \cite[Lemma~8.9]{PZ} a similar measurability result was proved by applying a selector theorem by Kuratowski and Ryll-Nardzewski.

Recall from before that a function $F:S\to \calL(H)$ is called $H$-strongly measurable if for all
$h\in H$, $s\mapsto F(s) h$ is strongly measurable.
\begin{lemma}\label{technical}
Let $(S,\Sigma)$ be a measurable space and let $H$ be a separable Hilbert space. Let $H_0\subseteq H$ be a finite dimensional subspace. Let $F :S \to \mathcal L(H)$ be a function such that:
\begin{enumerate}
\item $F$ is $H$-strongly measurable;
\item for all $s\in S$ and $h\in H$, $F(s)^*  = F(s)$ and $\lb F(s)h, h \rb\geq 0$.
\end{enumerate}
For each $s\in S$, let $P(s)\in \calL(H)$ be the orthogonal projection onto $F(s) H_0$.
Then there exist $H$-strongly measurable functions $\tilde P, L :S \to \calL(H)$ such that
\begin{equation}\label{nice}
\tilde PF = FP  \ \ \text{and} \ \ L F = P,
\end{equation}
pointwise in $S$. Moreover, $\tilde P$ is a projection.
\end{lemma}
The operator $\tilde P$ will not be an orthogonal projection in general.
\begin{proof}
Let $P_0$ be the orthogonal projection onto $H_0$.
For each $s \in S$ define $\tilde{P}(s) \in \calL(H)$ as follows:
\[ \tilde{P}(s)P_0F(s)^2P_0 h = F(s)^2P_0 h, \ \ \text{for $h\in H$},\]
and set $\tilde{P}(s) = 0$ on
$\text{ker }P_0F(s)^2 P_0$. Notice, that there is no contradiction, since if $P_0F(s)^2P_0 h = 0$ for some $h\in H$ and $s \in S$, then
\[
0 =\langle P_0F(s)^2P_0 h,h \rangle = \|F(s) P_0h\|^2
\]
and hence $h\in \ker F(s)P_0 \subseteq \ker F(s)^2P_0$. Since $P_0F(s)^2P_0$ is a finite-rank self-adjoint operator for each $s\in S$, we have $H = \ker P_0F(s)^2P_0 \oplus \text{ran } P_0F(s)^2P_0$, and thus $P_0F(s)^2P_0$ is a bounded linear operator (see \cite[Theorem~6.2-G]{Tay}).

In the sequel we suppress the $s\in S$ from the formulas.
We claim that
\begin{enumerate}
\item[(i)] $\tilde{P}h = 0$ for each $h \in H_0^{\perp}$;
\item[(ii)] $\tilde{P}F^2h = F^2h$ for $h\in H_0$.
\item[(iii)] $\tilde{P} F = F P$
\end{enumerate}
Property (i) is clear from  $H_0^{\perp}\subseteq \ker  P_0F^2P_0$.
For (ii) note that for every $h\in H_0$, we can write $F^2 h = P_0 F^2 h + (1-P_0) F^2 h$. Since for all $g\in H_0$,
\[\lb (1-P_0)F^2 h, g\rb = \lb (1-P_0)F^2 h, P_0 g\rb = \lb P_0 (1-P_0)F^2 h, g\rb = 0,\]
we find that $(1-P_0)F^2 h\in H_0^{\perp}$. Thus by (i)  and the definition of $\tilde{P}$,
\[\tilde{P}F^2h = \tilde{P} P_0 F^2 h + \tilde{P} (1-P_0) F^2 h = F^2 h\]
and (ii) follows.
To prove (iii) let $g\in \text{ran } P$. Choosing $h \in H_0$ s.t.\ $g = F h$ we find
\[F P g = F g= F^2 h \stackrel{(ii)}{=} \tilde{P} F^2 h = \tilde{P} F g.\]
On the other hand, $F P$ vanishes on the space
\[\ker P_0F = (\text{ran }FP_0)^{\perp} = (\text{ran }P)^{\perp} = \ker P.\]
The same holds true for $\tilde P F$. Indeed, since $(1-P_0) F h \in \ker P_0F^2P_0$, it follows that for $h \in \ker P_0F$
\[
\tilde P F h = \tilde P (1-P_0) F h = 0
\]
and this gives (iii).

Next we claim that $\tilde{P}^2 = \tilde{P}$. Indeed, for each $h\in H$, $\tilde P h \in \text{ran } F^2P_0$ by the definition of $\tilde P$. Thus by (ii) $\tilde P F^2 P_0 = F^2 P_0$ and therefore $\tilde P^2 h= \tilde P h$.

To prove $H$-strong measurability fix
an orthonormal basis $(h_i)_{i=1}^k$  for $H_0$.  For each subset $\alpha \subseteq \{1,\ldots,k\}$ there exists a measurable $S_\alpha \subseteq S$
such that $(F h_i)_{i \in \alpha}$ is a basis of $\text{span}(F h_i)_{1\leq i\leq k}$
(because of the strong measurability of $F h$ for each $h \in H$ and using the Gramian matrix
technique). Notice that if $(F h_i)_{i \in \alpha}$ is a basis of $\text{span}(F h_i)_{1\leq i\leq k}$,
then $(F^2 h_i)_{i \in \alpha}$ is a basis of $\text{span}(F^2 h_i)_{1\leq i\leq k}$.
Indeed, let $g =\sum_{i\in \alpha} c_i F h_i$ be a combination
of $(F h_i)_{i \in \alpha}$ with some scalars $(c_i)_{i\in \alpha}$.
If $Fg = 0$, then $g \in\ker F = (\text{ran }F)^{\perp}$, so $g=0$.
Let $\alpha,\beta \subseteq \{1,\ldots,k\}$. We will say that $\alpha< \beta$ if
 $\sum_{i\in \alpha}2^i < \sum_{i\in \beta} 2^i$.
 If $\alpha < \beta$, one has to redefine
$S_\alpha:= S_\alpha \setminus S_{\beta}$. After the iterations of this procedure for all pairs
$\alpha,\beta\subseteq \{1,\ldots,k\}$ the sets $(S_\alpha)_{\alpha \subseteq \{1,\ldots,k\}}$ will be pairwise disjoint.

Now fix $\alpha \subseteq \{1,\ldots,k\}$. Let $(g_i)_{i\in \alpha}$ be obtained from $(P_0 F^2h_i)_{i\in \alpha}$
by the Gram--Schmidt process. These vectors are orthonormal and measurable
because $(\langle P_0F^2h_i, P_0F^2h_j\rangle)_{i,j \in \alpha}$ are measurable. Moreover,
the transformation matrix $C=(c_{ij})_{i,j \in \alpha}$ such that
\begin{align*}
 g_i = \sum_{j\in \alpha} c_{ij}P_0F^2h_j, \;\;\; i\in \alpha,
\end{align*}
has measurable elements.
So, $\tilde P g_i = \tilde P\sum_{j\in \alpha} c_{ij}P_0F^2h_j = \sum_{j\in \alpha} c_{ij}F^2h_j $.
This means that for each $h\in H$ the following hold true:
\[
\tilde P h = \sum_{i\in \alpha}\langle h,g_i\rangle \tilde P g_i =
\sum_{i\in \alpha}\langle h,g_i\rangle  \sum_{j\in \alpha} c_{ij}F^2h_j,
\]
which is obviously measurable.

Now define $L$ as an operator with values in $F(H_0) = P(H_0)$ as follows:
\begin{align*}
L (F^2 h) &= P F h,\;\;\; h \in H,\\
 L h &= 0,\;\;\; h \in \ker F.
\end{align*}
Then $L$ is well-defined since $\ker F = \ker F^2$. Also for each $1\leq i\leq k$ and $h \in H$
\begin{align*}
 |\langle L (F^2 h), Fh_i\rangle |
&= |\langle P F h, F h_i\rangle | =|\langle Fh, P Fh_i\rangle | =|\langle F^2h, h_i\rangle |\leq \|F^2h\|.
\end{align*}
Since the range of $L$ is finite dimensional and equal to $F H_0$, the operator $L$ is bounded.
Since $H = \overline{\text{ran }F}\oplus \ker F$ and $\ker F = \ker P$ we find $L F = P$.

As before one can show that $L$ is $H$-strongly measurable,
This time fixing $\alpha \subseteq \{1,\ldots,k\}$ one considering the orthogonal basis $(g_i)_{i\in \alpha}$ for
$\text{span}(Fh_i)_{i\in \alpha}$.
\end{proof}

\def\cprime{$'$}

\end{document}